\documentclass[12pt]{amsart}

\usepackage{amsmath}
\usepackage{amsfonts}
\usepackage{latexsym}
\usepackage{graphicx}
\usepackage{amssymb}
\usepackage{amsthm}
\usepackage[margin=2.5cm]{geometry}
\usepackage{url}
\usepackage{color}
\usepackage{enumerate}
\usepackage[shortlabels]{enumitem} 
\usepackage[small]{caption}
\usepackage{cite}       

\usepackage[T1]{fontenc}     
\usepackage{lmodern}         
\usepackage[utf8]{inputenc}  

\usepackage{stmaryrd}              
\usepackage{algorithm}             
\usepackage[noend]{algpseudocode}  
\usepackage{accents}               
\usepackage{hyperref}
\usepackage{tikz}
\usepackage{tikz-cd}               

\usepackage{todonotes}

\newtheorem{definition}{Definition}[section]
\newtheorem{lemma}[definition]{Lemma}
\newtheorem{proposition}[definition]{Proposition}

\newtheorem{theorem}[definition]{Theorem}
\newtheorem{remark}[definition]{Remark}

\newtheorem{openquestion}[definition]{Open Question}

\newcommand{\N}{\mathbb{N}}
\newcommand{\Z}{\mathbb{Z}}

\newcommand{\R}{\mathbb{R}}

\newcommand{\zero}{{\boldsymbol{0}}}
\newcommand{\UN}{{\boldsymbol{1}}}
\newcommand{\ba}{{\boldsymbol{a}}}
\newcommand{\be}{{\boldsymbol{e}}}
\newcommand{\bh}{{\boldsymbol{h}}}
\newcommand{\bk}{{\boldsymbol{k}}}
\newcommand{\bn}{{\boldsymbol{n}}}
\newcommand{\bp}{{\boldsymbol{p}}}
\newcommand{\bq}{{\boldsymbol{q}}}
\newcommand{\Bm}{{\boldsymbol{m}}} 
\newcommand{\bu}{{\boldsymbol{u}}}
\newcommand{\bv}{{\boldsymbol{v}}}
\newcommand{\bx}{{\boldsymbol{x}}}
\newcommand{\by}{{\boldsymbol{y}}}
\newcommand{\balpha}{{\boldsymbol{\alpha}}}
\newcommand{\bbeta}{{\boldsymbol{\beta}}}

\newcommand{\Acal}{\mathcal{A}} 
\newcommand{\Bcal}{\mathcal{B}}
\newcommand{\Ccal}{\mathcal{C}}
\newcommand{\Fcal}{\mathcal{F}}

\newcommand{\Ical}{\mathcal{I}}
\newcommand{\Jcal}{\mathcal{J}}

\newcommand{\Lcal}{\mathcal{L}}
\newcommand{\Pcal}{\mathcal{P}}

\newcommand{\Scal}{\mathcal{S}}
\newcommand{\Tcal}{\mathcal{T}}
\newcommand{\Ucal}{\mathcal{U}}
\newcommand{\Vcal}{\mathcal{V}}
\newcommand{\Wcal}{\mathcal{W}}
\newcommand{\Xcal}{\mathcal{X}} 
\newcommand{\Zcal}{\mathcal{Z}}

\newcommand{\factorialclosure}[1]{{\overline{#1}^{Fact}}}
\newcommand{\shiftclosure}[1]{{\overline{#1}^{\sigma}}}
\newcommand{\topologicalclosure}[1]{{\overline{#1}}}
\newcommand{\dist}{\operatorname{dist}}
\newcommand{\shape}{\textsc{shape}}
\newcommand{\Zrange}[1]{\llbracket0,#1\rrbracket}
\newcommand{\ZZrange}[2]{{\llbracket#1,#2\rrbracket}}
\newcommand{\height}{\textsc{height}}
\newcommand{\width}{\textsc{width}}
\newcommand{\RecurrentVertices}{\textsc{RecurrentVertices}}
\newcommand{\SFT}{\textsc{SFT}}
\newcommand{\interior}[1]{\accentset{\circ}{#1}}
\newcommand{\sctop}{\textsc{top}}
\newcommand{\scbottom}{\textsc{bottom}}
\newcommand{\scleft}{\textsc{left}}
\newcommand{\scright}{\textsc{right}}
\newcommand{\generictorus}{{\boldsymbol{T}}}
\newcommand{\torus}{\mathbb{T}}
\newcommand{\scConfig}{\textsc{Config}}
\newcommand{\sccode}{\textsc{Code}}
\newcommand{\scReturnWord}{\textsc{ReturnWord}}

\newcommand\tile[4]{
    \raisebox{-3mm}{
\begin{tikzpicture}[scale=0.9]
\draw (0, 0) -- (1, 0);
\draw (0, 0) -- (0, 1);
\draw (1, 1) -- (1, 0);
\draw (1, 1) -- (0, 1);
\node[rotate=0,font=\footnotesize] at (0.8, 0.5) {#1};
\node[rotate=0,font=\footnotesize] at (0.5, 0.8) {#2};
\node[rotate=0,font=\footnotesize] at (0.2, 0.5) {#3};
\node[rotate=0,font=\footnotesize] at (0.5, 0.2) {#4};
\end{tikzpicture}}}

\usepackage[skins,listings,most]{tcolorbox}
\newtcolorbox[auto counter,number within=section]{exo}[1][]{colback=green!5!white,
colframe=green!15!black,fonttitle=\bfseries,
coltitle=red!10!black,
colbacktitle=green!40!white,enhanced,
attach boxed title to top left={yshift=-2mm,xshift=2mm},
title=Exercise \thetcbcounter,{#1}}

\newtcolorbox{exosolution}[2][]{colback=yellow!5!white,
colframe=yellow!15!black,fonttitle=\bfseries,
coltitle=red!10!black,
colbacktitle=yellow!40!white,enhanced,
attach boxed title to top left={yshift=-2mm,xshift=2mm},
title=Solution to Exercise {#2},{#1},breakable}

\newtcolorbox{sagecommandlinetcb}[1][]{
    oversize,
    right=1mm,
    left=1mm,
    top=-2mm,
    bottom=-2mm,
    colback=red!5!white,
    colframe=red!75!black,
    {#1}}

\usepackage{sagetex}
\usepackage{accsupp}
\newcommand{\noncopynumber}[1]{
    \BeginAccSupp{method=escape,ActualText={}}
    #1
    \EndAccSupp{}
}
\lstdefinestyle{SageInput}{style=DefaultSageInput,basicstyle={\footnotesize\ttfamily}}

\definecolor{pblue}{rgb}{0.13,0.13,1}
\definecolor{pgreen}{rgb}{0,0.5,0}
\newtcblisting[]{sageverbatimtcb}{%
listing only,
    listing options={
        language=python,
        basicstyle=\footnotesize\fontfamily{AnonymousPro}\selectfont,
        keywordstyle=\bfseries\color{pblue},
        stringstyle=\bfseries\itshape\color{green!40!black},
        commentstyle=\bfseries\itshape\color{black!60},
        showspaces=false,
        showtabs=false,
        breaklines=true,
        showstringspaces=false,
        tabsize=1,
        emph={sage,as},
        numbers=none,
        emphstyle={\bfseries\color{pblue}}
    },
    oversize,
    right=1mm,
    left=1mm,
    top=-2mm,
    bottom=-2mm,
colback=red!5!white,
colframe=red!75!black}

\algrenewcommand\algorithmicrequire{\textbf{Precondition:}}
\algrenewcommand\algorithmicensure{\textbf{Postcondition:}}

\usepackage{xifthen}

\newcommand{\ifnv}[2]{\ifthenelse{\equal{#1}{}}{}{#2}}
\newcommand\defn[2][]{\ifnv{#1}{\index{#1|\textbf}}{\index{#2|\textbf}}\textbf{#2}}


\makeatletter
\@namedef{subjclassname@2020}{\textup{2020} Mathematics Subject Classification}
\makeatother


\thanks{
The author acknowledges financial support from 
the Agence Nationale de la Recherche through the projects 
CODYS (ANR-18-CE40-0007),
PARADIS (ANR-18-CE23-0007-01) and
IZES (ANR-22-CE40-0011).}

\begin{document}


\title[Three characterizations of a self-similar aperiodic 2-dimensional subshift]
      {Three characterizations of\\a self-similar aperiodic 2-dimensional subshift}
\author[S.~Labb\'e]{S\'ebastien Labb\'e}
\address[S.~Labb\'e]{Univ. Bordeaux, CNRS,  Bordeaux INP, LaBRI, UMR 5800, F-33400, Talence, France}
\email{sebastien.labbe@labri.fr}


\begin{abstract}
The goal of this chapter is to illustrate a generalization of the Fibonacci
word to the case of 2-dimensional configurations on $\mathbb{Z}^2$.  More
precisely, we consider a particular subshift of $\mathcal{A}^{\mathbb{Z}^2}$ on
the alphabet $\mathcal{A}=\{0,\dots,15\}$ for which we give three
characterizations: as the subshift $\mathcal{X}_\Phi$ generated by a
2-dimensional morphism $\Phi$ defined on $\mathcal{A}$; as the Wang shift
$\Omega_\mathcal{Z}$ defined by a set $\mathcal{Z}$ of 16 Wang tiles; as the
symbolic dynamical system $\mathcal{X}_{\mathcal{P}_\mathcal{Z},R_\mathcal{Z}}$
representing the orbits under some $\mathbb{Z}^2$-action $R_\mathcal{Z}$
defined by rotations on $\mathbb{T}^2$ and coded by some topological partition
$\mathcal{P}_\mathcal{Z}$ of $\mathbb{T}^2$ into 16 polygonal atoms.  We prove
their equality $\Omega_\mathcal{Z}
=\mathcal{X}_\Phi=\mathcal{X}_{\mathcal{P}_\mathcal{Z},R_\mathcal{Z}}$ by
showing that they are self-similar with respect to the substitution $\Phi$.

This chapter provides a transversal reading of results divided into four
different articles obtained through the study of the Jeandel-Rao Wang shift.
It gathers in one place the methods introduced to desubstitute Wang shifts and
to desubstitute codings of $\mathbb{Z}^2$-actions by focussing on a simple
2-dimensional self-similar subshift. SageMath code to find marker tiles and
compute the Rauzy induction of $\mathbb{Z}^2$-rotations is provided 
allowing to reproduce the computations. The chapter contains many exercises
whose solutions are provided at the end.
\end{abstract}

\maketitle

\setcounter{tocdepth}{1}
\tableofcontents



\newpage
\section{Introduction}\label{chap:Labbe:sec:Introduction}

The rule
$s:a\mapsto ab, b\mapsto a$
defines a morphism on the monoid $\{a,b\}^*$.
The successive application of this morphism on the letter $a$
defines longer and longer words covering the negative and non-negative
integers:
\[
\begin{array}{r|l}
    s^0(a)&s^0(a)\\
    s^1(a)&s^1(a)\\
    s^2(a)&s^2(a)\\
    s^3(a)&s^3(a)\\
    s^4(a)&s^4(a)\\
    s^5(a)&s^5(a)\\
\end{array}
\qquad
=
\qquad
\begin{array}{r|l}
    \underline{a}& 
    a\\
    \underline{ab}& 
    ab\\
    a\underline{ba}& 
    aba\\
    aba\underline{ab}& 
    abaab\\
    abaaba\underline{ba}& 
    abaababa\\
    abaababaaba\underline{ab}& 
    abaababaabaab\\
\end{array}
\]
The letters that change from line to line are underlined.
It is an interesting exercise to show that at the limit, we obtain
$\lim_{n\to\infty}s^{2n}(a)|s^{2n}(a)= \widetilde{F}ba|F$
and
$\lim_{n\to\infty}s^{2n+1}(a)|s^{2n+1}(a)= \widetilde{F}ab|F$
where $F$ is the well-known right-infinite Fibonacci word \cite{berstel_mots_1980}.
The rule $s$ can be seen as a substitution that we may apply on the biinfinite words
$x=\widetilde{F}ba|F$
and $y=\widetilde{F}ab|F$ and we observe that
$s(x)=y$ and $s(y)=x$.
Thus $s^2(x)=x$ and $s^2(y)=y$ and we say that
$x$ and $y$ are fixed points of $s^2$.
The set of finite words that appear in $x=s^2(x)$ defines a language
\[
    \Lcal_s = \{\varepsilon, a, b, aa, ab, ba, aab, aba, baa, bab, aaba, \dots\}\subset\{a,b\}^*
\]
and a subshift
\[
    \Xcal_s = \{u\in\{a,b\}^\Z\colon\text{the finite words that appear in }u\text{ are in }\Lcal_s\}.
\]
The subshift $\Xcal_s$ contains $x$, $y$, all shifts of $x$ and $y$, and much more.
Indeed, $\Xcal_s$ is a Sturmian shift which is an uncountable set.
The reader will find detailed information on Sturmian sequences in
\cite[Chapter 2]{MR1905123} and \cite[Chapter 6]{MR1970385}.
In particular, $\Xcal_s$ is aperiodic, that is, it is nonempty and none of the
sequences in $\Xcal_s$ is periodic.

It is known since the early work of Morse and Hedlund in \cite{MR0000745} 
and Coven and Hedlund in \cite{MR0322838}
that the 1-dimensional subshift $\Xcal_s$,
being a Sturmian subshift,
has many equivalent characterizations:
\begin{itemize}
    \item as the subshift generated by the 1-dimensional substitution $s$;
    \item as the subshift on $\{a,b\}$ having exactly $n+1$ factors of length $n$ and
        such that the ratio of the frequency of the two letters is $\varphi=\frac{1+\sqrt{5}}{2}$;
    \item as the symbolic representation of a rotation on the $1$-dimensional
        torus $\torus=\R/\Z$ through a partition into two intervals
        whose length ratio is the golden mean.
\end{itemize}

\begin{sagesilent}
from slabbe import Substitution2d
Phi = Substitution2d({0: [[14]], 1: [[13]], 2: [[12],[10]],
        3: [[11],[8]], 4: [[14],[7]], 5: [[13],[7]], 6: [[12],[7]],
        7: [[12,6]], 8: [[14,3]], 9: [[13,3]], 10: [[12,2]],
        11: [[12,6],[10,1]], 12: [[11,6],[8,1]], 13: [[15,5],[9,1]],
        14: [[11,4],[8,1]], 15: [[12,2],[7,0]]})
from slabbe import WangTileSet
L = ["DOJO", "DOHL", "JMDP", "DMDK", "HPJP", "HPHN", "HKDP",
     "BOIO", "ILEO", "ILCL", "ALIO", "EPIP", "IPIK",
     "IKBM", "IKAK", "CNIP"]
L = [tuple(tile) for tile in L]
Z = WangTileSet(L)
from slabbe.arXiv_1903_06137 import self_similar_19_atoms_partition
PU = self_similar_19_atoms_partition()
merge_dict = {0:0, 1:1, 2:2, 3:3, 4:4, 5:5, 6:6, 7:6, 8:7, 9:8, 10:9, 
              11:10, 12:11, 13:11, 14:12, 15:12, 16:13, 17:14, 18:15}
PZ = PU.merge_atoms(merge_dict)
load('axis_labels.sage')
axis0 = axis_HV()
PZtikz = PZ.tikz(extra_code=axis0, fontsize=r'\normalsize', scale=5)
\end{sagesilent}

The goal of this chapter is to illustrate a generalization of the above
to the case of 2-dimensional configurations on $\Z^2$.
More precisely, we consider a particular subshift of $\Zrange{15}^{\Z^2}$,
first considered in \cite{lepsova_thesis_2024},
for which we give three characterizations:
\begin{itemize}
    \item as the subshift $\Xcal_\Phi$ generated by the 2-dimensional morphism
        $\Phi$ defined on the alphabet $\Acal=\Zrange{15}$ by the rule 
\begin{equation} 
\label{eq:definition-of-Phi} 
\begin{array}{ll}
    \Phi:&\Zrange{15}\to\Zrange{15}^{*^2}\\[2mm]
    &\left\{\footnotesize\arraycolsep=1.8pt
         \sage{Phi._latex_(ncolumns=4, align='l')}
    \right.
\end{array}
\end{equation}
    \item as the Wang shift $\Omega_\Zcal$,
        that is, the set of valid configuration $\Z^2\to\Zrange{15}$
        describing valid tiling of the plane using the following set $\Zcal$ of
        16 Wang tiles:
\begin{center}
    \sageplot[][pdf]{Z.tikz(ncolumns=8)}
\end{center}
    \item as the symbolic dynamical system $\Xcal_{\Pcal_\Zcal,R_\Zcal}$
        representing the orbits under
        the $\Z^2$-action $R_\Zcal$ defined by
        rotations on 
        $\torus^2$
        and coded by the topological partition $\Pcal_\Zcal$ of
        $\torus^2$:
    \[
    \begin{array}{c}
        \Pcal_\Zcal=
    \end{array}
    \begin{array}{c} 
        \sageplot[][pdf]{PZtikz}
    \end{array}
    \quad
    \begin{array}{l}
        R_\Zcal^\bn(\bx)=\bx+\varphi^{-2}\bn
    \end{array}
    \] 
    where $\varphi=\frac{1+\sqrt{5}}{2}$ is the golden ratio.
\end{itemize}

The reader may observe that while increasing the dimension from 1 to 2, we
replaced the second characterization of the Fibonacci subshift $\Xcal_s$ based on the factor
complexity by the notion of Wang shift or more generally subshift of finite type (SFT).
It may seem counter-intuitive since the Fibonacci subshift is aperiodic and 
1-dimensional SFT always contain a periodic
configuration \cite{MR1369092}, but this is not a contradiction in higher dimension 
since there exist aperiodic 2-dimensional SFTs \cite{MR0216954}.

In this chapter, we show that $\Omega_\Zcal$ and $\Xcal_{\Pcal_\Zcal,R_\Zcal}$
are self-similar.
The tools used in the proofs are completely different in each case: based on the
notion of marker tiles in the former case and on Rauzy induction of
$\Z^2$-rotations in the latter.
It turns out that the $2$-dimensional morphism describing the self-similarities is
$\Phi$ in both cases.

\begin{theorem}\label{thm:main-theoremA}
    {\rm \cite{lepsova_thesis_2024}}
    The Wang shift $\Omega_\Zcal\subset\Zrange{15}^{\Z^2}$ is self-similar satisfying
    $\Omega_\Zcal=\shiftclosure{\Phi(\Omega_\Zcal)}$
    where $\Phi:\Zrange{15}\to\Zrange{15}^{*^2}$
    is defined in Equation~\eqref{eq:definition-of-Phi}.
\end{theorem}

Theorem~\ref{thm:main-theoremA} was first proved in \cite{lepsova_thesis_2024}.
The set $\Zcal$ of 16 tiles was introduced in \cite{lepsova_thesis_2024}
as a simplification of the set $\Ucal$ of 19 Wang tiles introduced in
\cite{MR3978536}. Lepšová proved that $\Omega_\Zcal$ is topologically conjugate
to $\Omega_\Ucal$. Therefore, the Wang shift $\Omega_\Zcal$ is also minimal,
 aperiodic and self-similar as the same was known for $\Omega_\Ucal$.
The proof of Theorem~\ref{thm:main-theoremA} provided here is constructive
and uses the tools developed in \cite{MR4226493}.

\begin{theorem}\label{thm:main-theoremB}
    The subshift $\Xcal_{\Pcal_\Zcal,R_\Zcal}\subset\Zrange{15}^{\Z^2}$ is self-similar satisfying 
    $\Xcal_{\Pcal_\Zcal,R_\Zcal}=\shiftclosure{\Phi(\Xcal_{\Pcal_\Zcal,R_\Zcal})}$
    where $\Phi:\Zrange{15}\to\Zrange{15}^{*^2}$
    is defined in Equation~\eqref{eq:definition-of-Phi}.
\end{theorem}

The proof of Theorem~\ref{thm:main-theoremB} provided here is also constructive
and uses the tools developed in \cite{MR4347332} to perform the Rauzy induction
of toral $\Z^2$-rotations coded by polygonal partitions.

The equality of the three subshifts follows from a criterion for the minimality
of self-similar subshifts stated in Lemma~\ref{lem:criterion-for-minimality}.

    \begin{theorem}\label{thm:equality-three-subshifts}
The three subshifts are equal:
$\Xcal_\Phi
=\Omega_\Zcal
=\Xcal_{\Pcal_\Zcal,R_\Zcal}$.
    \end{theorem}

The 2-dimensional subshift $\Omega_\Ucal$ was introduced in \cite{MR3978536} and
discovered during the study of the substitutive structure
\cite{MR4226493} of the Jeandel--Rao Wang shift
\cite{MR4210631}.
Its description as the coding of a toral $\Z^2$-action was presented in
\cite{MR4213162} and its substitutive structure was further developed
in \cite{MR4347332}.
This chapter provides a transversal reading of results divided in four different
articles about Jeandel--Rao tilings and gathers the methods introduced by focussing on the 
self-similar subshift hidden in the Jeandel--Rao Wang shift, which is more simple.
Thus we avoid the difficulty raised by the
Jeandel--Rao Wang shift itself which is not a minimal subshift, has a long preperiod
in its substitutive description and needs the definition of other tools
including the shear-conjugacy.

\textbf{Structure of the chapter}.
Section~\ref{chap:Labbe:sec:Preliminaries} gathers preliminary notions on
topological dynamical systems,
subshifts and shifts of finite type and $d$-dimensional languages.
In Section~\ref{chap:Labbe:sec:aperiodic-self-similar-subshifts}, we define a
$2$-dimensional self-similar subshift $\Xcal_\Phi$ from a $2$-dimensional substitution
$\Phi$ defined on 16-letter alphabet.
We show that $\Xcal_\Phi$ is aperiodic.
In Section~\ref{chap:Labbe:sec:wang-shifts}, we introduce a Wang shift
$\Omega_\Zcal$ defined from a set $\Zcal$ of 16 Wang tiles and we show using
the notion of marker tiles that it is self-similar and $\Omega_\Zcal=\Xcal_\Phi$.
In Section~\ref{chap:Labbe:sec:Z2-rotations}, we introduce a $2$-dimensional subshift
defined as the symbolic representation of a toral $\Z^2$-rotation using a
partition of $\R^2/\Z^2$ into 16 polygons. We show that it is also self-similar
and equal to $\Xcal_\Phi$. Around 40 exercises are included in the chapter. Their solutions
are gathered at the end of the chapter in Section~\ref{sec:exosolutions}.

Algorithms to find marker tiles and compute the Rauzy induction of
$\Z^2$-rotations are provided as well as the SageMath code to reproduce the
computations.

\textbf{Acknowledgments}.
The author is thankful to Jana Lepšová for her careful reading of a preliminary
version of this chapter and to an anonymous referee for their comments leading
to significant improvements to this chapter.
All computations made in this chapter were made with
the following versions of SageMath \cite{sagemathv10.4} and optional
package \texttt{slabbe} \cite{labbe_slabbe_0_7_6_2023}:
\begin{sagecommandlinetcb}
\begin{sagecommandline}
sage: version()
'SageMath version ..., Release Date: ...'
sage: import importlib.metadata
sage: importlib.metadata.version("slabbe")
'...'
\end{sagecommandline}
\end{sagecommandlinetcb}
All outputs within red boxes in this chapter are computed directly from SageMath using
the \texttt{sagetex} package. Please contact the author if you have trouble
reproducing any of the computations. It is possible
to doctest (check all outputs) using the command
``\texttt{sage -t chapter\_doctest.sage}'' on the file provided in the archive
(166 tests, 13.20 s).

\section{Preliminaries}\label{chap:Labbe:sec:Preliminaries}

\subsection{Topological dynamical systems}


Most of the notions introduced here can be found in \cite{MR648108}.
A \defn{dynamical system} is
a triple $(X,G,T)$, where $X$ is a topological space, $G$ is a topological
group and $T$ is a continuous function $G\times X\to X$ defining a left action
of $G$ on $X$:
if $x\in X$, $e$ is the identity element of $G$ and $g,h\in G$, then using
additive notation for the operation in $G$ we have $T(e,x)=x$
and $T(g+h,x)=T(g,T(h,x))$.
In other words, if one denotes the transformation $x\mapsto T(g,x)$
by $T^g$, then $T^{g+h}=T^g T^h$.
In this work, we consider the Abelian group $G=\Z\times\Z$.

If $Y\subset X$, let $\overline{Y}$ denote the topological closure of $Y$ and
let $T(Y):=\cup_{g\in G}T^g(Y)$ denote the $T$-closure of $Y$.
Alternatively, we also use the notation $\overline{Y}^T=T(Y)$ to denote the $T$-closure of $Y$.
A subset $Y\subset X$ is \defn{$T$-invariant} if $T(Y)=Y$.
A dynamical system $(X,G,T)$ is called \defn{minimal} if $X$ does
not contain any nonempty, proper, closed $T$-invariant subset.
The left action of $G$ on $X$ is \defn{free}
if $g=e$ whenever there exists $x\in X$ such that $T^g(x)=x$.


Let $(X,G,T)$ and $(Y,G,S)$ be two dynamical systems with
the same topological group $G$.
A \defn{homomorphism} $\theta:(X,G,T)\to(Y,G,S)$ is a continuous
function $\theta:X\to Y$ satisfying the commuting property
that $S^g\circ\theta=\theta\circ T^g$ for every $g\in G$.
A homomorphism $\theta:(X,G,T)\to(Y,G,S)$ is called an \defn{embedding}
if it is one-to-one, a \defn{factor map} if it is onto, and a \defn{topological
conjugacy} if it is both one-to-one and onto and its inverse map is continuous.
If $\theta:(X,G,T)\to(Y,G,S)$ is a factor map,
then $(Y,G,S)$ is called a \defn{factor} of $(X,G,T)$
and $(X,G,T)$ is called an \defn{extension} of $(Y,G,S)$.
Two dynamical systems are \defn{topologically conjugate} if there is a
topological conjugacy between them.



A \defn{measure-preserving dynamical system} is defined as a system
$(X,G,T,\mu,\Bcal)$, where $\mu$ is a probability measure defined on
the Borel $\sigma$-algebra $\Bcal$ of subsets of $X$,
and $T^g:X\to X$ is a measurable map
which preserves the measure $\mu$ for all $g\in G$, that is,
$\mu(T^g(B))=\mu(B)$ for all $B\in\Bcal$. The measure $\mu$ is said to be
\defn{$T$-invariant}.
In what follows, 
when it is clear from the context,
we omit the Borel $\sigma$-algebra $\Bcal$ of subsets of $X$ and write $(X,G,T,\mu)$ 
to denote a measure-preserving dynamical system.

The set of all $T$-invariant probability measures of a dynamical
system $(X,G,T)$ is denoted by $\mathcal{M}^T(X)$.
A $T$-invariant probability measure on $X$ is called \defn{ergodic} if for every set
$B\in\Bcal$ such that $T^{g}(B)=B$ for all $g\in G$, we have that $B$ has either
zero or full measure. A
dynamical system $(X,G,T)$ is \defn{uniquely ergodic}
if it has only one invariant probability measure, i.e., $|\mathcal{M}^T(X)|=1$.
One can prove that a uniquely ergodic dynamical system is ergodic.
A dynamical system $(X,G,T)$ is said \defn{strictly ergodic}
if it is uniquely ergodic and minimal.

Let $(X,G,T,\mu,\Bcal)$
and $(X',G,T',\mu',\Bcal')$ be two measure-preserving dynamical systems.
We say that the two systems are
\defn{isomorphic} 
if there exist measurable sets $X_0\subset X$ and $X'_0\subset X'$
of full measure (i.e., $\mu(X\setminus X_0)=0$
and $\mu'(X'\setminus X'_0)=0$) with
$T^g(X_0)\subset X_0$, $T'^g(X'_0)\subset X'_0$ for all $g\in G$
and there exists a map $\phi:X_0\to X'_0$, called an \defn{isomorphism},
that is one-to-one and onto and such that for all $A\in\Bcal'(X'_0)$,
\begin{itemize}
    \item $\phi^{-1}(A)\in\Bcal(X_0)$,
    \item $\mu(\phi^{-1}(A))=\mu'(A)$, and
    \item $\phi\circ T^g(x)=T'^g\circ\phi(x)$ for all $x\in X_0$ and $g\in G$.
\end{itemize}
The role of the set $X_0$ is to make precise the fact that the properties of
the isomorphism need to hold only on a set of full measure.

\subsection{Subshifts and shifts of finite type}

In this section, we introduce multidimensional subshifts,
a particular type of dynamical systems 
\cite[\S 13.10]{MR1369092},
\cite{MR1861953,MR2078846,MR3525488}.
Let $\Acal$ be a finite set, $d\geq 1$, and let $\Acal^{\Z^d}$ be the set of all maps
$x:\Z^d\to\Acal$, equipped with the compact product topology. 
An element $x\in\Acal^{\Z^d}$ is called \defn{configuration}
and we write it as $x=(x_\Bm)=(x_\Bm:\Bm\in\Z^d)$,
where $x_\Bm\in\Acal$ denotes the value of $x$ at $\Bm$. 
The topology on $\Acal^{\Z^d}$ is compatible with the metric defined for all
configurations $x,x'\in\Acal^{\Z^d}$ by $\dist(x,x')=2^{-\min\left\{\Vert\bn\Vert\,:\,
x_\bn\neq x'_\bn\right\}}$
where $\Vert\bn\Vert = |n_1| + \dots + |n_d|$.
The \defn{shift action} $\sigma:\bn\mapsto
\sigma^\bn$ of the additive group $\Z^d$ on $\Acal^{\Z^d}$ is defined by
\begin{equation}\label{eq:shift-action}
    (\sigma^\bn(x))_\Bm = x_{\Bm+\bn}
\end{equation}
for every $x=(x_\Bm)\in\Acal^{\Z^d}$ and $\bn\in\Z^d$. 
If $X\subset \Acal^{\Z^d}$,
let $\overline{X}$ denote the topological closure of $X$
and let $\shiftclosure{X}:=\{\sigma^\bn(x)\mid x\in X, \bn\in\Z^d\}$
denote the shift-closure of $X$.
A subset $X\subset
\Acal^{\Z^d}$ is \defn{shift-invariant} if 
$\shiftclosure{X}=X$. A closed, shift-invariant subset
$X\subset\Acal^{\Z^d}$ is a \defn{subshift}. 
If $X\subset\Acal^{\Z^d}$ is a subshift we write
$\sigma=\sigma^X$ for the restriction of the shift action
\eqref{eq:shift-action} to $X$. 
When $X$ is a subshift,
the triple $(X,\Z^d,\sigma)$ is a dynamical system
and the notions presented in the previous section hold.

A configuration $x\in X$ is \defn{periodic} if there is a nonzero vector
$\bn\in\Z^d\setminus\{\zero\}$ such that $x=\sigma^\bn(x)$
and otherwise it is \defn{nonperiodic}.
We say that a nonempty subshift $X$ is \defn{aperiodic}
if the shift action $\sigma$ on $X$ is free.

For any subset $S\subset\Z^d$ let $\pi_S:\Acal^{\Z^d}\to\Acal^S$ denote the
projection map which restricts every $x\in\Acal^{\Z^d}$ to $S$. 
A \defn{pattern} is a function $p\in\Acal^S$ for some finite subset
$S\subset\Z^d$.
To every pattern $p\in\Acal^S$ corresponds
a subset $\pi_S^{-1}(p)\subset\Acal^{\Z^d}$ called \defn{cylinder}.
A nonempty set $X\subset\Acal^{\Z^d}$ is a
\defn{subshift} if and only if there exists a set $\Fcal$
of \defn{forbidden} patterns such that
$X=X_\Fcal$ where
\begin{equation}\label{eq:SFT}
    X_\Fcal = \{x\in\Acal^{\Z^d} \mid \pi_S\circ\sigma^\bn(x)\notin\Fcal
    \text{ for every } \bn\in\Z^d \text{ and } S\subset\Z^d\},
\end{equation}
see \cite[Prop.~9.2.4]{MR3525488}.
A subshift $X\subset\Acal^{\Z^d}$ is a 
\defn{subshift of finite type} (SFT) if there exists a finite set $\Fcal$ such that $X=X_\Fcal$
A subshift $X\subset\Acal^{\Z^d}$ is \defn{effective}
if there exists a \defn{computably enumerable} family of forbidden patterns $\Fcal$
such that $X=X_\Fcal$.

In this chapter, we mostly consider subshifts of finite type on $\Z\times\Z$,
that is, the case $d=2$.

\subsection{$d$-dimensional word}

In this section, we recall the definition of $d$-dimensional word that appeared
in \cite{MR2579856} and we keep the notation $u\odot^i v$ they proposed
for the concatenation. 

We denote by
$\{\be_k|1\leq k\leq d\}$ the canonical
basis of $\Z^d$ where $d\geq1$ is an integer.
If $i\leq j$ are integers, then $\llbracket i, j\rrbracket$ denotes the
interval of integers $\{i, i+1, \dots, j\}$.
Let $\bn=(n_1,\dots,n_d)\in\N^d$ and $\Acal$ be an alphabet.
We denote by $\Acal^{\bn}$ the set of functions
\begin{equation*}
    u:
\llbracket 0,n_1-1\rrbracket
\times
\cdots
\times
\llbracket 0,n_d-1\rrbracket
\to\Acal.
\end{equation*}
An element $u\in\Acal^\bn$ is called a
\defn{$d$-dimensional word} of \defn{shape} $\bn=(n_1,\dots,n_d)\in\N^d$
on the alphabet~$\Acal$.
We use the notation $\shape(u)=\bn$ when necessary.
The set of all finite $d$-dimensional words is 
$\Acal^{*^d}=\{\Acal^\bn\mid\bn\in\N^d\}$.
A $d$-dimensional word of shape $\be_k+\sum_{i=1}^d\be_i$ is called a
\defn{domino in the direction $\be_k$}.
When the context is clear, we write $\Acal$ instead of $\Acal^{(1,\dots,1)}$.
When $d=2$, we represent a $d$-dimensional word $u$ of shape $(n_1,n_2)$ as a
matrix with Cartesian coordinates:
\begin{equation*}
    u=
    \left(\begin{array}{ccc}
        u_{0,n_2-1} &\dots   & u_{n_1-1,n_2-1} \\
        \dots   &\dots   & \dots \\
        u_{0,0} &\dots   & u_{n_1-1,0}
    \end{array}\right).
\end{equation*}
Let $\bn,\Bm\in\N^d$ and $u\in\Acal^\bn$ and $v\in\Acal^\Bm$.
If there exists an index $i$ such that 
$n_j=m_j$ for all $j\in\{1,\dots,d\}\setminus\{i\}$,
then the \defn{concatenation} of $u$ and $v$ in the direction $\be_i$ 
is defined: it is
the 
$d$-dimensional word $u\odot^i v$ of shape $(n_1,\dots,n_{i-1},n_i+m_i,n_{i+1},\dots,n_d)\in\N^d$
given as
\begin{equation*}
    (u\odot^i v) (\ba) = 
\begin{cases}
    u(\ba)          & \text{if}\quad 0 \leq a_i < n_i,\\
    v(\ba-n_i\be_i) & \text{if}\quad n_i \leq a_i < n_i+m_i.
\end{cases}
\end{equation*}
The following equation illustrates the concatenation of $2$-dimensional words 
in the direction $\be_2$:
\[
    \arraycolsep=2.5pt
\left(\begin{array}{ccccc}
4 & 5 \\
10 & 5
\end{array}\right)
\odot^2
\left(\begin{array}{ccccc}
3 & 10 \\
9 & 9 \\
0 & 0 \\
\end{array}\right)
    =
\left(\begin{array}{ccccc}
3 & 10 \\
9 & 9 \\
0 & 0 \\
4 & 5 \\
10 & 5
\end{array}\right)
\]
whereas
\[
    \arraycolsep=2.5pt
\left(\begin{array}{ccccc}
3 & 10 \\
9 & 9 \\
0 & 0 \\
\end{array}\right)
\odot^2
\left(\begin{array}{ccccc}
4 & 5 \\
10 & 5
\end{array}\right)
    =
\left(\begin{array}{ccccc}
4 & 5 \\
10 & 5 \\
3 & 10 \\
9 & 9 \\
0 & 0
\end{array}\right)
\]
and in the direction $\be_1$:
\[
    \arraycolsep=2.5pt
\left(\begin{array}{ccccc}
2 & 8 & 7  \\
7 & 3 & 9 \\
1 & 1 & 0 \\
6 & 6 & 7 \\
7 & 4 & 3 
\end{array}\right)
\odot^1
\left(\begin{array}{ccccc}
3 & 10 \\
9 & 9 \\
0 & 0 \\
4 & 5 \\
10 & 5
\end{array}\right)
    =
\left(\begin{array}{ccccc}
2 & 8 & 7 & 3 & 10 \\
7 & 3 & 9 & 9 & 9 \\
1 & 1 & 0 & 0 & 0 \\
6 & 6 & 7 & 4 & 5 \\
7 & 4 & 3 & 10 & 5
\end{array}\right).
\]

Let $\bn,\Bm\in\N^d$ and $u\in\Acal^\bn$ and $v\in\Acal^\Bm$.
We say that $u$ \defn{occurs in $v$ at position} $\bp\in\N^d$ if
$v$ is large enough, i.e., $\Bm-\bp-\bn\in\N^d$ and
\[
    v(\ba+\bp) = u(\ba)
\]
for all $\ba=(a_1,\dots,a_d)\in\N^d$ such that 
$0\leq a_i<n_i$ with $1\leq i\leq d$.
If $u$ occurs in $v$ at some position, then we say that $u$ is a
$d$-dimensional \defn{subword} or \defn{factor} of $v$.

\subsection{$d$-dimensional rectangular language}

A subset $L\subseteq\Acal^{*^d}$ is called a $d$-dimensional \defn{language}. The
\defn{factorial closure} of a language $L$ is
\begin{equation*}
    \factorialclosure{L}
    = \{u\in\Acal^{*^d} \mid u\text{ is a $d$-dimensional subword of some } 
                           v\in L\}.
\end{equation*}
A language $L$ is \defn{factorial} if $\factorialclosure{L}=L$.
All languages considered in this contribution are factorial.
Given a configuration $x\in\Acal^{\Z^d}$, the \defn{language} $\Lcal(x)$ defined by $x$ is
\begin{equation*}
    \Lcal(x) = \{u\in\Acal^{*^d} \mid u\text{ is a $d$-dimensional subword of } x\}.
\end{equation*}
The \defn{language} of a subshift $X\subseteq\Acal^{\Z^d}$ is
    $\Lcal_X = \cup_{x\in X} \Lcal(x)$.
Conversely, given a factorial language $L\subseteq\Acal^{*^d}$ we define the subshift
\begin{equation*}
    \Xcal_L = \{x\in\Acal^{\Z^d}\mid \Lcal(x)\subseteq L\}.
\end{equation*}
A $d$-dimensional subword $u\in\Acal^{*^d}$ is \defn{allowed} in a 
subshift $X\subset\Acal^{\Z^d}$ if $u\in\Lcal_X$
and it is \defn{forbidden} in $X$ if $u\notin\Lcal_X$.
A language $L\subseteq\Acal^{*^d}$ is \defn{forbidden} in a 
subshift
$X\subset\Acal^{\Z^d}$ if $L\cap\Lcal_X=\varnothing$.

\begin{exo}[label={exo:subshift-easy}]
    Let $x:\Z^d\to\{0,1\}$ be the configuration
    defined by 
    \[
        x(\bn) = \Vert\bn\Vert_1\bmod 2
        \quad
        \text{ for every }
        \bn\in\Z^d.
    \]
    Describe $\Lcal(x)$. 
    How many elements are in the subshift $\Xcal_{\Lcal(x)}$?
\end{exo}

\subsection{$d$-dimensional morphisms}

In this section, we generalize the definition of $d$-dimensional morphisms
\cite{MR2579856} to the case where the domain and codomain are different as in
the case of $S$-adic systems \cite{MR3330561}.

Let $\Acal$ and $\Bcal$ be two alphabets.
Let $L\subseteq\Acal^{*^d}$ be a factorial language.
A function $\omega:L\to\Bcal^{*^d}$ is a \defn{$d$-dimensional
morphism} if for every
$i$ with $1\leq i\leq d$,
and every $u,v\in L$ such that 
$u\odot^i v$ is defined
and
$u\odot^i v\in L$,
we have
that the concatenation $\omega(u)\odot^i \omega(v)$
in direction $\be_i$ is defined and
\begin{equation*}
    \omega(u\odot^i v) = \omega(u)\odot^i \omega(v).
\end{equation*}
Note that the left-hand side of the equation is defined since
$u\odot^i v$ belongs to the domain of $\omega$.
A $d$-dimensional morphism $L\to\Bcal^{*^d}$ is thus completely defined from the
image of the letters in $\Acal$, so we sometimes denote
a $d$-dimensional morphism as a rule $\Acal\to\Bcal^{*^d}$ 
when the language $L$ is unspecified.

The next lemma can be deduced from the definition.
It says that when $d\geq 2$ every $d$-dimensional morphism
defined on the whole set $L=\Acal^{*^d}$ is uniform.
We say that a $d$-dimensional morphism $\omega:L\to\Bcal^{*^d}$ is \defn{uniform}
if there exists a shape $\bn\in\N^d$ such that $\omega(a)\in\Bcal^\bn$ for
every letter $a\in\Acal$.
These are called block-substitutions in \cite{frank_introduction_2018}.

\begin{lemma}
Let $\omega:L\to\Bcal^{*^d}$ be a $d$-dimensional morphism.
If $d\geq 2$ and $L=\Acal^{*^d}$, then $\omega$ is uniform.
\end{lemma}

Therefore, to consider non-uniform $d$-dimensional morphisms when $d\geq 2$, we
need to restrict the domain to a strict subset $L\subsetneq\Acal^{*^d}$.
In \cite{MR2579856} and \cite[p.144]{MR1014984}, they consider the case $\Acal=\Bcal$
and they restrict the domain of $d$-dimensional morphisms to the language they
generate.

Given a language $L\subseteq\Acal^{*^d}$ of $d$-dimensional words and 
a $d$-dimensional morphism $\omega:L\to\Bcal^{*^d}$, we define the image of the
language $L$ under $\omega$ as the language
\begin{equation*}
\factorialclosure{\omega(L)}
    = \{u\in\Bcal^{*^d} \mid u\text{ is a $d$-dimensional subword of }
                  \omega(v) \text{ with } v\in L\}
    \subseteq \Bcal^{*^d}.
\end{equation*}
Observe that some elements of 
$\factorialclosure{\omega(L)}$ do not have a preimage under $\omega$.

Let $L\subseteq\Acal^{*^d}$ be a factorial language
and $\Xcal_L\subseteq\Acal^{\Z^d}$ be the subshift generated by $L$.
A $d$-dimensional morphism 
$\omega:L \to\Bcal^{*^d}$ 
can be extended to a continuous map 
$\omega:\Xcal_L\to\Bcal^{\Z^d}$
in such a way that the origin of $\omega(x)$ is at zero position
in the word $\omega(x_\zero)$
for all $x\in\Xcal_L$. More precisely, the image
under $\omega$ of the configuration $x\in\Xcal_L$ is
\[
    \omega(x) =
    \lim_{n\to\infty}\sigma^{f(n)}\omega\left(\sigma^{-n\UN}(x|_{\llbracket-n\UN,n\UN\llbracket})\right)
    \in\Bcal^{\Z^d}
\]
where $\UN=(1,\dots,1)\in\Z^d$,
$f(n)=\shape\left(\omega(\sigma^{-n\UN}(x|_{\llbracket-n\UN,\zero\llbracket}))\right)$
for all $n\in\N$ and
$\llbracket\Bm,\bn\llbracket
=
\ZZrange{m_1}{n_1-1}
\times \cdots \times
\ZZrange{m_d}{n_d-1}$.

In general, the closure under the shift of the image of a subshift
$X\subseteq\Acal^{\Z^d}$ under $\omega$
is the subshift
\begin{equation*}
\shiftclosure{\omega(X)}
    = \{\sigma^\bk\omega(x)\in\Bcal^{\Z^d} \mid \bk\in\Z^d, x\in X\}
    \subseteq \Bcal^{\Z^d}.
\end{equation*}


Now we show that $d$-dimensional morphisms preserve minimality of subshifts.

\begin{lemma}\label{lem:minimal-implies-minimal}
    Let $\omega:X\to\Bcal^{\Z^d}$ be a $d$-dimensional morphism for some 
    $X\subseteq\Acal^{\Z^d}$.
    If $X$ is a minimal subshift, then $\shiftclosure{\omega(X)}$ is
    a minimal subshift.
\end{lemma}

\begin{proof}
    Let $\varnothing\neq Z\subseteq\shiftclosure{\omega(X)}$ be a closed
    shift-invariant subset.
    We want to show that $\shiftclosure{\omega(X)}\subseteq Z$.
    Let $u\in\shiftclosure{\omega(X)}$.
    Thus $u=\sigma^\bk\omega(x)$ for some $\bk\in\Z^d$ and $x\in X$.
    Since $Z\neq\varnothing$, there exists $z\in Z$.
    Thus $z=\sigma^{\bk'}\omega(x')$ for some $\bk'\in\Z^d$ and $x'\in X$.
    Since $X$ is minimal, there exists a sequence $(\bk_n)_{n\in\N}$,
    $\bk_n\in\Z^d$, such that $x=\lim_{n\to\infty}\sigma^{\bk_n} x'$.
    For some other sequence $(\bh_n)_{n\in\N}$, $\bh_n\in\Z^d$, we have
    \begin{equation*}
        u 
        = \sigma^\bk\omega(x) 
        = \sigma^\bk\omega\left(\lim_{n\to\infty}\sigma^{\bk_n} x'\right) 
        = \sigma^\bk\lim_{n\to\infty}\sigma^{\bh_n} \omega\left(x'\right)
        = \lim_{n\to\infty}\sigma^{\bk+\bh_n-\bk'} z.
    \end{equation*}
    Since $Z$ is closed and shift-invariant, it follows that $u\in Z$.
\end{proof}

\begin{exo}[label={exo:easy-image-by-Phi}]
We consider the $2$-dimensional morphism $\Phi$ 
    defined in Equation~\eqref{eq:definition-of-Phi}
    on the alphabet $\Acal=\Zrange{15}$.
Compute $\Phi(u)$ for each 2-dimensional word $u$ below
    \[
    \left(\begin{array}{c}
    11 \\
    \end{array}\right)
    ,\quad
    \left(13, \, 7\right)
    ,\quad
    \left(\begin{array}{c}
    6 \\
    10 \\
    \end{array}\right)
    ,\quad
    \left(\begin{array}{cc}
    6 & 1 \\
    11 & 8
    \end{array}\right)
    \]
except one for which the image $\Phi(u)$ is not well-defined.
\end{exo}

\begin{exo}[label={exo:image-of-subshift-is-a-subshift}]
    Let $\omega:X\to\Bcal^{\Z^d}$ be a $d$-dimensional morphism for some 
    $X\subseteq\Acal^{\Z^d}$.
    Prove that if $X$ is a subshift, then
$\shiftclosure{\omega(X)}$ is a subshift.
\end{exo}

\section{An aperiodic self-similar subshift}\label{chap:Labbe:sec:aperiodic-self-similar-subshifts}

\subsection{Self-similar subshifts}\label{chap:Labbe:subsec:self-similar-subshifts}

In this section, we consider languages and subshifts defined from morphisms
leading to self-similar structures. 
In this situation, the domain and codomain
of morphisms are defined over the same alphabet. 
Formally, we consider the case of $d$-dimensional morphisms
$\Acal\to\Bcal^{*^d}$ where $\Acal=\Bcal$.

The definition of self-similarity depends on the notion of expansiveness.
It avoids the presence of lower-dimensional self-similar structure by having
expansion in all directions.

\begin{definition}
We say that a $d$-dimensional morphism $\omega:\Acal\to\Acal^{*^d}$ is
\defn{expansive}
if for every $a\in\Acal$ and $K\in\N$,
there exists $m\in\N$ such that 
    \[
        \min(\shape(\omega^m(a)))>K.
    \]
\end{definition}

\begin{definition}
A subshift $X\subseteq\Acal^{\Z^d}$ 
is \defn{self-similar}
if there exists an expansive
$d$-dimensional morphism $\omega:\Acal\to\Acal^{*^d}$ such that
$X=\shiftclosure{\omega(X)}$.
\end{definition}

Respectively, 
a language $L\subseteq\Acal^{*^d}$
is \defn{self-similar}
if there exists an expansive
$d$-dimensional morphism $\omega:\Acal\to\Acal^{*^d}$ such that
$L=\factorialclosure{\omega(L)}$.

Self-similar languages and subshifts can be constructed by iterative
application of a morphism $\omega$ starting with the letters.
The \defn{language} $\Lcal_\omega$ defined by an expansive $d$-dimensional
morphism $\omega:\Acal\to\Acal^{*^d}$ is
\begin{equation*}
    \Lcal_\omega = \{u\in\Acal^{*^d} \mid u\text{ is a $d$-dimensional subword of }
    \omega^n(a) \text{ for some } a\in\Acal\text{ and } n\in\N \}.
\end{equation*}
It satisfies
$\Lcal_\omega=\factorialclosure{\omega(\Lcal_\omega)}$
and thus is self-similar.
The \defn{substitutive shift} $\Xcal_\omega=\Xcal_{\Lcal_\omega}$
defined from the language of $\omega$ is a self-similar subshift
since $\Xcal_\omega=\shiftclosure{\omega(\Xcal_\omega)}$ holds.

\begin{figure}
\begin{center}
    \includegraphics{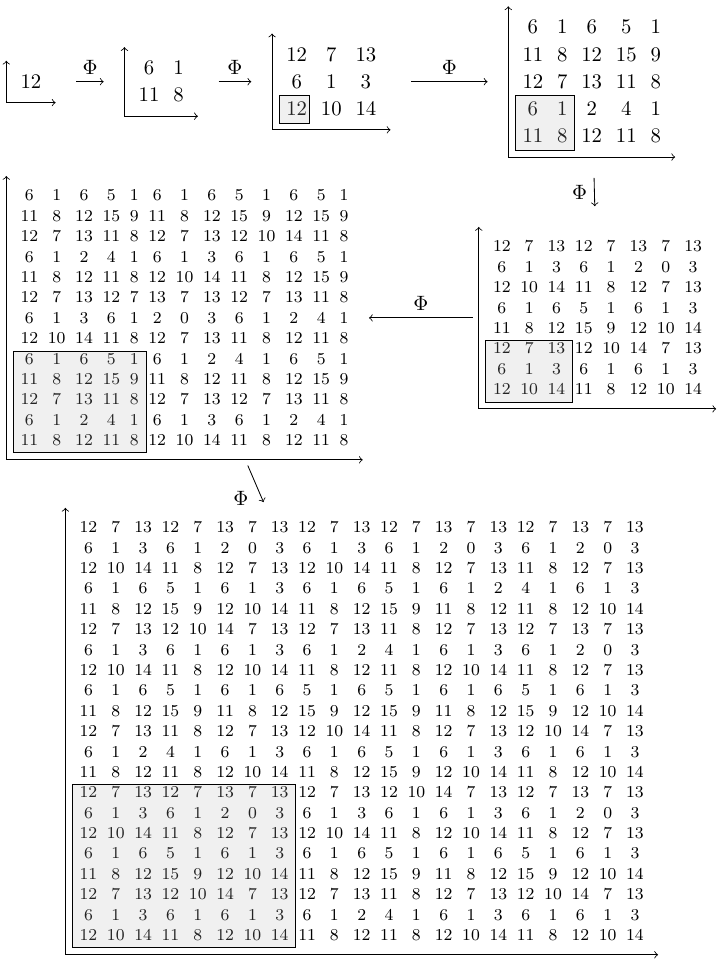}
\end{center}
\caption{
Building a configuration of the positive quadrant with $\Phi$.
We compute $(\Phi^n(12))_{n\in\N}$ for the first values of $n\in\{0,1,2,3,4,5,6\}$.
The gray rectangles surround
patterns seen two step before in the application of $\Phi$.
The limit $\lim_{n\to\infty}\Phi^{2n}(12)$ defines a configuration
of the positive quadrant $\N^2$ and similarly for
the limit $\lim_{n\to\infty}\Phi^{2n+1}(12)$.}
\label{fig:Phi-on-seed-12}
\end{figure}

\begin{figure}
\begin{center}
    \includegraphics{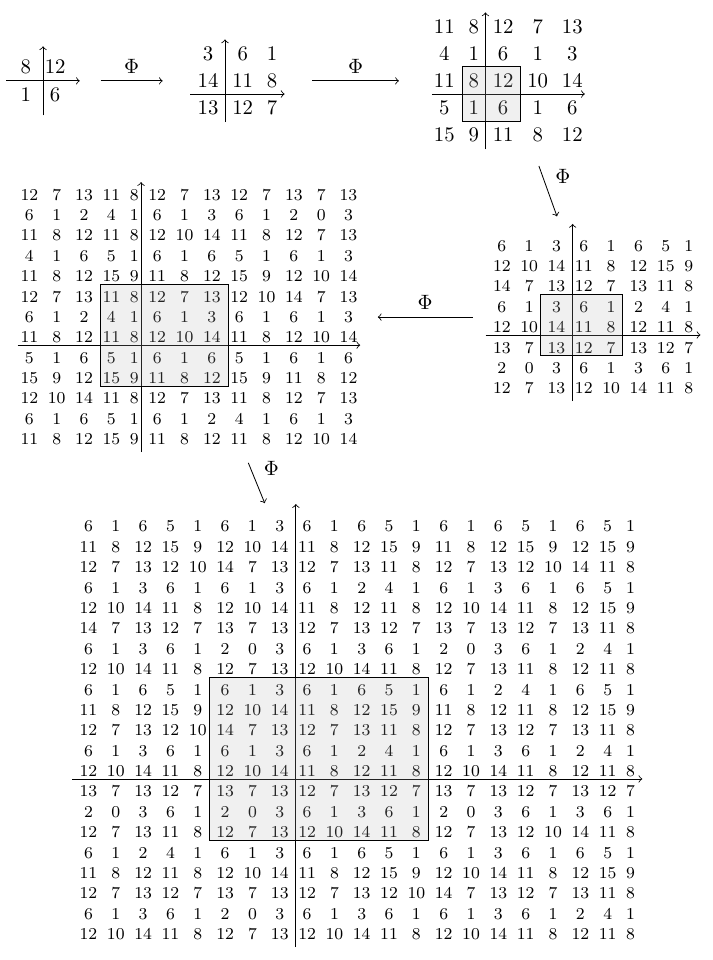}
\end{center}
\newsavebox{\smlmat}
\savebox{\smlmat}{$\left(\begin{smallmatrix}8&12\\1&6\end{smallmatrix}\right)$}
\newsavebox{\smlmatII}
\savebox{\smlmatII}{$\left(\begin{smallmatrix}17&13\\16&15\end{smallmatrix}\right)$}
\caption{Building a configuration of $\Z^2$ with $\Phi$.
We compute $(\Phi^n\usebox{\smlmat})_{n\in\N}$ for the first values of $n\in\{0,1,2,3,4,5\}$.
The gray rectangles surround
patterns seen two step before in the application of $\Phi$.
The limit $\lim_{n\to\infty}\Phi^{2n}\usebox{\smlmat}$ defines a configuration
of $\Z^2$ and similarly for
the limit $\lim_{n\to\infty}\Phi^{2n+1}\usebox{\smlmat}$.}
\label{fig:Phi-on-seed-8-12-1-6}
\end{figure}

Let $\Phi:\Zrange{15}\to\Zrange{15}^{*^2}$ 
be the $2$-dimensional morphism defined in Exercise~\ref{exo:easy-image-by-Phi}.
At Figure~\ref{fig:Phi-on-seed-12}, we compute the sequence of $2$-dimensional words
$(\Phi^n(12))_{n\in\N}$ for the first values of $n$. 
Since 12 appears in the lower left corner of $\Phi^2(12)$, 
then the rectangular pattern $\Phi^{n}(12)$ appears in the lower left corner of $\Phi^{n+2}(12)$
for every integer $n\geq0$.
Thus, the limit
$\lim_{n\to\infty}\Phi^{2n}(12)$ is well-defined and it defines a
configuration of the positive quadrant~$\N^2$.

This procedure can be done in each of the four quadrants.
At Figure~\ref{fig:Phi-on-seed-8-12-1-6}, we compute the sequence of $2$-dimensional words
$(\Phi^n\left(\begin{smallmatrix}8&12\\1&6\end{smallmatrix}\right))_{n\in\N}$
for the first values of~$n$. 
The limits
\begin{align*}
x &= \lim_{n\to\infty}\Phi^{2n}\left(\begin{smallmatrix}8&12\\1&6\end{smallmatrix}\right)\\
y &= \lim_{n\to\infty}\Phi^{2n+1}\left(\begin{smallmatrix}8&12\\1&6\end{smallmatrix}\right)
\end{align*}
are well-defined and define two configurations of $\Z^2$.
They satisfy $\Phi(x)=y$ and $\Phi(y)=x$.
This implies that the configurations $x$ and $y$ are fixed points of $\Phi^2$
since $\Phi^2(x)=x$ and $\Phi^2(y)=y$.

%

\begin{exo}[label={exo:prove-expansive}]
    Prove that $\Phi$ defined in 
    Equation~\eqref{eq:definition-of-Phi} 
    is expansive.
\end{exo}

\begin{exo}[label={exo:substitutive-shift-nonempty}]
    Prove that $\Xcal_\Phi\neq\varnothing$.
\end{exo}

\clearpage 

\begin{exo}[label={exo:Phi-sagemath}]
    Using the \texttt{slabbe} package of the SageMath open-source software,
    define the 2-dimensional substitution $\Phi$ as follows.
    Note that $2$-dimensional words are encoded using Cartesian-like
    coordinates instead of matrix-like coordinates.
\begin{sagecommandlinetcb}
\begin{sagecommandline}
sage: from slabbe import Substitution2d
sage: Phi = Substitution2d({0: [[14]], 1: [[13]], 2: [[12],[10]],
....: 3: [[11],[8]], 4: [[14],[7]], 5: [[13],[7]], 6: [[12],[7]],
....: 7: [[12,6]], 8: [[14,3]], 9: [[13,3]], 10: [[12,2]],
....: 11: [[12,6],[10,1]], 12: [[11,6],[8,1]], 13: [[15,5],[9,1]],
....: 14: [[11,4],[8,1]], 15: [[12,2],[7,0]]})
\end{sagecommandline}
\end{sagecommandlinetcb}
    Reproduce the computations of the 
    Figure~\ref{fig:Phi-on-seed-12} and
    Figure~\ref{fig:Phi-on-seed-8-12-1-6}
    in SageMath.
\end{exo}

\begin{exo}[label={exo:HV-dominoes}]
    The language of horizontal and vertical dominoes that we see in 
    Figure~\ref{fig:Phi-on-seed-12} 
    and
    Figure~\ref{fig:Phi-on-seed-8-12-1-6}
    obtained from the morphism $\Phi$ are
    \[
        H = \left\{
            \begin{array}{l}
\left(\begin{smallmatrix}0  & 3 \end{smallmatrix}\right),
\left(\begin{smallmatrix}1  & 2 \end{smallmatrix}\right),
\left(\begin{smallmatrix}1  & 3 \end{smallmatrix}\right),
\left(\begin{smallmatrix}1  & 6 \end{smallmatrix}\right),
\left(\begin{smallmatrix}2  & 0 \end{smallmatrix}\right),
\left(\begin{smallmatrix}2  & 4 \end{smallmatrix}\right),
\left(\begin{smallmatrix}3  & 6 \end{smallmatrix}\right),
\left(\begin{smallmatrix}4  & 1 \end{smallmatrix}\right),
\left(\begin{smallmatrix}5  & 1 \end{smallmatrix}\right),\\
\left(\begin{smallmatrix}6  & 1 \end{smallmatrix}\right),
\left(\begin{smallmatrix}6  & 5 \end{smallmatrix}\right),
\left(\begin{smallmatrix}7  & 13\end{smallmatrix}\right),
\left(\begin{smallmatrix}8  & 12\end{smallmatrix}\right),
\left(\begin{smallmatrix}9  & 11\end{smallmatrix}\right),
\left(\begin{smallmatrix}9  & 12\end{smallmatrix}\right),
\left(\begin{smallmatrix}10 & 14\end{smallmatrix}\right),
\left(\begin{smallmatrix}11 & 8 \end{smallmatrix}\right),\\
\left(\begin{smallmatrix}12 & 7 \end{smallmatrix}\right),
\left(\begin{smallmatrix}12 & 10\end{smallmatrix}\right),
\left(\begin{smallmatrix}12 & 11\end{smallmatrix}\right),
\left(\begin{smallmatrix}12 & 15\end{smallmatrix}\right),
\left(\begin{smallmatrix}13 & 7 \end{smallmatrix}\right),
\left(\begin{smallmatrix}13 & 11\end{smallmatrix}\right),
\left(\begin{smallmatrix}13 & 12\end{smallmatrix}\right),\\
\left(\begin{smallmatrix}14 & 7 \end{smallmatrix}\right),
\left(\begin{smallmatrix}14 & 11\end{smallmatrix}\right),
\left(\begin{smallmatrix}15 & 9 \end{smallmatrix}\right)
            \end{array}\right\}
    \]
    and
    \[
        V = \left\{
            \begin{array}{l}
\left(\begin{smallmatrix}7  \\ 0 \end{smallmatrix}\right),
\left(\begin{smallmatrix}7  \\ 1 \end{smallmatrix}\right),
\left(\begin{smallmatrix}8  \\ 1 \end{smallmatrix}\right),
\left(\begin{smallmatrix}10 \\ 1 \end{smallmatrix}\right),
\left(\begin{smallmatrix}13 \\ 2 \end{smallmatrix}\right),
\left(\begin{smallmatrix}13 \\ 3 \end{smallmatrix}\right),
\left(\begin{smallmatrix}11 \\ 4 \end{smallmatrix}\right),
\left(\begin{smallmatrix}11 \\ 5 \end{smallmatrix}\right),
\left(\begin{smallmatrix}12 \\ 6 \end{smallmatrix}\right),
\left(\begin{smallmatrix}14 \\ 6 \end{smallmatrix}\right),\\
\left(\begin{smallmatrix}0  \\ 7 \end{smallmatrix}\right),
\left(\begin{smallmatrix}8  \\ 7 \end{smallmatrix}\right),
\left(\begin{smallmatrix}10 \\ 7 \end{smallmatrix}\right),
\left(\begin{smallmatrix}1  \\ 8 \end{smallmatrix}\right),
\left(\begin{smallmatrix}9  \\ 8 \end{smallmatrix}\right),
\left(\begin{smallmatrix}1  \\ 9 \end{smallmatrix}\right),
\left(\begin{smallmatrix}1  \\ 10\end{smallmatrix}\right),
\left(\begin{smallmatrix}9  \\ 10\end{smallmatrix}\right),
\left(\begin{smallmatrix}4  \\ 11\end{smallmatrix}\right),
\left(\begin{smallmatrix}6  \\ 11\end{smallmatrix}\right),\\
\left(\begin{smallmatrix}15 \\ 11\end{smallmatrix}\right),
\left(\begin{smallmatrix}2  \\ 12\end{smallmatrix}\right),
\left(\begin{smallmatrix}6  \\ 12\end{smallmatrix}\right),
\left(\begin{smallmatrix}11 \\ 12\end{smallmatrix}\right),
\left(\begin{smallmatrix}15 \\ 12\end{smallmatrix}\right),
\left(\begin{smallmatrix}3  \\ 13 \end{smallmatrix}\right),
\left(\begin{smallmatrix}12 \\ 13\end{smallmatrix}\right),
\left(\begin{smallmatrix}14 \\ 13\end{smallmatrix}\right),
\left(\begin{smallmatrix}3  \\ 14\end{smallmatrix}\right),\\
\left(\begin{smallmatrix}12 \\ 14\end{smallmatrix}\right),
\left(\begin{smallmatrix}5  \\ 15\end{smallmatrix}\right),
        \end{array}\right\}
    \]
    Prove that $H$ and $V$ are exactly the dominoes that appear in $\Lcal_\Phi$,
    that is, 
    \[
    (\Acal\odot^1\Acal)\cap\Lcal_\Phi=H
    \text{ and }
    (\Acal\odot^2\Acal)\cap\Lcal_\Phi=V
    \]
    where $\Acal=\Zrange{15}$ is the alphabet on which the morphism $\Phi$ is
    defined.
\end{exo}


\begin{exo}[label={exo:belongs-in-the-set}]
    Is the $2\times2$ word
    $\left(\begin{smallmatrix}1 & 2\\9 & 12\end{smallmatrix}\right)$
    in the set $\Lcal_\Phi$?
\end{exo}

\begin{exo}[label={exo:list-45-2x2-elements-in-substitutive-shift}]
    List the 45 elements of the set $\Acal^{(2,2)}\cap\Lcal_\Phi$.
\end{exo}

\begin{exo}[label={exo:substitutive-shift-8-periodic-points}]
    Describe the 8 periodic points of $\Phi$, i.e. the configurations
    $x\in\Acal^{\Z^2}$ such that $\Phi^k(x)=x$ for some $k\geq1$.
\end{exo}

\subsection{$d$-dimensional recognizability and aperiodicity}

The definition of recognizability dates back to the work of Host, Qu\'effelec and
Moss\'e \cite{MR1168468}. 
The definition introduced below is based on work of Berth\'e et al.
\cite{MR4015135} on the recognizability in the case
of $S$-adic systems where more than one substitution is involved.

\begin{definition}[recognizable]
Let $X\subseteq\Acal^{\Z^d}$ and
$\omega:X\to\Bcal^{\Z^d}$ be a $d$-dimensional morphism.
If $y\in\shiftclosure{\omega(X)}$, i.e.,
$y=\sigma^\bk\omega(x)$ for some $x\in X$ and $\bk\in\Z^d$, where $\sigma$ is
the $d$-dimensional shift map, we say that $(\bk,x)$ is an
\defn{$\omega$-representation of $y$}. We say that it is \defn{centered} if
$y_\zero$ lies inside of the image of $x_\zero$, i.e., if
$\zero\leq\bk<\shape(\omega(x_\zero))$ coordinate-wise.
    We say that $\omega$ is \defn{recognizable in} $X\subseteq\Acal^{\Z^d}$
if each $y\in\Bcal^{\Z^d}$ has at most one centered $\omega$-representation 
$(\bk,x)$ with $x\in X$.
\end{definition}


\begin{lemma}\label{lem:aperiodic-implies-aperiodic}
    Let $\omega:X\to Y$ be some $d$-dimensional morphism between two
    subshifts $X$ and $Y$.
\begin{enumerate}
\item 
   If $Y$ is aperiodic, then $X$ is aperiodic.
\item If $X$ is aperiodic and $\omega$ is recognizable in $X$,
    then $\shiftclosure{\omega(X)}$ is aperiodic.
\end{enumerate}
\end{lemma}

\begin{proof}
    If $X$ contains a periodic configuration $x$, then
    $\omega(x)\in Y$ is periodic.

    (ii)
    Let $y\in\shiftclosure{\omega(X)}$.
    Then, there exist $\bk\in\Z^d$ and $x\in X$ such that 
    $(\bk, x)$ is a centered $\omega$-representation of $y$, i.e.,
    $y=\sigma^\bk\omega(x)$.
    Suppose by contradiction that $y$ has a nontrivial period
    $\bp\in\Z^d\setminus\zero$.
    Since $y =\sigma^\bp y =\sigma^{\bp+\bk}\omega(x)$,
    we have that
    $(\bp+\bk, x)$ is an $\omega$-representation of $y$.
    Since $\omega$ is recognizable, this representation is not centered.
    Therefore there exists $\bq\in\Z^d\setminus\zero$ such that
    $y_\zero$ lies in the image of $x_\bq=(\sigma^\bq x)_\zero$.
    Therefore there exists $\bk'\in\Z^d$ such that
    $(\bk',\sigma^\bq x)$ is a centered $\omega$-representation of $y$.
    Since $\omega$ is recognizable, we conclude that
    $\bk=\bk'$ and $x=\sigma^\bq x$. Then $x\in X$ is periodic which is a
    contradiction.
\end{proof}

In general, $\omega(X)$ is not closed under the shift which implies that
$\omega$ is not onto $Y$. This motivates the following definition.

\begin{definition}
Let $X$, $Y$ be two subshifts
and $\omega:X\to Y$ be a $d$-dimensional morphism.
If $Y=\shiftclosure{\omega(X)}$, then
we say that $\omega$
is \defn{onto up to a shift}.
\end{definition}

The next proposition is well-known, see \cite{MR1637896,MR1168468}, who showed that
recognizability and aperiodicity are equivalent for primitive substitutive
sequences. We state and prove only one direction (the easy one) of the equivalence
which does not need the notion of primitivity.

\begin{proposition}\label{prop:expansive-recognizable-aperiodic}
    Let $X\subseteq\Acal^{\Z^d}$ be a self-similar subshift 
    satisfying $\shiftclosure{\omega(X)}=X$
    for some expansive $d$-dimensional morphism
    $\omega:\Acal\to\Acal^{\Z^d}$.
    If $\omega$ is recognizable in $X$, then $X$ is aperiodic.
\end{proposition}

\begin{proof}
    Suppose that there exists a periodic configuration $y\in X$ with
    period $\bp\in\Z^d\setminus\{(0,0)\}$ satisfying $\sigma^\bp y=y$.
    Since $\omega$ is expansive,
    let $m\in\N$ such that the shape of the image of every letter
    $a\in\Acal$ by $\omega^m$ is large enough, that is,
    $\shape(\omega^m(a))>\bp$ for every
    letter $a\in\Acal$.
    By hypothesis,
    every $y\in X$ has an $\omega$-representation. 
    Recursively, there
    exists an $\omega^m$-representation $(\bk,x)$ of $y$ satisfying
    $y=\sigma^\bk\omega^m(x)$.
    We may assume that it is centered since $X$ is shift-invariant.
    By definition of centered representation,
    for every $\bu\in\Z^d$ such that $\zero\leq\bu<\shape(\omega^m(x_\zero))$,
    $(\bu,x)$ is a centered $\omega^m$-representation of 
    $\sigma^\bu\omega^m(x)=\sigma^{\bu-\bk}y$.
    By the choice of $m$, there exists $\bu\in\Z^d$ such that
    $\zero\leq\bu<\shape(\omega^m(x_\zero))$
    and
    $\zero\leq\bu+\bp<\shape(\omega^m(x_\zero))$.
    Therefore
    $(\bu,x)$ is a centered $\omega^m$-representation of 
    $\sigma^\bu\omega^m(x)=\sigma^{\bu-\bk}y$
    and
    $(\bu+\bp,x)$ is a centered $\omega^m$-representation of 
    $\sigma^{\bu+\bp}\omega^m(x)
    =\sigma^{\bu+\bp-\bk}y
    =\sigma^{\bu-\bk}\sigma^\bp y
    =\sigma^{\bu-\bk}y$.
    Therefore, $\omega^m$ is not recognizable
    which implies that $\omega$ is not recognizable
    which is a contradiction.
    We conclude that there is no periodic configuration $y\in X$.
\end{proof}

\begin{exo}[label={exo:representations}]
    Let $\Phi:\Zrange{15}\to\Zrange{15}^{*^2}$
    be the morphism defined in Exercise~\ref{exo:easy-image-by-Phi}.
    Find periodic configurations
    $x,y\in\Zrange{15}^{\Z^2}$
    and $\bk\in\Z^2$
    such that 
    \begin{itemize}
    \item $(\bk,x)$ is a $\Phi$-representation of $y$,
    \item $(\bk,x)$ is a centered $\Phi$-representation of $y$,
    \item $(\bk,x)$ is a $\Phi$-representation of $y$ which is not centered.
    \end{itemize}
\end{exo}

\begin{exo}[label={exo:more-one-centered-representations}]
    Does there exist a configuration $y\in\Zrange{15}^{\Z^2}$ that
    has more than one centered $\Phi$-representation $(\bk,x)$ with
    $x\in\Zrange{15}^{\Z^2}$?
\end{exo}

\begin{exo}[label={exo:prove-recognizable-aperiodic}]
Let $\Phi:\Zrange{15}\to\Zrange{15}^{*^2}$
be the morphism defined in Exercise~\ref{exo:easy-image-by-Phi}.
    \begin{enumerate}
        \item Prove that $\Phi$ is recognizable in $\Xcal_\Phi$.
        \item Prove that $\Xcal_\Phi$ is aperiodic.
    \end{enumerate}
\end{exo}

\subsection{Primitivity and minimality of self-similar subshifts}\label{sec:self-similar}

Substitutive shifts obtained from expansive and primitive morphisms are
interesting for their properties.
As in the one-dimensional case, we say that $\omega$ is \defn{primitive}
if there exists $m\in\N$ such that
for every $a,b\in\Acal$ the letter $b$ occurs in $\omega^m(a)$.

\begin{lemma}\label{lem:xcal_omega_minimal}
    Let $\omega:\Acal\to\Acal^{*^d}$ be an expansive and primitive $d$-dimensional
    morphism. Then $\Xcal_\omega$ is minimal, i.e., it contains no nonempty proper
    subshift.
\end{lemma}

\begin{proof}
    The substitutive shift of $\omega$ is well-defined since $\omega$ is expansive
    and it is minimal since $\omega$ is primitive 
    using standard arguments \cite[\S 5.2]{MR2590264}.
\end{proof}

The following two lemmas were proved in \cite{labbe_metallic_I_2024}.
We reproduce their proof here for completeness.

\begin{lemma}\label{lem:substitutive-contains-self-similar-part}
    {\rm\cite[Lemma~10.1]{labbe_metallic_I_2024}}
    Let $\omega:\Acal\to\Acal^{*^d}$ be an expansive and primitive $d$-dimensional
    morphism. Let $X\subseteq\Acal^{\Z^d}$ be a nonempty subshift 
    such that $X=\shiftclosure{\omega(X)}$. Then 
    $\Xcal_\omega\subseteq X$.
\end{lemma}

\begin{proof}
    We show that $\Lcal_\omega\subseteq\Lcal(X)$
    which implies that $\Xcal_\omega\subseteq X$.
    Let $u\in\Lcal_\omega$.
    From the definition of $\Lcal_\omega$,
    there exists $b\in\Acal$ and $n\in\N$ such that
    $u$ is a $d$-dimensional subword of $\omega^n(b)$.

    Since $X$ is nonempty,
    there exists a letter $a\in\Acal\cap\Lcal(x)$.
    From the primitivity of $\omega$, 
    there exists $m\geq1$ such that
    $\omega^m(a)$ contains an occurrence of the letter $b$.
    Therefore $\omega^{m+n}(a)$ contains an occurence of $u$.

    Since $X$ is self-similar, its
    language is also self-similar satisfying
    $\Lcal(X)=\factorialclosure{\omega(\Lcal(X))}$.
    Since $\omega(\Lcal(X))\subset\Lcal(X)$
    and $a\in\Lcal(x)$,
    for every integer $N\geq1$,
    the $d$-dimensional word $\omega^N(a)$ is in the
    language $\Lcal(X)$.
    Thus, we have
    \[
        u 
        \in
        \Lcal(\omega^n(b))
        \subset
        \Lcal(\omega^{m+n}(a))
        \subset
        \Lcal(X).
    \]
    We conclude that $\Lcal_\omega\subseteq\Lcal(X)$ and
    $\Xcal_\omega\subseteq X$.
\end{proof}

Recall that $\Xcal_\omega=\shiftclosure{\omega(\Xcal_\omega)}$.
Thus from Lemma~\ref{lem:xcal_omega_minimal} and
Lemma~\ref{lem:substitutive-contains-self-similar-part},
when $\omega$ is expansive and primitive, then 
$\Xcal_\omega$ is the smallest nonempty subshift
$X\subseteq\Acal^{\Z^d}$ satisfying $X=\shiftclosure{\omega(X)}$.
The next result provides a criterion
for the minimality of a self-similar subshift $X$ satisfying
$X=\shiftclosure{\omega(X)}$.

To achieve this goal, it is convenient to
consider the $2\times 2$ patterns as well as the domino patterns
of shape $1\times 2$ and $2\times 1$.
We use these dominoes to define two equivalence relations on the alphabet $\Acal$.
Formally, the vertical dominoes of shape $1\times 2$ 
appearing in the language $\Lcal(\Xcal_\omega)$
define an equivalence relation
$\equiv_2$ on $\Acal$ given as the reflexive, symmetric and transitive closure
of the pairs 
$\{(a,c)\in\Acal\times\Acal\mid 
 \left(\begin{smallmatrix} a\\ c \end{smallmatrix}\right)
 \in\Lcal(\Xcal_\omega)\}$.
Informally, $a\equiv_2 c$ for some letters $a,c\in\Acal$ means that letters $a$ and $c$
may appear in the same column in some configuration of $\Xcal_\omega$.
Similarly, the horizontal dominoes of shape $2\times 1$ 
appearing in the language $\Lcal(\Xcal_\omega)$
define an equivalence relation
$\equiv_1$ on $\Acal$ given as the reflexive, symmetric and transitive closure
of the pairs 
$\{(a,b)\in\Acal\times\Acal\mid 
 \left(\begin{smallmatrix} a& b \end{smallmatrix}\right)
 \in\Lcal(\Xcal_\omega)\}$.

Using these two equivalence relations $\equiv_1$ and $\equiv_2$ on the alphabet $\Acal$,
we consider the following graphs:
\begin{itemize}
    \item
Let $G_\omega^{2\times2}=(V_\omega^{2\times2},E_\omega^{2\times2})$ be the
directed graph whose vertices and edges are
\begin{align*}
    V_\omega^{2\times2} &= \left\{
        \left(\begin{smallmatrix}
            a&b\\
            c&d
        \end{smallmatrix}\right)
        \in\Acal^{2\times 2}
        \mid
        a\equiv_1 b,
        c\equiv_1 d,
        a\equiv_2 c,
        b\equiv_2 d
        \right\},\\
    E_\omega^{2\times2} &= \left\{
        \left(\begin{smallmatrix}
            e&f\\
            g&h
        \end{smallmatrix}\right)
        \to
        \left(\begin{smallmatrix}
            a&b\\
            c&d
        \end{smallmatrix}\right)
        \middle|
        \begin{array}{l}
        a \text{ is the lower right letter of } \omega(e),\\
        b \text{ is the lower left letter of }  \omega(f),\\
        c \text{ is the top right letter of }   \omega(g),\\
        d \text{ is the top left letter of }    \omega(h)\\
        \end{array}
        \right\}.
\end{align*}
    \item
Let $G_\omega^{2\times1}=(V_\omega^{2\times1},E_\omega^{2\times1})$ be the
directed graph whose vertices and edges are 
\begin{align*}
    V_\omega^{2\times1} &= \left\{
        \left(\begin{smallmatrix}
            a&b
        \end{smallmatrix}\right)
        \in\Acal^{2\times1}
        \mid
        a\equiv_1 b
        \right\},\\
    E_\omega^{2\times1} &= \left\{
        \left(\begin{smallmatrix}
            e&f
        \end{smallmatrix}\right)
        \to
        \left(\begin{smallmatrix}
            a&b
        \end{smallmatrix}\right)
        \middle|
        \begin{array}{l}
        \text{there exists an integer $j$ such that $0\leq j < \height(\omega(e))$ and}\\
        a \text{ is the letter in the $j$-th row in the right-most column of } \omega(e),\\
        b \text{ is the letter in the $j$-th row in the left-most column of } \omega(f)\\
        \end{array}
        \right\}.
\end{align*}
    \item
Let $G_\omega^{1\times2}=(V_\omega^{1\times2},E_\omega^{1\times2})$ be the
directed graph whose vertices and edges are 
\begin{align*}
    V_\omega^{1\times2} &= \left\{
        \left(\begin{smallmatrix}
            a\\
            c
        \end{smallmatrix}\right)
        \in\Acal^{1\times2}
        \mid
        a\equiv_2 c
        \right\},\\
    E_\omega^{1\times2} &= \left\{
        \left(\begin{smallmatrix}
            e\\
            g
        \end{smallmatrix}\right)
        \to
        \left(\begin{smallmatrix}
            a\\
            c
        \end{smallmatrix}\right)
        \middle|
        \begin{array}{l}
        \text{there exists an integer $i$ such that $0\leq i < \width(\omega(e))$ and}\\
        a \text{ is the letter in the $i$-th column in the bottom-most row of } \omega(e),\\
        c \text{ is the letter in the $i$-th column in the top-most row of } \omega(g)\\
        \end{array}
        \right\}.
\end{align*}
\end{itemize}
    Finally, for every directed graph $G=(V,E)$, we define
    \[
        \RecurrentVertices(G)
            =
            \{v\in V
            \mid
            v \text{ belongs to a cycle of }G\}.
    \]
    A vertex is a recurrent vertex in a graph if and only
    if it belongs to a biinfinite path in the graph,
    thus the terminology of \emph{recurrent}.

    The following result from \cite{labbe_metallic_I_2024}
    allows to conclude that a self-similar subshift is minimal 
    even when the 2-dimensional substitution admits more than one self-similar
    subshift (some made of configurations which are not uniformly recurrent).
    We reproduce its proof here for completeness.

\begin{lemma}\label{lem:criterion-for-minimality}
    {\rm\cite[Lemma~10.4]{labbe_metallic_I_2024}}
    Let $X=\shiftclosure{\omega(X)}$ be a nonempty self-similar subshift
    where $\omega:\Acal\to\Acal^{*^d}$ is an expansive and primitive $2$-dimensional
    morphism.
    The following are equivalent:
    \begin{enumerate}[(i)]
        \item $\Lcal(X)\cap\RecurrentVertices(G^s_\omega)\subset\Lcal(\Xcal_\omega)$
              for every shape $s\in\{2\times2, 2\times1, 1\times2\}$,
    \item $X=\Xcal_\omega$,
    \item $X$ is minimal.
    \end{enumerate}
\end{lemma}


\begin{proof}
    Assume that $X=\overline{\omega(X)}^{\sigma}$ for some $\varnothing\neq X\subseteq\Acal^{\Z^d}$.

    (i) $\implies$ (ii)
    From Lemma~\ref{lem:substitutive-contains-self-similar-part},
    we have $\Xcal_\omega\subseteq X$.
    Let $z\in\Lcal(X)$. 
    We want to show that $z\in\Lcal(\Xcal_\omega)$. 
    Since $\omega$ is expansive,
    let $m\in\N$ such that the image of every letter
    $a\in\Acal$ by $\omega^m$ is larger than $z$, that is, 
    $\shape(\omega^m(a))\geq\shape(z)$ for all $a\in\Acal$.
    We have $z\in\Lcal(X)=\Lcal\left(\omega^m(\Lcal(X))\right)$.
    By the choice of $m$, $z$ cannot overlap more than two blocks
    $\omega^m(a)$ in the same direction. Thus, there exists a word $u\in\Lcal(X)$ of
    shape 
    $1\times 1$,
    $2\times 1$,
    $1\times 2$ or
    $2\times 2$ such that $z$ is a subword of $\omega^m(u)$.
    If $u$ is of shape $1\times 1$, then $z\in\Lcal(\Xcal_\omega)$.
    We may assume that the word $u$ has the smallest possible rectangular shape
    $s\in \{2\times 1, 1\times 2, 2\times 2\}$.

    We have $u\in V_\omega^{s}$. Since $u\in\Lcal(X)$ and $X$ is self-similar,
    there exists a sequence $(u_k)_{k\in\N}$ with $u_k\in V_\omega^{s}\cap\Lcal(X)$ 
    for all $k\in\N$
    such that
    \[
        \cdots\rightarrow
        u_{k+1} \rightarrow 
        u_k \rightarrow 
        \cdots\rightarrow
        u_1 \rightarrow u_0 = u
    \]
    is a left-infinite path in the graph $G_\omega^{s}$.
    Since $V_\omega^{s}$ is finite, there exist some $k,k'\in\N$ with $k<k'$
    such that $u_k=u_{k'}$.
    Thus $u_k\in\RecurrentVertices(G^s_\omega)$
    and $u$ is a subword of $\omega^k(u_k)$.
    From the hypothesis, we have $u_k\in\Lcal(\Xcal_\omega)$.
    Since $\omega$ is primitive, there exists $\ell$ such that 
    $u_k$ is a subword of $\omega^\ell(a)$ for every $a\in\Acal$.
    Therefore, $z$ is a subword of $\omega^{m+k+\ell}(a)$ for every $a\in\Acal$.
    Then $z\in\Lcal(\Xcal_\omega)$ and $\Lcal(X)\subseteq\Lcal(\Xcal_\omega)$.
    Thus $X\subseteq\Xcal_\omega$ and $X=\Xcal_\omega$.

    (ii) $\implies$ (i)
    If $X=\Xcal_\omega$,
    then $\Lcal(X)=\Lcal(\Xcal_\omega)$.
    Thus $\Lcal(X)\cap\RecurrentVertices(G^s_\omega)\subset\Lcal(X)=\Lcal(\Xcal_\omega)$
    for every shape $s\in\{2\times2, 2\times1, 1\times2\}$.

    (ii) $\implies$ (iii)
    From Lemma~\ref{lem:xcal_omega_minimal}, the substitutive shift
    $\Xcal_\omega$ is minimal.

    (iii) $\implies$ (ii)
    From Lemma~\ref{lem:substitutive-contains-self-similar-part},
    we have $\Xcal_\omega\subseteq X$.
    Since $X$ is minimal, we conclude that $\Xcal_\omega=X$.
\end{proof}

\begin{exo}[label={exo:prove-primitive-minimal}]
Let $\Phi:\Zrange{15}\to\Zrange{15}^{*^2}$
be the morphism defined in Exercise~\ref{exo:easy-image-by-Phi}.
    \begin{enumerate}
        \item Prove that $\Phi$ is primitive.
        \item Prove that $\Xcal_\Phi$ is minimal.
    \end{enumerate}
\end{exo}

\begin{exo}[label={exo:recurrent-vertices}]
    Compute the sets
    $\RecurrentVertices(G^s_\Phi)$
    for every shape $s\in\{2\times2, 2\times1, 1\times2\}$.
    Conclude that the 2-dimensional substitution $\Phi$ does not have 
    a unique self-similar subshift $X=\shiftclosure{\Phi(X)}$.
\end{exo}

\subsection{Markers}\label{subsection:markers}

The goal of this section is to prove Theorem~\ref{thm:if-markers-desubstitute}
which states that we can desubstitute a subshift provided that its alphabet
contains a subset of markers.
Markers are such that they appear on non-consecutive layers in the
configurations of the subshift, see Definition~\ref{def:markers}.
The results are stated for arbitrary dimension since their proofs are
independent of the dimension, but we will use them in the 2-dimensional case
afterwards.

We now define the notion of markers for subshifts $X\subset\Acal^{\Z^d}$
and prove that their presence allows to desubstitute uniquely the configurations in $X$ 
using a $d$-dimensional morphism.
Originally, those results were proved for $d=2$ in order
to desubstitute configurations from Wang shifts, see
\cite{MR3978536} and \cite{MR4226493}.
It turns out that the notion of markers is more general and the results hold in general subshifts 
$X\subset\Acal^{\Z^d}$.

Recall that if $w:\Z^d\to\Acal$ is a configuration and $a\in\Acal$ is a letter, then
$w^{-1}(a)\subset\Z^d$ is the set of positions where the letter $a$ appears in
$w$.

\begin{definition}\label{def:markers}
    Let $\Acal$ be an alphabet
    and $X\subset\Acal^{\Z^d}$ be a subshift.
    A nonempty subset $M\subset\Acal$ is called \defn{subset of markers in the
    direction $\be_i$}, with $i\in\{1,\dots,d\}$,
    if positions of the letters of $M$ in any configuration are nonadjacent $(d-1)$-dimensional
    layers orthogonal to $\be_i$, that is,
    for all configurations $w\in X$
    there exists $P\subset\Z$ such that
    the positions of the markers satisfy
    \begin{equation*}
        w^{-1}(M) = P\be_i +\sum_{k\neq i} \Z\be_k
        \quad
        \text{ with }
        \quad
        1\notin P-P
    \end{equation*}
    where $P-P=\{b-a\mid a\in P, b\in P\}$ is the set of differences
    between elements of $P$.
\end{definition}

In Figure~\ref{fig:Phi-on-seed-8-12-1-6}, we may observe that not every letter
appear at every row. In particular, the letters in the set $\{0,1,2,3,4,5,6\}$
appear on nonadjacent rows. Thus, this is an example of a subset of markers
in the direction $\be_2$ (see Exercise~\ref{exo:marker-e2-in-XPhi}).

Note that it follows from the definition that a subset of markers is a proper
subset of $\Acal$ as the case $M=\Acal$ is impossible.

Proving that a subset $M\subset\Acal$ is a subset of
markers uses very local observations, namely the set of dominoes in the
language of the subshift. It leads to the following criterion.

\begin{lemma}\label{lem:criteria-markers}
    Let $\Acal$ be an alphabet
    and $X\subset\Acal^{\Z^d}$ be a subshift.
    A nonempty subset $M\subset\Acal$ is a subset of markers in the
    direction $\be_i$ if and only if
    \begin{equation*}
        M\odot^i M, \qquad
        M\odot^k (\Acal\setminus M), \qquad
        (\Acal\setminus M) \odot^k M
    \end{equation*}
    are forbidden in $X$ for every $k\in\{1,\dots,d\}\setminus\{i\}$.
\end{lemma}

\begin{proof}
    Suppose that $M\subset\Acal$ is a subset of markers in the
    direction $\be_i$.
    For any configuration $w\in X$,
    there exists $P\subset\Z$ such that
        $w^{-1}(M) = P\be_i +\sum_{k\neq i} \Z\be_k$ with 
        $1\notin P-P$.
    In any configuration $w\in X$ such that $w(p)\in M$,
    then $w(p\pm\be_k)\in M$ also belongs to $M$ for every $k\neq i$.
    Therefore, $M\odot^k (\Acal\setminus M)$ and $(\Acal\setminus M) \odot^k M$
    are forbidden in $X$ for every $k\neq i$.
    Moreover, the fact that
        $1\notin P-P$ implies that
        $M\odot^k M$ is forbidden in $X$.

    Conversely, suppose that
        $M\odot^i M$,
        $M\odot^k (\Acal\setminus M)$ and
        $(\Acal\setminus M) \odot^k M$
    are forbidden in $X$ for every $k\neq i$.
    The last two conditions implies that 
    in any configuration $w\in X$ such that $w(p)\in M$,
    then $w(p\pm\be_k)\in M$ also belongs to $M$ for every $k\neq i$.
    Therefore letters in $M$ appear as complete layers in $w$, that is,
        $w^{-1}(M) = P\be_i +\sum_{k\neq i} \Z\be_k$
        for some $P\subset\Z$.
    Since $M\odot^i M$ is forbidden in $X$, it means that
    the layers are nonadjacent, or equivalently, $1\notin P-P$.
    We conclude that $M$ is a subset of markers in the
    direction $\be_i$.
\end{proof}

If $a\odot^i b$ is a domino in the direction $\be_i$,
we say that $a$ is \defn{on the left} position and $b$ is \defn{on the right}
position in the domino.

The existence of a subset of markers allows to desubstitute a subshift by
``merging'' each marker to the letter on its right (or on its left).
This procedure creates a substitution with a specific form sending a letter on
a letter or a domino. 
In the next lemma, we provide a sufficient condition for
such substitutions to be recognizable.

\begin{lemma}\label{lem:recognizable}
    Let $d\geq1$ and $i$ such that $1\leq i\leq d$.
Let $\omega:\Bcal\to\Acal^{\Z^d}$ be a $d$-dimensional morphism such that
the image of letters are letters or dominoes in the direction $\be_i$.
If $\omega|_\Bcal$ is injective and there exists a subset $M\subset \Acal$ such that
\begin{equation}\label{eq:markers-on-right}
    \omega(\Bcal)
    \subseteq (\Acal\setminus M)
    \cup \left((\Acal\setminus M)\odot^i M\right),
\end{equation}
or
\begin{equation}\label{eq:markers-on-left}
    \omega(\Bcal)
    \subseteq (\Acal\setminus M)
    \cup \left(M\odot^i (\Acal\setminus M)\right),
\end{equation}
    then $\omega$ is recognizable in $\Bcal^{\Z^d}$.
\end{lemma}

\begin{proof}
    Let $(\bk,x)$ and $(\bk',x')$ be two
    centered $\omega$-representations of $y\in\Acal^{\Z^d}$ with
    $\bk,\bk'\in\Z^d$ and $x,x'\in\Bcal^{\Z^d}$.
    We want to show that they are equal.

    Since the image of a letter under $\omega$ is a letter or a domino in the
    direction $\be_i$, then $\bk,\bk'\in\{\zero,\be_i\}$.
    If $y_\zero\in M$, then $y_\zero$ appears as the left or right part of a
    domino and thus $\bk=\bk'=\be_i$
    if Equation~\eqref{eq:markers-on-right} holds
    or $\bk=\bk'=\zero$
    if Equation~\eqref{eq:markers-on-left} holds.

    Suppose now that $y_\zero\in \Acal\setminus M$. 
    If Equation~\eqref{eq:markers-on-right} holds,
    then $\bk=\bk'=\zero$.
    Suppose that Equation~\eqref{eq:markers-on-left} holds.
    By contradiction, suppose that $\bk\neq\bk'$ 
    and assume without loss of generality that
    $\bk=\zero$ and $\bk'=\be_i$.
    This means that $\omega(x'_\zero)=y_{-\be_i} \odot^i y_\zero$ is
    a domino in the direction $\be_i$. 
    Since $y_\zero\in \Acal\setminus M$, we must have that $y_{-\be_i}\in M$
    is a marker on the left.
    This is impossible as
    $\omega(x_{-\be_i})=y_{-\be_i}\in \Acal\setminus M$
    or
    $\omega(x_{-\be_i})=y_{-2\be_i} \odot^i y_{-\be_i}
    \in M\odot^i (\Acal\setminus M)$.
    Therefore, we must have $\bk=\bk'$ and
    $\omega(x)=\omega(x')$.

    Suppose by contradiction that $x\neq x'$.
    Let $\ba=(a_1,\dots,a_d)\in\Z^d$ be some minimal vector with respect to
    $\Vert\ba\Vert_\infty$ such that $x_\ba\neq x'_\ba$.
    From the minimality of the norm of $\ba$, we have that
    $\omega(x_\ba)$ occurs in $\omega(x)$
    at the same position as
    $\omega(x'_\ba)$ occurs in $\omega(x')$.
    If 
    $\omega(x_\ba)$ and
    $\omega(x'_\ba)$ have the same shape, then it implies
    that $\omega(x_\ba)=\omega(x'_\ba)$,
    which contradicts the injectivity of $\omega|_\Bcal$.
    Thus $\omega(x_\ba)$ and
    $\omega(x'_\ba)$ must have different shapes.
    Suppose without loss of generality that
    $\omega(x_\ba)\in\Acal $ and
    $\omega(x'_\ba)=b\odot^i c\in\Acal\odot^i\Acal$.
    We need to consider two cases: $a_i\geq 0$ and $a_i<0$.

    Suppose $a_i\geq 0$.
    We must have that Equation~\eqref{eq:markers-on-right} holds.
    We have $\omega(x_\ba)=b\in\Acal\setminus M$ and $c\in M$.
    But then $\omega(x_{\ba+\be_i})=c$
    or $\omega(x_{\ba+\be_i})=c\odot^i d$ for some 
    $c\in\Acal\setminus M$ and $d\in\Acal$ which is a contradiction.

    Suppose $a_i<0$.
    We must have that Equation~\eqref{eq:markers-on-left} holds.
    We have $\omega(x_\ba)=c\in\Acal\setminus M$ and $b\in M$.
    But then $\omega(x_{\ba-\be_i})=b$
    or $\omega(x_{\ba+\be_i})=d\odot^i b$ for some 
    $b\in\Acal\setminus M$ and
    $d\in\Acal$ which is a contradiction.
    We conclude that $x=x'$.
\end{proof}

The presence of markers allows to
desubstitute uniquely the configurations of a subshift.
There is even a choice to be made in the construction of the substitution. 
We may construct the substitution in such a way that the markers are on the left or
on the right in the image of letters that are dominoes 
in the direction $\be_k$. We make this distinction in the statement of
the following result which was stated in the context of Wang shifts
in \cite{MR3978536} and \cite{MR4226493}.

\begin{theorem}\label{thm:if-markers-desubstitute}
    Let $\Acal$ be an alphabet
    and $X\subset\Acal^{\Z^d}$ be a subshift.
    If there exists a subset
    $M\subset\Acal$ 
    of markers in the direction 
    $\be_i\in\{\be_1,\dots,\be_d\}$,
    then 
\begin{enumerate}
    \item[(i)] (markers on the right) there exists an alphabet $\Bcal_R$,
    a subshift $Y\subset\Bcal_R^{\Z^d}$
    and a $d$-dimensional morphism
    $\omega_R:Y\to X$
    such that 
    \begin{equation*}
        \omega_R(\Bcal_R)\subseteq (\Acal\setminus M)\cup 
        \left((\Acal\setminus M)\odot^i M\right)
    \end{equation*}
    which is recognizable and onto up to a shift and
\item[(ii)] (markers on the left) there exists an alphabet $\Bcal_L$,
    a subshift $Y\subset\Bcal_L^{\Z^d}$
    and a $d$-dimensional morphism
    $\omega_L:Y\to X$
    such that 
    \begin{equation*}
        \omega_L(\Bcal_L)\subseteq (\Acal\setminus M)\cup 
        \left(M\odot^i (\Acal\setminus M)\right)
    \end{equation*}
    which is recognizable and onto up to a shift.
\end{enumerate}
\end{theorem}

\begin{proof}
    We do only the proof of (i) when the markers are on the right, since
    one case can be deduced from the other using symmetry.

    Since $X$ is a subshift, there exists a language $F\subset\Acal^{*^d}$
    such that $X$ is the set of configurations of $\Acal^{\Z^d}$ without any
    occurrence of patterns from $F$. Notice that since $M$ is a set of markers
    in the direction $\be_i$, we may assume $M\odot^i M\subset F$.
    Let $P\subset\Acal\odot^i\Acal$ and $Q\subset\Acal$ be the following sets:
    \begin{align*}
        P &= \left((\Acal\setminus M) \odot^i M\right) \setminus F,\\
        Q &= \left\{u\in\Acal\setminus M\,\middle|\, \text{ there exists }
        v\in\Acal\setminus
            M \text{ such that } u\odot^iv \notin F\right\}.
    \end{align*}
    We choose some ordering of their elements with indices starting from zero:
    \begin{align*}
		P&=\{p_0,\dots, p_{|P|-1}\},\\
		Q&=\{q_0,\dots, q_{|Q|-1}\}.
    \end{align*}
    We construct the alphabet $\Bcal=\left\{0,1,\dots,|Q|+|P|-1\right\}$
    and define the rule $\omega$ by
    \begin{equation}\label{eq:omegaStoT}
    \begin{array}{rccl}
    \omega:&\Bcal & \to & \Acal^{*^d}\\
        &j& \mapsto&
    \begin{cases}
        q_j         & \text{ if } 0\leq j <  |Q|,\\
        p_{j-|Q|} & \text{ if }|Q|\leq j < |Q|+|P|.\\
    \end{cases}
    \end{array}
    \end{equation}

    We want to show that $\omega$ extends to a map from a set of configurations to $X$ which
    is onto up to a shift.
    Let $x\in X$ be a configuration which can be seen as a function
    $x:\Z^d\to\Acal$.
    Consider the set $x^{-1}(M)\subset\Z^d$ of positions of markers
    in~$x$. From the definition of markers in the direction $\be_i$, 
    markers appear in nonadjacent hyperplanes orthogonal to $\be_i$ in the configuration $x$.
    Formally, there exists
    a set $H\subset\Z$ such that $x^{-1}(M)=\Z^{i-1}\times H\times\Z^{d-i}$ and
    $1\notin H-H$.
    Since $1\notin H-H$, there exists a strictly increasing sequence
    $(a_k)_{k\in\Z}$ such that $\Z\setminus H=\{a_k\mid k\in\Z\}$.
    We assume that $a_0=0$ if $0\in\Z\setminus H$ and $a_0=-1$ if $0\in H$
    which makes the sequence $(a_k)_{k\in\Z}$ uniquely defined.

    In order to define the preimage of $x$ under $\omega$, it is convenient to
    represent $x$ fiber by fiber. For every $\Bm\in\Z^{d-1}$, let
    $x_\Bm:\Z\to\Acal$, be the sequence such that 
    \[
        x_{(n_1,\dots,n_{i-1},n_{i+1},\dots,n_d)}(n_i)
        =
        x(n_1,\dots,n_{i-1},n_i,n_{i+1},\dots,n_d)
    \]
    for every $\bn=(n_1,\dots,n_{i-1},n_i,n_{i+1},\dots,n_d)\in\Z^d$.
    For every $\Bm\in\Z^{d-1}$, let $y_\Bm:\Z\to\Bcal$ be defined as
    \begin{equation*}
    \begin{array}{rccl}
        y_\Bm:&\Z & \to & \Bcal\\
        &k & \mapsto &
    \begin{cases}
        j & \text{ if }a_k+1\in H \text{ and } x_\Bm(a_k)=q_j,\\
        j & \text{ if }a_k+1\notin H    \text{ and } x_\Bm(a_k)\odot^i x_\Bm(a_k+1)=p_{j-|Q|}, 
    \end{cases}
    \end{array}
    \end{equation*}
    The function $y_\Bm$ is well-defined.
    Indeed, let $k\in\Acal$ and $\Bm\in\Z^{d-1}$.
    If $a_k+1\in H$, then 
    $x_\Bm(a_k+1)\in\Acal\setminus M$
    and
    $y_\Bm(k)= x_\Bm(a_k)\in Q$.
    Also if $a_k+1\notin H$, then 
    $x_\Bm(a_k)\in\Acal\setminus M$
    and 
    $x_\Bm(a_k+1)\in M$.
    Since $x\in X$, then $x_\Bm(a_k)\odot^i x_\Bm(a_k+1)\notin F$
    and therefore $x_\Bm(a_k)\odot^i x_\Bm(a_k+1)\in P$.
    We define the configuration $y:\Z^d\to\Bcal$ by its fibers constructed above
    \[
        y(n_1,\dots,n_{i-1},n_i,n_{i+1},\dots,n_d)
        =
        y_{(n_1,\dots,n_{i-1},n_{i+1},\dots,n_d)}(n_i)
    \]
    for every $\bn=(n_1,\dots,n_{i-1},n_i,n_{i+1},\dots,n_d)\in\Z^d$.

    We may now finish the proof of surjectivity.
If $0\in H$, 
then the configuration $x$ is exactly the image under $\omega$ of the configuration $y\in\Bcal^{\Z^d}$ that
we constructed: $x=\omega(y)$.
If $0\notin H$, then 
the configuration $x$ is a shift of the image under $\omega$ of the
configuration $y\in\Bcal^{\Z^d}$: $x=\sigma^{\be_i}\omega(y)$.
    Let
    \[
        Y = \omega^{-1}(X) = \{y\in\Bcal^{\Z^d} \mid \omega(y)\in X\}
    \]
    making $\omega:Y\to X$ a continuous map which is onto up to a shift.
    Notice that the subset $Y$ is closed since $\omega$ is continuous and $X$ is closed.
    Moreover, $Y$ is shift-invariant since 
    for all $y\in Y$ and $\bn\in\Z^d$
    there exists $\bk\in\Z^d$ such that $\omega(\sigma^\bn(y))=\sigma^\bk(\omega(y))$
    meaning that $\sigma^\bn(y)\in Y$.
    Thus $Y$ is a subshift.

The function $\omega$ is of the form
\begin{equation*}
    \omega(\Bcal)
    \subseteq (\Acal\setminus M)
    \cup \left((\Acal\setminus M)\odot^i M\right)
\end{equation*}
and its restriction on $\Bcal$ is injective by construction.
    Therefore, we conclude from Lemma~\ref{lem:recognizable} that
    $\omega$ is recognizable in $\Bcal^{\Z^d}$.
\end{proof}

    Remark that if $X$ is an effective subshift
    one may also show that $Y$ is an effective subshift.
    Moreover if $X$ is a Wang shift, then $Y$ is a Wang shift
    and the Wang tiles defining $Y$ can be obtained from the Wang tiles
    defining $X$ together with some fusion operation. This is what was done in
    \cite{MR3978536} and \cite{MR4226493}.


\begin{exo}[label={exo:marker-e2-in-XPhi}]
    Using the value of $H$ and $V$ from Exercise~\ref{exo:HV-dominoes}, prove that
    $\{0,1,2,3,4,5,6\}$ is a subset of markers for the direction $\be_2$ in
    the subshift $\Xcal_\Phi$.
\end{exo}

\begin{exo}[label={exo:marker-e1-in-XPhi}]
    Using the value of $H$ and $V$ from Exercise~\ref{exo:HV-dominoes}, prove that
    $\left\{0, 1, 7, 8, 9, 10\right\}$, 
    is a subset of markers for the direction $\be_1$ in the subshift
    $\Xcal_\Phi$.
\end{exo}

\begin{exo}[label={exo:markers-find-a-subshift}]
    Using $M=\{0,1,2,3,4,5,6\}$ as subset of markers for the direction $\be_2$,
    find an alphabet $\Bcal$, a subshift $Y\subset\Bcal^{\Z^2}$ and 
    a 2-dimensional morphism $\alpha:\Bcal\to\Zrange{15}^{*^2}$ 
    such that
    \begin{equation*}
        \alpha(\Bcal)
        \subseteq (\Zrange{15}\setminus M)
        \cup \left((\Zrange{15}\setminus M)\odot^2 M\right)
    \end{equation*}
    which extends to a recognizable continuous map $\alpha:Y\to\Xcal_\Phi$ which is
    onto up to a shift.
\end{exo}

\begin{exo}[label={exo:desubstitute-markers-e1}]
    Using the subset $M=\{0, 1, 7, 8, 9, 10\}\subset\Acal$ 
    of markers for the direction $\be_1$,
    find an alphabet $\Ccal$, a subshift $Y'\subset\Ccal^{\Z^2}$ and 
    a 2-dimensional morphism $\xi:\Ccal\to\Zrange{15}^{*^2}$ 
    such that
    \begin{equation*}
        \xi(\Ccal)
        \subseteq (\Zrange{15}\setminus M)
        \cup \left((\Zrange{15}\setminus M)\odot^1 M\right)
    \end{equation*}
    which extends to a recognizable continuous map $\xi:Y'\to\Xcal_\Phi$ which is
    onto up to a shift.
\end{exo}

\section{A self-similar Wang shift}\label{chap:Labbe:sec:wang-shifts}

A \defn{Wang tile} 
is a tuple of four colors $(a,b,c,d)\in I\times J\times
I\times J$
where $I$
is a finite set of vertical colors
and $J$
is a finite set of horizontal colors, see
\cite{wang_proving_1961,MR0297572}.
A Wang tile is represented as a unit square with colored edges:
\begin{center}
\tile{$a$}{$b$}{$c$}{$d$}
\end{center}
For each Wang tile $\tau=(a,b,c,d)$, let
$\scright(\tau)=a$,
$\sctop(\tau)=b$,
$\scleft(\tau)=c$,
$\scbottom(\tau)=d$
denote respectively the colors of the right, top, left and bottom edges of $\tau$.

\begin{figure}[h]
\begin{center}
    \sageplot[][pdf]{Z.tikz(ncolumns=8)}
\end{center}
    \caption{The set $\Zcal=\{u_0,\dots,u_{15}\}$ of 16 Wang tiles 
    is a simplification made by Jana Lepšová \cite{lepsova_thesis_2024}
    of the set $\Ucal$ of 19 Wang tiles
    introduced in \cite{MR3978536}.  
    Each tile is identified uniquely by an index from the
    set $\{0,1,\dots,15\}$ written at the center each tile.}
    \label{fig:Z}
\end{figure}

Let $\Tcal=\{t_0,\dots,t_{m-1}\}$ be a set of Wang tiles as the one shown in Figure~\ref{fig:Z}.
A configuration $x:\Z^2\to\{0,\dots,m-1\}$ is \defn{valid} with respect to $\Tcal$ if
it assigns a tile in $\Tcal$ to each position of $\Z^2$ so that contiguous edges
of adjacent tiles have the same color, that is,
\begin{align}
    \scright(t_{x(\bn)})&=\scleft(t_{x(\bn+\be_1)})\label{eq:validwangtiling1}\\
    \sctop(t_{x(\bn)})&=\scbottom(t_{x(\bn+\be_2)})\label{eq:validwangtiling2}
\end{align}
for every $\bn\in\Z^2$ where $\be_1=(1,0)$ and $\be_2=(0,1)$.
A finite pattern which is valid with respect to $\Zcal$ is shown in Figure~\ref{fig:Z-5x5-tiling}.

\begin{sagesilent}
from slabbe import WangTiling
tilingZ5x5 = WangTiling(Phi([[12]], 3), Z.tiles())
\end{sagesilent}

\begin{figure}[h]
\begin{center}
\begin{tikzpicture}
    \node (A) at (0,0) {$\arraycolsep=1.8pt\sage{matrix(tilingZ5x5)}$};
    \node (B) at (5,0) {\sageplot[][pdf]{tilingZ5x5.tikz()}};
    \draw[|->] (A) to (B);
\end{tikzpicture}
\end{center}
    \caption{The finite $5\times 5$ pattern on the left is valid with respect
    to $\Zcal$ since it respects Equations~\eqref{eq:validwangtiling1}
    and~\eqref{eq:validwangtiling2} which we can verify on the tiling shown on
    the right.}
    \label{fig:Z-5x5-tiling}
\end{figure}

Let $\Omega_\Tcal\subset\{0,\dots,m-1\}^{\Z^2}$ denote the set of all valid 
configurations with respect to $\Tcal$,
called the \defn{Wang shift} of $\Tcal$. 
To a configuration $x\in\Omega_\Tcal$ corresponds a tiling of the plane $\R^2$ by
the tiles $\Tcal$ where the unit square Wang tile $t_{x(\bn)}$ is placed at position $\bn$ for every
$\bn\in\Z^2$, as in Figure~\ref{fig:Z-5x5-tiling}.
Together with the shift action $\sigma$ of $\Z^2$,
$\Omega_\Tcal$ is a SFT of the form \eqref{eq:SFT}
since there exists a finite set of
forbidden patterns made of all horizontal and vertical dominoes of two tiles
that do not share an edge of the same color.
This definition of Wang shifts allows to use the concepts of languages,
$2$-dimensional morphisms, recognizability introduced in the previous sections.


A configuration $x\in\Omega_\Tcal$ is \defn{periodic} if there exists
$\bn\in\Z^2\setminus\{0\}$ such that $x=\sigma^\bn(x)$.
A set of Wang tiles $\Tcal$ is \defn{periodic} if there exists a periodic configuration
$x\in\Omega_\Tcal$. 
Originally, Wang thought that every set of Wang tiles $\Tcal$ is periodic 
as soon as $\Omega_\Tcal$ is nonempty \cite{wang_proving_1961}.
This statement is equivalent to the existence of an algorithm 
solving the \emph{domino problem}, that is, taking as input a set of Wang tiles
and returning \textit{yes} or \textit{no} whether there exists a valid
configuration with these tiles. 
Berger, a student of Wang, later proved that the domino problem is undecidable
and he also provided a first example of an aperiodic set of Wang tiles
\cite{MR0216954}.
A set of Wang tiles $\Tcal$ is \defn{aperiodic} if
the Wang shift $\Omega_\Tcal$ is a nonempty aperiodic subshift.
This means that in general one cannot decide the emptiness of a Wang shift
$\Omega_\Tcal$. This illustrates that the behavior of $d$-dimensional SFTs when $d\geq2$
is much different than the one-dimensional case where emptiness of a SFT is
decidable \cite{MR1369092}.
Note that another important difference between $d=1$ and $d\geq2$
is expressed in terms of the possible values of entropy of $d$-dimensional SFTs,
see \cite{MR2680402}.

The goal of this section is to prove Theorem~\ref{thm:main-theoremA}, i.e.,
that the Wang shift $\Omega_\Zcal \subset \{0,\dots,15\}^{\Z^2}$
defined by the set of Wang tiles $\Zcal$ shown in Figure~\ref{fig:Z} is self-similar
where the self-similarity is given by 
the $2$-dimensional morphism $\Phi$ defined in 
Equation~\eqref{eq:definition-of-Phi}.

\begin{exo}[label={exo:valid-7x7}]
    Find a valid $7\times 7$ tiling with the set $\Zcal$ of Wang tiles.
\end{exo}

\subsection{Markers in the context of Wang tilings}

A tiling with the tiles from the set $\Zcal$ is shown in
Figure~\ref{fig:Z-tiling-5x10-markers}.
It illustrates that there exists a subset $M\subset\Zcal$ of tiles
such that each horizontal row of tiles in the tiling
is using either tiles from $M$ or from $\Zcal\setminus M$.
Moreover, the horizontal lines using tiles from $M$ are nonadjacent.
If these conditions are satisfied for all configurations in $\Omega_\Zcal$,
then $M$ is a subset of markers in the direction $\be_2$.

\begin{sagesilent}
tableZ5x10 = Phi([[12]], 5)
tableZ5x10 = [col[:10] for col in tableZ5x10[3:8]]
tilingZ5x10 = WangTiling(tableZ5x10, Z.tiles())
M = [0, 1, 2, 3, 4, 5, 6]
M_non = list(range(7,16))
position1 = tilingZ5x10.tile_positions(M)
position2 = tilingZ5x10.tile_positions(M_non)
extra_before = []
for (x,y) in position1:
    extra_before.append(r'\fill[yellow!40] ({},{}) '
                        'rectangle ({},{});'.format(x,y,x+1,y+1))
for (x,y) in position2:
    extra_before.append(r'\fill[green!40] ({},{}) '
                        'rectangle ({},{});'.format(x,y,x+1,y+1))
extra_before = '\n'.join(extra_before)
tilingZ5x10tikz = tilingZ5x10.tikz(extra_before=extra_before)
\end{sagesilent}

\begin{figure}[h]
\begin{center}
    \sageplot[][pdf]{tilingZ5x10tikz}
\end{center}
    \caption{
        A $5\times 10$ tiling with tiles from the set $\Zcal$. The tiles
        labeled from 0 to 6 (shown with yellow background) are marker tiles for
        the direction $\be_2$ in the Wang shift $\Omega_\Zcal$ since they
        always appear on nonadjacent rows.}
    \label{fig:Z-tiling-5x10-markers}
\end{figure}

In this section we propose an algorithm to find and prove that a subset of
tiles is a subset of markers in a Wang shift.
We use Lemma~\ref{lem:criteria-markers} which provides a way to prove that a
subset of Wang tiles is a subset of markers and searching for them. 
To use it in the context of Wang shifts and more generally in the context of
SFTs, we need the following definition.

\begin{definition}[\bf surrounding of radius $r$]
    Let $X=\SFT(\Fcal)\subset\Acal^{\Z^2}$ be a shift of finite type for some
    finite set $\Fcal$ of forbidden patterns. A $2$-dimensional word
$u\in\Acal^\bn$,
with $\bn=(n_1,n_2)\in\N^2$,
admits a \defn{surrounding of radius} $r\in\N$
if there exists $w\in\Acal^{\bn+2(r,r)}$ 
such that 
    $u$ occurs in $w$ at position $(r,r)$
    and
$w$ contains no occurrences of forbidden patterns from $\Fcal$.
\end{definition}

If a word admits a surrounding of radius $r\in\N$, it does not mean
it is in the language of the SFT. But if it admits no surrounding of radius $r$
for some $r\in\N$, then for sure it is not in the language of the SFT.
We state the following lemma in the context of Wang tiles.

\begin{lemma}\label{lem:surrounding}
    Let $\Tcal$ be a set of Wang tiles and $u\in\Tcal^\bn$ be a rectangular pattern
    seen as a $2$-dimensional word with $\bn=(n_1,n_2)\in\N^2$. If
    $u$ is allowed in $\Omega_\Tcal$, then for every $r\in\N$ the word $u$ has a
    $\Tcal$-surrounding of radius $r$.
\end{lemma}

Equivalently the lemma says that
if there exists $r\in\N$ such that $u$ has no $\Tcal$-surrounding of radius~$r$,
then~$u$ is forbidden in $\Omega_\Tcal$ and this is how we use
Lemma~\ref{lem:surrounding} to find markers.
We propose Algorithm~\ref{alg:find-markers} to compute markers from a Wang
tile set and a chosen surrounding radius so that the computation terminates.
If the algorithm finds nothing, then maybe there are no markers or maybe 
one should try again after increasing the surrounding radius.
We prove in the next lemma that if the output is nonempty, it contains a subset of markers.

\begin{algorithm}[h]
    \caption{Find markers.
        If no markers are found, one should try increasing the radius $r$.}
    \label{alg:find-markers}
  \begin{algorithmic}[1]
    \Require $\Tcal$ is a set of Wang tiles;
             $i\in\{1,2\}$ is a direction $\be_i$;
             $r\in\N$ is some radius.
      \Function{FindMarkers}{$\Tcal$, $i$, $r$}
        \State $j\gets 3-i$
        \State $D_j \gets \left\{(u,v)\in\Tcal^2\mid \text{ domino } u\odot^jv \text{
             admits a $\Tcal$-surrounding of radius $r$}\right\}$
        \State $U \gets \{\{u\}\mid u\in\Tcal\}$
            \Comment Suggestion: use a union-find data structure 
        \ForAll{$(u,v) \in D_j$}
            \State Merge the sets containing $u$ and $v$ in the partition $U$.
        \EndFor
        \State $D_i \gets \left\{(u,v)\in\Tcal^2\mid \text{ domino } u\odot^iv \text{
             admits a $\Tcal$-surrounding of radius $r$}\right\}$
         \State\Return $\{\text{set } M \text{ in the partition } U \mid
                               \left(M\times M\right) \cap D_i=\varnothing\}$
      \EndFunction
      \Ensure The output contains zero, one or more subsets of markers 
              in the direction $\be_i$.
  \end{algorithmic}
\end{algorithm}

\begin{lemma}\label{lem:algo1-works}
    If there exists $r\in\N$ and $i\in\{1,2\}$ such that the output of
      $\Call{FindMarkers}{\Tcal, i, r}$
      contains a set $M$, then $M\subset\Tcal$ is a subset of markers in the
      direction $\be_i$.
\end{lemma}

\begin{proof}
    Suppose that $i=2$, the case $i=1$ being similar.
    The output set $M$ is nonempty since it was created from the union of
    nonempty sets (see lines 4-6 in Algorithm~\ref{alg:find-markers}).
    Using Lemma~\ref{lem:surrounding},
    lines 3 to 6 imply that
        $M\odot^1 (\Tcal\setminus M)$ and
        $(\Tcal\setminus M) \odot^1 M$
    are forbidden in $\Omega_\Tcal$.
    The lines 7 and 8 imply that
        $M\odot^2 M$ is forbidden in $\Omega_\Tcal$.
    Then we deduce from Lemma~\ref{lem:criteria-markers} that
    $M\subset\Tcal$ is a subset of markers in the direction~$\be_2$.
\end{proof}

We believe that if a set of Wang tiles $\Tcal$ has a subset of markers in the
direction $\be_i$ then there exists a surrounding radius $r\in\N$ such that
$\Call{FindMarkers}{\Tcal, i, r}$ outputs this set of markers, so that
Lemma~\ref{lem:algo1-works} is in fact a \emph{if and only if} but we do not
provide a proof of that here.
The fact that there is no upper bound for the surrounding radius is related to
the undecidability of the domino problem.
In practice, in the study of Jeandel--Rao tilings done in
\cite{MR4226493}, a surrounding radius of 1, 2 or 3 was
enough.

\begin{exo}[label={exo:HVdominoes-in-OmegaU}]
    Using the set $\Zcal$ of Wang tiles defined in Figure~\ref{fig:Z},
    compute the sets of horizontal and vertical dominoes that
    admit a $\Zcal$-surrounding of radius 2:
    \begin{align*}
        D_1 &= \left\{(u,v)\in\Zcal^2\mid u\odot^1v \text{
             admits a $\Zcal$-surrounding of radius $2$}\right\},\\
        D_2 &= \left\{(u,v)\in\Zcal^2\mid u\odot^2v \text{
             admits a $\Zcal$-surrounding of radius $2$}\right\}.
    \end{align*}
    Compare them with the set $H$ and $V$ found in
    Exercise~\ref{exo:HV-dominoes}.
\end{exo}

\begin{sagesilent}
L = Z.tiles()
M = [0, 1, 2, 3, 4, 5, 6]
Markers_for_Zcal = WangTileSet([L[m] for m in M])
\end{sagesilent}

\begin{exo}[label={exo:markers-for-Zcal}]
    Use the function \textsc{FindMarkers} defined in
    Algorithm~\ref{alg:find-markers}
    with a surrounding radius $2$ to
    show that the subset
\[M=\left\{\raisebox{-4.5mm}{
    \sageplot[][pdf]{Markers_for_Zcal.tikz(ncolumns=8)}
}\right\}\]
of $\Zcal$ is a subset of markers for the direction $\be_2$ in $\Omega_\Zcal$.
\end{exo}

\subsection{Fusion of Wang tiles}

Recall that a \defn{magma} is a set $\Ical$
equipped with a binary operation $\bullet$
such that
for all $a,b\in\Ical$, the result of the operation $a\bullet b$ is also in $\Ical$.
If the operation $\bullet$ is associative and has an identity, then $\Ical$ is a monoid
and the operation $\bullet$ can be omitted and represented as concatenation.
But, 
in the general context of fusion of Wang tiles defined below where $\Ical$ is
the set of horizontal or vertical colors, we cannot assume that the operation
$\bullet$ is associative. 

The fusion operation on Wang tiles
is defined on pair of tiles
sharing an edge in a tiling according to 
Equations~\ref{eq:validwangtiling1} and~\ref{eq:validwangtiling2}.
Let $(\Ical,\bullet)$ and
$(\Jcal,\bullet)$ be two magmas and
let $\{A,C,Y,W\}\subset\Ical$ be some vertical colors and 
$\{B,D,X,Z\}\subset\Jcal$ be some horizontal colors.
We define two binary operations $\boxbar$ and $\boxminus$ on Wang tiles as
\begin{center}
    \raisebox{-4mm}{\includegraphics{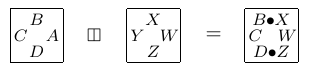}}
    if $A=Y$
\end{center}
and
\begin{center}
    \raisebox{-4mm}{\includegraphics{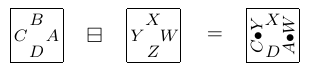}}
    if $B=Z$.
\end{center}
If $A\neq Y$, the operation $\boxbar$ is not defined.
Similarly, if $B\neq Z$, the operation $\boxminus$ is not defined.
For the Wang tiles considered in this contribution, the operation $\bullet$ is
associative so we always denote it implicitly by concatenation of colors.

In what follows, we propose algorithms and results that works for both
operations $\boxbar$ and $\boxminus$. 
It is thus desirable to have a common notation to denote both, so we define
\[
    u \boxslash^1 v = u \boxbar v
    \qquad
    \text{ and }
    \qquad
    u \boxslash^2 v = u \boxminus v.
\]
If $u\boxslash^i v$ is defined for some $i\in\{1,2\}$, it means that tiles $u$
and $v$ can appear at position $\bn$ and $\bn+\be_i$ in a tiling for some
$\bn\in\Z^d$. 
For each $i\in\{1,2\}$, one can define a new set of tiles from 
two sets $\Tcal$ and $\Scal$ of Wang tiles as
\begin{equation*}
    \Tcal\boxslash^i\Scal = \{u\boxslash^i v \text{ defined }\mid
                        u\in\Tcal,v\in\Scal\}.
\end{equation*}

\begin{exo}[label={exo:fusion-of-tiles}]
    Using the set $D_2$ of dominoes that admits a $\Zcal$-surrounding of
    radius~$2$ computed in Exercise~\ref{exo:HVdominoes-in-OmegaU} and the
    subset $M\subset\Zcal$ of markers for the direction $\be_2$ in
    $\Omega_\Zcal$ computed in Exercise~\ref{exo:markers-for-Zcal}, compute the
    set of fusion tiles:
    \[
        \{u\boxminus v \mid (u,v) \in D_2 \text{ and } v \in M\}.
    \]
    What is the meaning of this set?
\end{exo}

\subsection{Desubstitution of Wang shifts}

In Theorem~\ref{thm:if-markers-desubstitute}, we proved that the presence of markers allows
to desubstitute uniquely the configurations of a subshift on $\Z^d$. In case of Wang
shifts, we show in this section that the preimage is also a Wang shift and we may
construct the preimage set of Wang tiles using the fusion operation defined in
the previous section.
We also propose an algorithm to find the desubstitution of Wang
shifts when there exists a subset of marker tiles.

\begin{sagesilent}
M = [0, 1, 2, 3, 4, 5, 6]
V,alpha0 = Z.find_substitution(M, i=2, radius=2, side="right", solver="dancing_links")
tableV5x6 = [[11,3,17,12,4,11],
         [8,1,9,8,0,8],
         [14,4,12,13,5,14],
         [10,0,8,7,0,10],
         [16,5,13,15,5,16]]
assert alpha0(tableV5x6) == tableZ5x10
tilingV5x6 = WangTiling(tableV5x6, V.tiles())
M = [0, 1, 2, 3, 4, 5, 6]
M_non = list(range(7,18))
position1 = tilingV5x6.tile_positions(M)
position2 = tilingV5x6.tile_positions(M_non)
extra_before = []
for (x,y) in position1:
    extra_before.append(r'\fill[green!40] ({},{}) '
                        'rectangle ({},{});'.format(x,y,x+1,y+1))
for (x,y) in position2:
    extra_before.append(r'\fill[blue!40] ({},{}) '
                        'rectangle ({},{});'.format(x,y,x+1,y+1))
extra_before = '\n'.join(extra_before)
tilingV5x6tikz = tilingV5x6.tikz(extra_before=extra_before)
\end{sagesilent}

\begin{figure}[h]
\begin{center}
\begin{tikzpicture}
    \node (B) at (0,0) {\sageplot[][pdf]{tilingZ5x10tikz}};
    \node (C) at (6,0) {\sageplot[][pdf]{tilingV5x6tikz}};
    \draw[<-|] (B) to node[above] {$\alpha_0$} (C);
\end{tikzpicture}
\end{center}
    \caption{
        A $5\times 10$ tiling with tiles from the set $\Zcal$ is shown on the left.
            The tiles labeled from 0 to 6 (shown with yellow background) are
            marker tiles for the direction $\be_2$ since they appear on
            nonadjacent rows.
        It can be desubstituted as a $5\times 6$ pattern with tiles from the
        set $\Vcal$ using a substitution $\alpha_0$.
    Each marker tile (yellow background) is glued with its below tile (green background) to form a
    new Wang tile (blue background) using the fusion operation $\boxminus$.
    The remaining tiles are kept the same (green background) but get new indices in the set $\Vcal$.
    This process is uniquely defined since the substitution $\alpha_0$ is recognizable.}
    \label{fig:Z-tiling-desubstitution}
\end{figure}

Before stating the result, let us see how the markers allow to desubstitute tilings. 
In Figure~\ref{fig:Z-tiling-desubstitution}.
we observe that markers $M$ computed in Exercise~\ref{exo:markers-for-Zcal}
appear as nonadjacent rows in the Wang shift $\Omega_\Zcal$. 
Therefore the row above
(and below) a row of markers is made of nonmarker tiles. Let us consider the row below.
The idea is to collapse that row with the row of markers just above. 
Each tile is being collapsed with the above marker tile using the fusion of tiles.
The set of tiles that we obtain through this process is exactly the set computed in
Exercise~\ref{exo:fusion-of-tiles}.

Therefore to build a configuration in $\Omega_\Zcal$, it is sufficient to build a tiling
with another set $\Vcal$ of Wang tiles obtained from the set $\Zcal$ after removing the
markers and adding the tiles obtained from the fusion operation. 
One may also remove the tiles which always appear below of a marker tile.
One may recover some configuration in $\Omega_\Zcal$ by applying a $2$-dimensional morphism
$\alpha_0:\Vcal\to\Zcal^{*^2}$
which replaces the merged tiles by their associated equivalent vertical dominoes
and keeps the remaining tiles invariant, see Figure~\ref{fig:Z-tiling-desubstitution}. 
It turns out
that this decomposition is unique.
The creation of the set $\Vcal$ from $\Zcal$ gives the intuition on the
construction of Algorithm~\ref{alg:find-recognizable-sub-from-markers} which 
follows the same recipe and takes any set of Wang tiles with markers as input.

We now state the result that
if a set of Wang tiles $\Tcal$ has a subset of marker tiles, then
there exists another set $\Scal$ of Wang tiles and a 
nontrivial recognizable $2$-dimensional morphism
$\Omega_\Scal\to\Omega_\Tcal$ that is onto up to a shift.
Thus, every Wang
tiling by $\Tcal$ is up to a shift the image under a nontrivial $2$-dimensional
morphism $\omega$ of a unique Wang tiling in $\Omega_\Scal$.
The $2$-dimensional morphism is essentially $1$-dimensional as we show in the
next theorem.

\begin{theorem}\label{thm:if-markers-desubstitute-wang-version}
    {\rm\cite{MR3978536,MR4226493}}
    Let $\Tcal$ be a set of Wang tiles 
    and let $\Omega_\Tcal$ be its Wang shift.
    If there exists a subset
    $M\subset\Tcal$ 
    of markers in the direction 
    $\be_i$ for $i\in\{1,2\}$,
    then 
\begin{enumerate}
\item there exists
    a set of Wang tiles $\Scal_R$
    and a $2$-dimensional morphism
    $\omega_R:\Omega_{\Scal_R}\to\Omega_\Tcal$
    such that 
    \begin{equation*}
        \omega_R(\Scal_R)\subseteq (\Tcal\setminus M)\cup 
        \left((\Tcal\setminus M)\odot^i M\right)
    \end{equation*}
    which is recognizable and onto up to a shift and
\item
    there exists a set of Wang tiles $\Scal_L$ and a $2$-dimensional morphism
    $\omega_L:\Omega_{\Scal_L}\to\Omega_\Tcal$
    such that 
    \begin{equation*}
        \omega_L(\Scal_L)\subseteq (\Tcal\setminus M)\cup 
        \left(M\odot^i (\Tcal\setminus M)\right)
    \end{equation*}
    which recognizable and onto up to a shift.
\end{enumerate}
    There exists a surrounding radius $r\in\N$ such that $\omega_R$ and
    $\omega_L$ are computed using
    Algorithm~\ref{alg:find-recognizable-sub-from-markers}.
\end{theorem}


\begin{proof}
The existence of the recognizable $2$-dimensional morphism
which is onto up to a shift was done in Theorem~\ref{thm:if-markers-desubstitute}.
We only need to prove that the preimage of $\Omega_\Tcal$ is a Wang shift. The proof
    of this fact can be found in \cite{MR3978536} and
    \cite{MR4226493}.
It follows the line of Algorithm~\ref{alg:find-recognizable-sub-from-markers}.
\end{proof}

\begin{algorithm}
    \caption{Find a recognizable desubstitution of $\Omega_\Tcal$ from markers}
    \label{alg:find-recognizable-sub-from-markers}
  \begin{algorithmic}[1]
          \Require $\Tcal$ is a set of Wang tiles;
                   $M\subset\Tcal$ is a subset of markers;
                   $i\in\{1,2\}$ is a direction $e_i$;
                   $r\in\N$ is a surrounding radius;
                   $s\in\{\scleft,\scright\}$ determines
                   whether the image of merged tiles is 
                   in $M\odot^i (\Tcal\setminus M)$ (markers on the left)
                   or in $(\Tcal\setminus M)\odot^i M$ (markers on the right).
      \Function{FindSubstitution}{$\Tcal$, $M$, $i$, $r$, $s$}
        \State $D \gets \left\{(u,v)\in\Tcal^2\mid \text{ domino } u\odot^iv \text{
             admits a $\Tcal$-surrounding of radius $r$}\right\}$
        \If{$s=\scleft$}
            \State $ P \gets \left\{(u,v)\in D
                \mid u\in M \text{ and } v\in\Tcal\setminus M \right\}$
            \State $K \gets \left\{v\in\Tcal\setminus M\mid \text{ there exists }
                            u \in \Tcal\setminus M \text{ such that }
                            (u,v) \in D \right\}$
        \ElsIf{$s=\scright$}
            \State $ P \gets \left\{(u,v)\in D
                \mid u\in\Tcal\setminus M \text{ and } v\in M \right\}$
            \State $K \gets \left\{u\in\Tcal\setminus M\mid \text{ there exists }
                            v \in \Tcal\setminus M \text{ such that }
                            (u,v) \in D \right\}$
        \EndIf
        \State $K\gets\Call{Sort}{K}$, $P\gets\Call{Sort}{P}$
        \Comment lexicographically on the indices of tiles
        \State $\Scal\gets K\cup \{u\boxslash^iv\mid(u,v)\in P\}$
        \Comment defines uniquely indices of tiles in $\Scal$ from $0$ to
        $|\Scal|-1$.
        \State \Return $\Scal$, $\omega:\Omega_\Scal\to\Omega_\Tcal:
        \begin{cases}
        u\boxslash^i v
        \quad \mapsto \quad
        u\odot^i v
        & \text{ if }
        (u, v) \in P\\
        u \quad \mapsto \quad
        u & \text{ if } u \in K.
        \end{cases}$
      \EndFunction
      \Ensure $\Scal$ is a set of Wang tiles;
              $\omega:\Omega_\Scal\to\Omega_\Tcal$ is 
        recognizable and onto up to a shift.
  \end{algorithmic}
\end{algorithm}

In the definition of $\omega$
in Algorithm~\ref{alg:find-recognizable-sub-from-markers},
given two Wang tiles $u$ and $v$ such that
$u\boxslash^i v$ is defined for $i\in\{1,2\}$, the map
\begin{equation*}
u\boxslash^i v
\quad
\mapsto
\quad
u\odot^i v
\end{equation*}
can be seen as a decomposition of Wang tiles:
\begin{center}
    \includegraphics{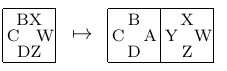}
\qquad
    \raisebox{5mm}{\text{ or }}
\qquad
    \includegraphics{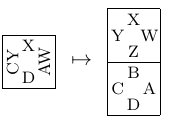}
\end{center}
whether $i=1$ or $i=2$ and where $A=Y$ and $B=Z$.
The reader may wonder how the substitution decides the color $A$ (color $B$ if
$i=2$) from its input tiles.  The answer is that
Algorithm~\ref{alg:find-recognizable-sub-from-markers} is performing
a desubstitution. Therefore the two tiles sharing the vertical side with
letter $A$ are known from the start and the algorithm just creates a new tile
$(W,BX,C,DZ)$ and claims that it will always get replaced by the two tiles with
shared edge with color $A$.

\begin{exo}[label={exo:FindSubstitution}]
    Using the function \textsc{FindSubstitution} defined in
Algorithm~\ref{alg:find-recognizable-sub-from-markers}
with the subset $M\subset\Zcal$ of markers for the direction $\be_2$ computed in
    Exercise~\ref{exo:markers-for-Zcal}, 
    construct a set of tiles $\Vcal$ and
    a recognizable $2$-dimensional morphism
$\alpha_0:\Omega_\Vcal\to\Omega_\Zcal$ which is onto up to a shift
    and such that 
    \[
    \alpha_0(u)\in\Zcal\setminus M
    \quad
    \text{ or }
    \quad
    \alpha_0(u)\in(\Zcal\setminus M)\odot^2 M.
    \]
\end{exo}

\subsection{Self-similarity of the Wang shift $\Omega_\Zcal$ defined by 16 tiles}
\label{subsection:self-similarity-wang-shift}

We prove that $\Omega_\Zcal$ is self-similar
by executing the function
$\textsc{FindMarkers}$ 
on the set of Wang tiles $\Zcal$ 
followed by 
$\textsc{FindSubstitution}$
and repeating this process until we obtain a set of Wang tiles which is equivalent to the original one (two steps are needed).
Each time $\textsc{FindMarkers}$ finds at least one subset of markers using a
surrounding radius of size at most 2.
Thus using Theorem~\ref{thm:if-markers-desubstitute-wang-version},
we find a desubstitution of tilings
with $\textsc{FindSubstitution}$.
The proof is done in SageMath \cite{sagemathv10.4}
using \texttt{slabbe} optional package
\cite{labbe_slabbe_0_7_6_2023}.

The following result was first shown in 
\cite{lepsova_thesis_2024} where it was also proved that
$\Omega_\Zcal$ is aperiodic and minimal.
It was deduced by showing that $\Omega_\Zcal$ is topologically
conjugate to $\Omega_\Ucal$.
Self-similarity, aperiodicity and minimality of $\Omega_\Ucal$
was shown in \cite{MR3978536}.
In \cite{MR4226493}, it was shown that the set
$\Omega_\Ucal$ describes the internal self-similar structure hidden in
Jeandel--Rao aperiodic tilings \cite{MR4210631}.

A consequence of Theorem~\ref{thm:main-theoremA} is that the Wang
shift $\Omega_\Zcal$ provides another description for the substitutive subshift
$\Xcal_\Phi$, see Exercise~\ref{exo:OmegaU=XPhi}.

\begin{proof}[Proof of Theorem~\ref{thm:main-theoremA}]
    In this proof, we show there exist sets of Wang tiles 
$\Vcal$ and $\Wcal$
together with their associated Wang shifts 
$\Omega_\Vcal$ and $\Omega_\Wcal$
and there exist two recognizable $2$-dimensional morphisms $\alpha_0$ and $\alpha_1$
and a bijection $\alpha_2$:
    \begin{equation*}
        \Omega_\Zcal \xleftarrow{\alpha_0}
        \Omega_\Vcal \xleftarrow{\alpha_1}
        \Omega_\Wcal \xleftarrow{\alpha_2}
        \Omega_\Zcal
    \end{equation*}
    that are onto up to a shift, i.e.,
    $\shiftclosure{\alpha_0(\Omega_\Vcal)}=\Omega_\Zcal$,
    $\shiftclosure{\alpha_1(\Omega_\Wcal)}=\Omega_\Vcal$ and
    $\alpha_2(\Omega_\Zcal)=\Omega_\Wcal$.

    First we define the set $\Zcal$ of Wang tiles in SageMath:
\begin{sagecommandlinetcb}
\begin{sagecommandline}
sage: from slabbe import WangTileSet
sage: tiles = ["DOJO", "DOHL", "JMDP", "DMDK", "HPJP", "HPHN", "HKDP", "BOIO", 
....:   "ILEO", "ILCL", "ALIO", "EPIP", "IPIK", "IKBM", "IKAK", "CNIP"]
sage: Z = WangTileSet([tuple(tile) for tile in tiles])
\end{sagecommandline}
\end{sagecommandlinetcb}
\[
\Zcal = \left\{\raisebox{-10mm}{
    \sageplot[][pdf]{Z.tikz(ncolumns=8)}
}\right\}
\]

    We desubstitute $\Zcal$ with the set $\{0, 1, 2, 3, 4, 5, 6\}$
    of markers in the direction $\be_2$:
\begin{sagecommandlinetcb}
\begin{sagecommandline}
sage: Z.find_markers(i=2,radius=2,solver="dancing_links")
[[0, 1, 2, 3, 4, 5, 6]]
sage: M = [0, 1, 2, 3, 4, 5, 6]
sage: V,alpha0 = Z.find_substitution(M, i=2, radius=2, side="right",
....:                                solver="dancing_links")
\end{sagecommandline}
\end{sagecommandlinetcb}
    \noindent
We obtain $\alpha_0:\Omega_\Vcal\to\Omega_\Zcal$ given as a rule of the form
\[
\begin{array}{ll}
    \alpha_0:&\Zrange{17}\to\Zrange{15}^{*^2}\\[2mm]
    &\left\{\arraycolsep=1.8pt
                \sage{alpha0._latex_(ncolumns=5, align='l')}
    \right.
\end{array}
\]
and the set $\Vcal$ of $\sage{len(V)}$ Wang tiles
\[
\Vcal = \left\{\raisebox{-15mm}{
    \sageplot[][pdf]{V.tikz(ncolumns=8)}
}\right\}.
\]

The set of tiles $\Vcal$ has three subsets of markers for the direction
    $\be_1$.
We desubstitute $\Vcal$ with the subset of markers
$\{0, 1, 2, 8, 9, 10, 11\}$:
\begin{sagecommandlinetcb}
\begin{sagecommandline}
sage: V.find_markers(i=1,radius=1,solver="dancing_links")
[[0, 1, 2, 7, 8, 9, 10]]
sage: M = [0, 1, 2, 7, 8, 9, 10]
sage: W,alpha1 = V.find_substitution(M, i=1, radius=1, side="right",
....:                                solver="dancing_links")
\end{sagecommandline}
\end{sagecommandlinetcb}
    \noindent
We obtain $\alpha_1:\Omega_\Wcal\to\Omega_\Vcal$ given as a rule of the form
\[
\begin{array}{ll}
    \alpha_1:&\Zrange{15}\to\Zrange{17}^{*^2}\\[2mm]
    &\left\{\arraycolsep=1.8pt
            \sage{alpha1._latex_(ncolumns=4, align='l')}
     \right.
\end{array}
\]
and the set $\Wcal$ of 16 Wang tiles
\[
\Wcal = \left\{\raisebox{-10mm}{
    \sageplot[][pdf]{W.tikz(ncolumns=8)}
}\right\}.
\]

It turns out that $\Zcal$ and $\Wcal$ are \defn{equivalent},
that is, they are the same set of Wang tiles up to 
a bijection of their horizontal and vertical edge labels.
This can be checked in SageMath as follows:
\begin{sagecommandlinetcb}
\begin{sagecommandline}
sage: W.is_equivalent(Z)
True
\end{sagecommandline}
\end{sagecommandlinetcb}
    \noindent
The bijection \texttt{vert} between the vertical colors,
the bijection \texttt{horiz} between the horizontal colors
and bijection $\alpha_2$ from $\Zcal$ to $\Wcal$ is computed as follows:
\begin{sagecommandlinetcb}
\begin{sagecommandline}
sage: _,vert,horiz,alpha2 = Z.is_equivalent(W, certificate=True)
sage: vert
{'A': 'IJ', 'B': 'IH', 'C': 'BD', 'D': 'I', 'E': 'AD', 'H': 'B', 'I': 'ID', 'J': 'A'}
sage: horiz
{'K': 'PO', 'L': 'M', 'M': 'PL', 'N': 'MO', 'O': 'K', 'P': 'KO'}
\end{sagecommandline}
\end{sagecommandlinetcb}
    \noindent
The equivalence of two sets of Wang tiles is decided by computing a graph
isomorphism between the representation of a set of Wang tiles as a graph where
the edges link the left to the right colors of each tile.
The curious reader may discover the algorithm by reading the source code 
of the above method using two question marks in SageMath
(\texttt{Z.is\_equivalent??}).

We obtain the morphism $\alpha_2:\Omega_\Zcal\to\Omega_\Wcal$ given as a rule of the form
\[
\begin{array}{ll}
    \alpha_2:&\Zrange{15}\to\Zrange{15}^{*^2}\\[2mm]
    &\left\{\arraycolsep=1.8pt
                \sage{alpha2._latex_(ncolumns=4, align='l')}
    \right.
\end{array}
\]

We may check that $\alpha_0\circ\alpha_1\circ\alpha_2=\Phi$
where the variable \texttt{Phi} was created in Exercise~\ref{exo:Phi-sagemath}:
\begin{sagecommandlinetcb}
\begin{sagecommandline}
sage: alpha0 * alpha1 * alpha2 == Phi
True
\end{sagecommandline}
\end{sagecommandlinetcb}
\noindent
We conclude that
$\Omega_\Zcal
    =\shiftclosure{\alpha_0(\Omega_\Vcal)}
    =\shiftclosure{\alpha_0\alpha_1(\Omega_\Wcal)}
    =\shiftclosure{\alpha_0\alpha_1\alpha_2(\Omega_\Zcal)}
    =\shiftclosure{\Phi(\Omega_\Zcal)}$.
\end{proof}

In the proof, we used Knuth's dancing links algorithm \cite{knuth_dancing_2000}
    because it is faster at this particular task than the MILP solver Gurobi
    \cite{gurobi} or the SAT solvers Glucose
    \cite{doi:10.1142/S0218213018400018} as we can see below:
\begin{sageverbatimtcb}
sage: %time Z.find_markers(i=2,radius=2,solver="dancing_links")
CPU times: user 3.34 s, sys: 0 ns, total: 3.34 s
Wall time: 3.34 s
[[0, 1, 2, 3, 4, 5, 6]]
sage: %time Z.find_markers(i=2,radius=2,solver="gurobi")
CPU times: user 12.4 s, sys: 572 ms, total: 13 s
Wall time: 13 s
[[0, 1, 2, 3, 4, 5, 6]]
sage: %time Z.find_markers(i=2,radius=2,solver="glucose")
CPU times: user 50.6 s, sys: 2.53 s, total: 53.1 s
Wall time: 2min 10s
[[0, 1, 2, 3, 4, 5, 6]]
\end{sageverbatimtcb}
    Note that for other tasks like finding a valid tiling of a $n\times n$
    square with Wang tiles, the Glucose SAT
	solver \cite{doi:10.1142/S0218213018400018} based on MiniSAT \cite{minisat}
    is faster \cite{labbe_comparison_2018}
    than Knuth's dancing links algorithm or MILP solvers.
  
\begin{exo}[label={exo:OmegaU=XPhi}]
        Using the criterion given in Lemma~\ref{lem:criterion-for-minimality},
        prove that $\Omega_\Zcal$ is minimal and $\Omega_\Zcal=\Xcal_\Phi$.
\end{exo}

\begin{exo}[label={exo:markers-e1-and-then-e2}]
    Prove the self-similarity of $\Omega_\Zcal$ by using first markers in the
    direction $\be_1$ in the set $\Zcal$ and then using markers in the
    direction $\be_2$.
\end{exo}

\section{A self-similar symbolic dynamical system}\label{chap:Labbe:sec:Z2-rotations}

In this section, we consider the dynamical system $(\torus^2,\Z^2,R_\Zcal)$ defined
on the 2-dimensional torus $\torus^2=(\R/\Z)^2$
by the continuous $\Z^2$-action
\[
\begin{array}{rccl}
    R_\Zcal:&\Z^2\times\torus^2 & \to & \torus^2\\
    &(\bn,\bx) & \mapsto &\bx+\varphi^{-2}\bn
\end{array}
\]
where $\varphi=\frac{1+\sqrt{5}}{2}$.
We define a symbolic representation of that dynamical system
using a well-chosen partition $\Pcal_\Zcal$ of $\torus^2$.
The partition $\Pcal_\Zcal$ is
is a simplification  of the partition $\Pcal_\Ucal$
that was introduced in \cite{MR4213162}
where it was proved to be a Markov partition 
for the dynamical system $(\torus^2,\Z^2,R_\Zcal)$.
As discovered during the PhD thesis of Jana Lepšová, 
it turns out that the vertical line at $x=\varphi^-2$ is not necessary
in the partition $\Pcal_\Ucal$. 
Removing the vertical line at $x=\varphi^{-2}$ in the partition $\Pcal_\Ucal$
reduces the number of atoms in the partition from 19 to 16.
The indices used to define the partitions are consistent with the choices made in
\cite{MR4213162} for $\Ucal$ and \cite{lepsova_thesis_2024} for $\Zcal$.
As illustrated in Figure~\ref{fig:construction-of-partitionZ},
the partition $\Pcal_\Zcal$ can be defined from the following 7 segments in $\R^2$:
\begin{align*}
    &(1,\varphi^2) \to (0,\varphi^2) \to (\varphi,0) \to (\varphi,1),\\
    &(1,1)\to (0,1) \to (1,0) \to (1,1),\\
    &\left(\frac{1}{\varphi^2},2\right)
\to\left(1+\frac{1}{\varphi^2},1\right).
\end{align*}

\begin{figure}[h]
\begin{center}
    \includegraphics{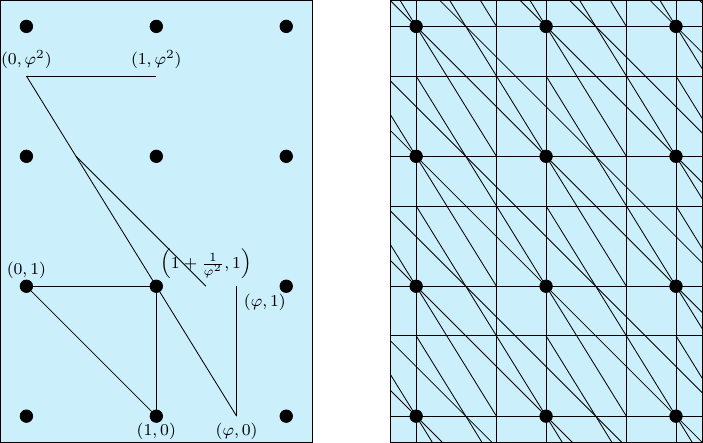}
\end{center}
    \caption{
        The partition $\Pcal_\Zcal$ of $\torus^2$ can be constructed from 7
        segments in $\R^2$ (left) and their images under the group of
        translations $\Z^2$ (right).}
    \label{fig:construction-of-partitionZ}
\end{figure}

\begin{figure}[h]
\begin{center}
    \includegraphics[width=\linewidth]{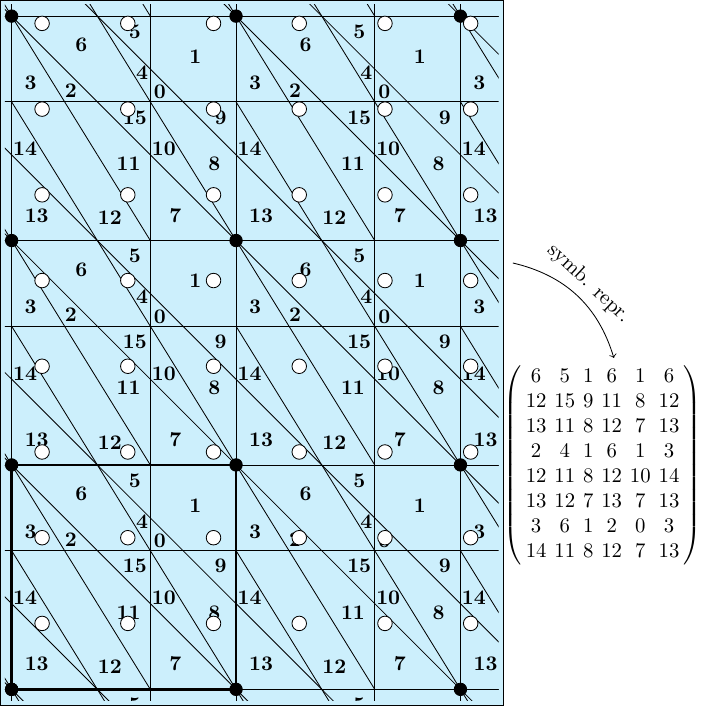}
\end{center}
    \caption{
    The polygonal partition $\Pcal_\Zcal$ of $\torus^2$ 
    with indices in the set $\Zrange{15}$.
    The coding of the shifted lattice $\bp+\frac{1}{\varphi^2}\Z^2$
    by the partition
    defines a configuration in $\Zrange{15}^{\Z^2}$. In the figure, the points
    $\bp+\frac{1}{\varphi^2}(m,n)$ are shown in white with $\bp=(0.1357, 0.2938)$ for each
    $m\in\Zrange{5}$ and $n\in\Zrange{7}$ and
    are coded by a $2$-dimensional word of shape $(6,8)$.}
    \label{fig:walk-on-partition}
\end{figure}

The translations of the 7 segments under the group of translation $\Z^2$
splits the torus $\torus^2$ into 16 polygonal regions indexed with integers
from the set $\Zrange{15}$.
The coding by the partition
$\Pcal_\Zcal$ of the orbit of a starting point in $\torus^2$ by the
$\Z^2$-action of $R_\Zcal$ 
defines a configuration $w\in\Zrange{15}^{\Z^2}$,
see Figure~\ref{fig:walk-on-partition}.
The topological closure of the set of all such configurations
is the symbolic dynamical system $\Xcal_{\Pcal_\Zcal,R_\Zcal}$ corresponding
to $\Pcal_\Zcal,R_\Zcal$ (see Lemma~\ref{lem:closure-of-tilings}).
It turns out that $\Xcal_{\Pcal_\Zcal,R_\Zcal}$ is a \defn{subshift} as it is
also closed under the shift $\sigma$ by integer translations.

The goal of the next sections is to prove 
that the symbolic
dynamical system $\Xcal_{\Pcal_\Zcal,R_\Zcal}$ is self-similar
where the self-similarity is given by 
the $2$-dimensional morphism $\Phi$ defined in 
Equation~\eqref{eq:definition-of-Phi}
(Theorem~\ref{thm:main-theoremB}).

\begin{exo}[label={exo:merge-the-atoms}]
    Define the polygonal partition $\Pcal_\Zcal$ of $\torus^2$ 
    using the simplification proposed by Jana Lepšová
    \cite{lepsova_thesis_2024}
    of the partition $\Pcal_\Ucal$ 
    that was introduced in \cite{MR4213162},
    that is, 
    merging the atoms labeled 6 and 7,
    merging the atoms labeled 12 and 13 and
    merging the atoms labeled 14 and 15.
    See Figure~\ref{fig:partition-PU-PZ}.
\end{exo}

\begin{sagesilent}
axis0 = axis_HV()
PUtikz = PU.tikz(extra_code=axis0, fontsize=r'\normalsize', scale=5)
\end{sagesilent}

\begin{figure}[h]
    \[
    \begin{array}{c}
        \Pcal_\Ucal=
    \end{array}
    \begin{array}{c} 
        \sageplot[][pdf]{PUtikz}
    \end{array}
    \quad
    \begin{array}{c}
        \Pcal_\Zcal=
    \end{array}
    \begin{array}{c} 
        \sageplot[][pdf]{PZtikz}
    \end{array}
    \] 
\caption{The partition $\Pcal_\Ucal$
    introduced in \cite{MR4213162}
    and the partition $\Pcal_\Zcal$
    following the simplification suggested in \cite{lepsova_thesis_2024}.}
    \label{fig:partition-PU-PZ}
\end{figure}

\begin{exo}[label={exo:word-shape-68-8-10}]
    Figure~\ref{fig:walk-on-partition} provides the construction of a 
    2-dimensional word of shape $(6,8)$.
    Using the same construction, extend that pattern by one unit in all
    directions to obtain a word of shape $(8,10)$.
\end{exo}

\subsection{Toral $\Z^2$-rotations and polygon exchange transformations (PETs)}

Let $\Gamma$ be a \defn{lattice} in $\R^2$, i.e., a discrete subgroup of the
additive group $\R^2$ with $2$ linearly independent generators.
This defines a $2$-dimensional torus $\generictorus=\R^2/\Gamma$. 
By analogy with the rotation $x\mapsto x+\alpha$ on the circle $\R/\Z$ for
an $\alpha\in\R$, we use the terminology of \defn{rotation} 
(sometimes also called \defn{translation})
to denote the following $\Z^2$-action defined on a 2-dimensional torus.

\begin{definition}\label{def:Z2-rotation}
Let $\generictorus=\R^2/\Gamma$ where $\Gamma$ is a lattice in $\R^2$.
For some $\balpha,\bbeta\in\generictorus$, we consider
the dynamical system $(\generictorus, \Z^2, R)$ where
$R:\Z^2\times\generictorus\to\generictorus$ 
is the continuous $\Z^2$-action on $\generictorus$
defined by
\[
R^\bn(\bx):=R(\bn,\bx)=\bx + n_1\balpha + n_2\bbeta
\]
for every $\bn=(n_1,n_2)\in\Z^2$.
We say that the $\Z^2$-action $R$ is a 
    \defn{toral $\Z^2$-rotation} or a
    \defn{$\Z^2$-rotation on} $\generictorus$
    which defines a dynamical system $(\generictorus,\Z^2,R)$.
\end{definition}

It is practical to represent a toral $\Z^2$-rotation in terms of polygon
exchange transformations \cite{MR3010377,MR3186232}.

\begin{definition}{\rm\cite{alevy_kenyon_yi}}
Let $X$ be a polygon together with
two topological partitions of $X$ into polygons
\[
    X=\bigcup_{k=0}^N P_k
     =\bigcup_{k=0}^N Q_k
\] 
such that for each $k$, $P_k$ and $Q_k$ are translation equivalent, i.e.,
there exists $\bv_k\in\R^2$ such that $P_k=Q_k+\bv_k$.
A \defn{polygon exchange transformation} (PET) is the piecewise translation
on $X$ defined for $\bx\in P_k$ by
$T(\bx) = \bx+\bv_k$.
The map is not defined for points $\bx\in\bigcup_{k=0}^N\partial P_k$.
\end{definition}

\begin{sagesilent}
from slabbe import PolyhedronExchangeTransformation as PET
base = diagonal_matrix((1,1))
translation = vector((1/3, 1/2))
T = PET.toral_translation(base, translation)
s,beta = T.induced_transformation([1,-1,-1])
ticks_at_6 = [a/6 for a in range(7)]
axis01 = axis_HV(horizontal_axis=ticks_at_6,vertical_axis=ticks_at_6)
PETtriangledomaintikz = s.partition().tikz(extra_code=axis01, label_format='$P_{{{}}}$', fontsize=r'\normalsize', scale=5)
PETtriangleimagetikz = s.image_partition().tikz(extra_code=axis01, label_format='$Q_{{{}}}$', fontsize=r'\normalsize', scale=5)
\end{sagesilent}

\begin{figure}[h]
\begin{center}
    \sageplot[][pdf]{PETtriangledomaintikz}
    \qquad
    \sageplot[][pdf]{PETtriangleimagetikz}
\end{center}
\caption{A polygon exchange transformation defined on the triangle with
    vertices $(0,0)$, $(0,1)$ and $(1,0)$.}
    \label{fig:PET-triangle}
\end{figure}

A PETs can be quite complicated.
For example, a polygon exchange transformation is shown in
Figure~\ref{fig:PET-triangle}. In this chapter, we consider pairs of commuting 
PETs that are much simpler given by the exchanges of two rectangles, see 
Figure~\ref{fig:PET-golden-rotation-RZ}.

\begin{figure}[h]
\begin{center}
    \includegraphics{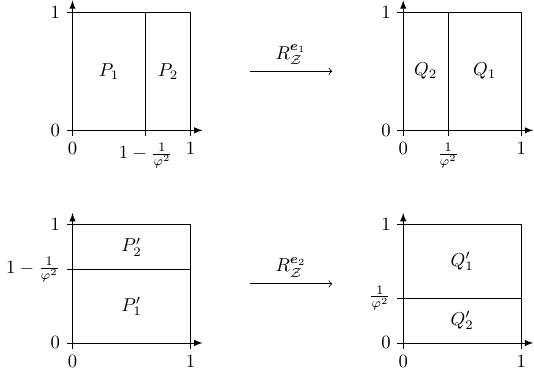}
\end{center}
\caption{The $\Z^2$-action $R_\Zcal$ can be seen as a pair of commuting
    polygon exchange transformations on the unit square $[0,1)^2$.}
    \label{fig:PET-golden-rotation-RZ}
\end{figure}

A $\Z^2$-rotation $R$ 
on a torus $\generictorus$
can be decomposed into two commuting maps
as follows:
\[
    R^\bn
    =R^{n_1\be_1+n_2\be_2}
    =R^{n_1\be_1} \circ R^{n_2\be_2}
    =\left(R^{\be_1}\right)^{n_1} \circ 
     \left(R^{\be_2}\right)^{n_2}.
\]
where each map $R^{\be_1}$ and $R^{\be_2}$ can be seen as polygon exchange
transformation 
defined by the exchange of at most $4$ pieces 
on a fundamental domain of~$\generictorus$
having for shape a parallelogram.
We state this as a lemma because we use this connection two times in the proof
of Theorem~\ref{thm:main-theoremB}.

\begin{lemma}\label{lem:PET-iff-toral-rotation}
    Let $\Gamma=\ell_1\Z\times\ell_2\Z$ be a lattice in $\R^2$
    and its rectangular fundamental domain
    $D=[0,\ell_1)\times[0,\ell_2)$.
    For every $\balpha=(\alpha_1,\alpha_2)\in D$,
    the dynamical system $(\R^2/\Gamma, \Z, \bx\mapsto\bx+\balpha)$
    is measurably conjugate to
    the dynamical system $(D, \Z, T)$ where $T:D\to D$
    is the polygon exchange transformation shown
    in Figure~\ref{fig:PET-toral-rotation}.
\end{lemma}

\begin{figure}[h]
\begin{center}
    \includegraphics{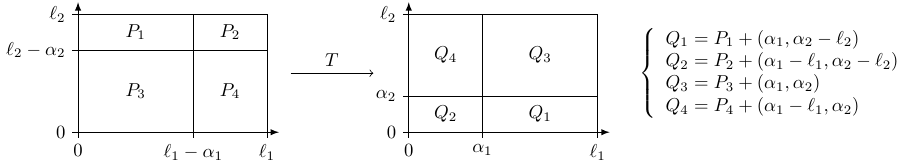}
\end{center}
\caption{The polygon exchange transformation $T$ of the rectangle
    $[0,\ell_1)\times[0,\ell_2)$ as defined on the figure can be seen as a
    toral rotation by the vector $(\alpha_1,\alpha_2)$ on the torus
    $\R^2/(\ell_1\Z\times\ell_2\Z)$.}
    \label{fig:PET-toral-rotation}
\end{figure}

\begin{proof}
    It follows from the fact that toral rotations and such polygon exchange
    transformations are the Cartesian product of circle rotations and exchange
    of two intervals. The fact that a rotation on a circle can be seen as an
    exchange of two intervals is well-known as noticed for example in
    \cite{MR551341}.
\end{proof}

\begin{exo}[label={exo:PETs-easy}]
    Recall that $R_\Zcal^\bn(\bx) = \bx+\varphi^{-2}\bn$
    is a $\Z^2$-action defined on $\torus^2$ where
    $\varphi=\frac{1+\sqrt{5}}{2}$.
    Prove that the maps $R_\Zcal^{\be_1}$
    and $R_\Zcal^{\be_2}$ 
    can be expressed as polygon exchange transformations
    on the unit square $[0,1)\times[0,1)$
    as in Figure~\ref{fig:PET-golden-rotation-RZ}.
\end{exo}

\subsection{Symbolic dynamical systems for toral $\Z^2$-rotations}

Let $\Gamma$ be a lattice in $\R^2$ and
$\generictorus=\R^2/\Gamma$ be a $2$-dimensional torus.
Let $(\generictorus,\Z^2,R)$ be the dynamical system 
given by a $\Z^2$-rotation $R$ on $\generictorus$.
For some finite set $\Acal$,
a \defn{topological partition} of $\generictorus$ is a finite
collection $\{P_a\}_{a\in\Acal}$ of disjoint open sets $P_a\subset\generictorus$
such that 
    $\generictorus = \bigcup_{a\in\Acal} \overline{P_a}$.
If $S\subset\Z^2$ is a finite set,
we say that a pattern $w\in\Acal^S$
is \defn{allowed} for $\Pcal,R$ if
\begin{equation}\label{eq:allowed-if-nonempty}
    \bigcap_{\bk\in S} R^{-\bk}(P_{w_\bk}) \neq \varnothing.
\end{equation}
Let $\Lcal_{\Pcal,R}$ be the collection of all allowed patterns for $\Pcal,R$.
The set $\Lcal_{\Pcal,R}$ is the language of a subshift 
$\Xcal_{\Pcal,R}\subseteq\Acal^{\Z^2}$ defined as follows,
see \cite[Prop.~9.2.4]{MR3525488},
\[
    \Xcal_{\Pcal,R} = 
    \{x\in\Acal^{\Z^2} \mid \pi_S\circ\sigma^\bn(x)\in\Lcal_{\Pcal,R}
    \text{ for every } \bn\in\Z^2 \text{ and finite subset } S\subset\Z^2\}.
\]

\begin{definition}
We call $\Xcal_{\Pcal,R}$ the \defn{symbolic dynamical
system} corresponding to $\Pcal,R$.
\end{definition}

For each $w\in\Xcal_{\Pcal,R}\subset\Acal^{\Z^2}$ and $n\geq 0$ there is a corresponding nonempty open set
\[
    D_n(w) = \bigcap_{\Vert\bk\Vert\leq n} R^{-\bk}(P_{w_\bk}) \subseteq \generictorus.
\]
The closures $\overline{D}_n(w)$ of these sets are compact
and decrease with $n$, so that
$\overline{D}_0(w)\supseteq
\overline{D}_1(w)\supseteq
\overline{D}_2(w)\supseteq
\dots$.
It follows that $\cap_{n=0}^{\infty}\overline{D}_n(w)\neq\varnothing$.
In order for points in
$\Xcal_{\Pcal,R}$
to correspond to points in $\generictorus$, this intersection should contain only one point.
This leads to the following definition.

\begin{definition}
A topological partition $\Pcal$ of $\generictorus$ \defn{gives a symbolic representation}
of $(\generictorus,\Z^2,R)$ if for every $w\in\Xcal_{\Pcal,R}$ the intersection
$\cap_{n=0}^{\infty}\overline{D}_n(w)$ consists of exactly one
point $\bx\in\generictorus$.
We call $w$ a \defn{symbolic representation of $\bx$}.
\end{definition}

In general, the existence of an atom of the partition of the torus $\generictorus$ which is invariant only under the trivial
translation is a sufficient condition for
the partition to give a symbolic representation of a minimal $\Z^2$-rotation on $\generictorus$.

\begin{lemma}\label{lem:symbolic-representation}
    {\rm\cite[Lemma~3.4]{MR4213162}}
    Let $(\generictorus,\Z^2,R)$ be a minimal dynamical system
    and $\Pcal=\{P_0,P_1,...,P_{r-1}\}$ be a topological partition
    of $\generictorus$.
    If there exists an atom $P_i$ which is invariant only under the trivial
    translation in $\generictorus$,
    then $\Pcal$ gives a symbolic representation of $(\generictorus,\Z^2,R)$.
\end{lemma}

\begin{proof}
    Let $\Acal=\{0,1,\dots,r-1\}$.
    Let $w\in\Xcal_{\Pcal,R}\subset\Acal^{\Z^2}$.
    As already noticed,
    the closures $\overline{D}_n(w)$ are compact
    and decrease with $n$, so that
    $\overline{D}_0(w)\supseteq
    \overline{D}_1(w)\supseteq
    \overline{D}_2(w)\supseteq
    \dots$.
    It follows that $\cap_{n=0}^{\infty}\overline{D}_n(w)\neq\varnothing$.

    We show that $\cap_{n=0}^{\infty}\overline{D}_n(w)$
    contains at most one element.
    Let $\bx,\by\in\generictorus$.
    We assume $\bx\in \cap_{n=0}^{\infty}\overline{D}_n(w)$
    and we want to show that $\by\notin \cap_{n=0}^{\infty}\overline{D}_n(w)$
    if $\bx\neq\by$.
    Let $P_i\subset\generictorus$
    for some $i\in\Acal$ be an atom which is invariant only under the
    trivial translation. 
    Since $\bx\neq\by$,
    $\overline{P_i}\setminus (\overline{P_i}-(\by-\bx))$ contains an open set $O$.
    Since $(\generictorus,\Z^2,R)$ is minimal,
    any orbit $\{R^\bk\bx\mid\bk\in\Z^2\}$ is dense in $\generictorus$.
    Therefore, there exists $\bk\in\Z^2$ such that
    $R^\bk\bx\in O\subset\interior{P_i}$. 
    Also $\bx\in \cap_{n=0}^{\infty}\overline{D}_n(w) \subset R^{-\bk}
    \overline{P_{w_\bk}}$ which implies
    $R^\bk\bx\in\overline{P_{w_\bk}}$.
    Thus $\overline{P_{w_\bk}}\cap\interior{P_i}\neq\varnothing$ which implies
    that $P_{w_\bk}=P_i$ and $w_\bk=i$
    since $\Pcal$ is a topological partition. Thus
    \[
        \cap_{n=0}^{\infty}\overline{D}_n(w)
        \subset R^{-\bk} \overline{P_{w_\bk}}
        =R^{-\bk} \overline{P_i}.
    \]
    The fact that $R^\bk\bx\in O$ 
    also means that 
    $R^\bk\bx\notin \overline{P_i}-(\by-\bx)$
    which can be rewritten as $R^\bk\by\notin \overline{P_i}$ or
    $\by\notin R^{-\bk}\overline{P_i}$ and we conclude that $\by\notin
    \cap_{n=0}^{\infty}\overline{D}_n(w)$.
    Thus $\Pcal$ gives a symbolic representation of
    $(\generictorus,\Z^2,R)$.
\end{proof}

The set
\begin{equation*}\label{eq:boundaries}
    \Delta_{\Pcal,R}:=\bigcup_{\bn\in\Z^2}R^\bn
         \left(\bigcup_{a\in\Acal}\partial P_a\right)
         \subset \generictorus
\end{equation*}
is the set of points whose orbit under the $\Z^2$-action of $R$ intersect
the boundary of the topological partition
$\Pcal=\{P_a\}_{a\in\Acal}$.
From Baire Category Theorem \cite[Theorem 6.1.24]{MR1369092}, the set
$\generictorus\setminus\Delta_{\Pcal,R}$ is dense in $\generictorus$.

A topological partition $\Pcal=\{P_a\}_{a\in \Acal}$ of $\generictorus=\R^2/\Gamma$
is associated to a coding map
\[
\begin{array}{rccl}
    \sccode:& \generictorus\setminus\left(\bigcup_{a\in \Acal}\partial
    P_a\right) &\to & \Acal \\
    &\bx &\mapsto & a \quad\text{ if and only if }
    \quad \bx\in P_a.
\end{array}
\]
For every starting point 
$\bx\in \generictorus\setminus \Delta_{\Pcal,R}$, 
the coding of its orbit under the
$\Z^2$-action of $R$ is a configuration
$\scConfig^{\Pcal,R}_{\bx}\in\Acal^{\Z^2}$ defined by
\[
    \scConfig^{\Pcal,R}_{\bx}(\bn)=\sccode(R^\bn(\bx)).
\]
for every $\bn\in\Z\times\Z$.

\begin{lemma}\label{lem:closure-of-tilings}
The symbolic dynamical system $\Xcal_{\Pcal,R}$
    corresponding to $\Pcal,R$ is the topological closure of the set of configurations:
\[
    \Xcal_{\Pcal,R}
    = \overline{\left\{\scConfig^{\Pcal,R}_\bx\mid
                       \bx\in\generictorus\setminus \Delta_{\Pcal,R}\right\}}.
\]
\end{lemma}

\begin{proof}
    ($\supseteq$) 
    Let $\bx\in\generictorus\setminus \Delta_{\Pcal,R}$.
    The patterns appearing in the configuration
    $\scConfig^{\Pcal,R}_\bx$ are in $\Lcal_{\Pcal,R}$.
    Thus $\scConfig^{\Pcal,R}_\bx\in\Xcal_{\Pcal,R}$.
    The topological closure of such configurations is in $\Xcal_{\Pcal,R}$ 
    since $\Xcal_{\Pcal,R}$ is topologically closed.

    ($\subseteq$)
    Let $w\in\Acal^S$ be a pattern with finite support $S\subset\Z^2$ appearing
    in $\Xcal_{\Pcal,R}$.
    Then $w\in\Lcal_{\Pcal,R}$ and from Equation~\eqref{eq:allowed-if-nonempty}
    there exists $\bx\in\generictorus\setminus\Delta_{\Pcal,R}$
    such that
        $\bx\in\bigcap_{\bk\in S} R^{-\bk}(P_{w_\bk})$.
        The pattern $w$ appears in the configuration
    $\scConfig^{\Pcal,R}_\bx$.
    Any configuration
    in $\Xcal_{\Pcal,R}$ is 
    the limit of a sequence $(w_n)_{n\in\N}$ of patterns covering
    a ball of radius $n$ around the origin,
    thus equal to
    some limit $\lim_{n\to\infty}\scConfig^{\Pcal,R}_{\bx_n}$
    with $\bx_n\in\generictorus\setminus\Delta_{\Pcal,R}$ for every $n\in\N$.
\end{proof}

\begin{exo}[label={exo:symbolic-representation}]
    Prove that $\Pcal_\Zcal$ gives a symbolic representation of the dynamical system
    $(\torus^2,\Z^2,R_\Zcal)$.
\end{exo}

\subsection{Factor map}

An important consequence of the fact that
a partition $\Pcal$ gives a symbolic representation of the dynamical system
$(\generictorus,\Z^2,R)$ is the existence of a
factor map $f:\Xcal_{\Pcal,R}\to\generictorus$ which commutes the
$\Z^2$-actions.
In the spirit of \cite[Prop.~6.5.8]{MR1369092} for $\Z$-actions,
we have the following proposition whose
proof can be found in \cite{MR4213162}.

\begin{proposition}\label{prop:factor-map}
    {\rm\cite[Prop.~5.1]{MR4213162}}
    Let $\Pcal$ give a symbolic representation of the dynamical system
    $(\generictorus,\Z^2,R)$.
    Let $f:\Xcal_{\Pcal,R}\to\generictorus$ be defined 
    such that $f(w)$ is the unique point
    in the intersection $\cap_{n=0}^{\infty}\overline{D}_n(w)$.
    The map $f$ is a factor map from
            $(\Xcal_{\Pcal,R},\Z^2,\sigma)$ to $(\generictorus,\Z^2,R)$
            such that $R^\bk\circ f = f\circ\sigma^\bk$
    for every $\bk\in\Z^2$.
    The map $f$ is one-to-one on
    $f^{-1}(\generictorus\setminus\Delta_{\Pcal,R})$.
\end{proposition}

Using the factor map, one can prove the following lemma.

\begin{lemma}\label{lem:minimal-aperiodic}
    {\rm\cite[Lemma~5.2]{MR4213162}}
    Let $\Pcal$ give a symbolic representation of the dynamical system
    $(\generictorus,\Z^2,R)$. Then
\begin{enumerate}[\rm (i)]
    \item if $(\generictorus,\Z^2,R)$ is minimal,
        then $(\Xcal_{\Pcal,R},\Z^2,\sigma)$ is minimal,
    \item if $R$ is a free $\Z^2$-action on $\generictorus$,
        then $\Xcal_{\Pcal,R}$ aperiodic.
\end{enumerate}
\end{lemma}

Of course, Lemma~\ref{lem:minimal-aperiodic}
does not hold
if $\Pcal$ does not give a symbolic representation of
$(\generictorus,\Z^2,R)$.

\begin{exo}[label={exo:symbolic-ds-minimal-aperiodic}]
    Prove that $(\Xcal_{\Pcal_\Zcal,R_\Zcal},\Z^2,\sigma)$ is minimal and aperiodic.
\end{exo}

\subsection{Induced $\Z^2$-actions}

Renormalization schemes also known as \emph{Rauzy induction} were originally
defined for dynamical systems including interval exchange transformations
(IET) \cite{MR543205}. A natural way to generalize it to higher
dimension is to consider polygon exchange transformations
\cite{MR3010377,alevy_kenyon_yi} or even polytope exchange
transformations \cite{MR3186232,schwartz_outer_2011} where only one map is
considered. But more dimensions also allow to induce two or more (commuting)
maps at the same time.

In this section, we define the induction of $\Z^2$-actions on a sub-domain.
We consider the torus $\generictorus=\R^2/\Gamma$
where $\Gamma$ is a lattice in $\R^2$.
Let $(\generictorus,\Z^2,R)$ be a minimal dynamical system
given by a $\Z^2$-action $R$ on $\generictorus$.
For every $\bn\in\Z^2$, the toral translation $R^{\bn}$ can be seen as a pair
of polygon exchange transformations on a fundamental domain of
$\generictorus$.

Let $W\subset\generictorus$ be a set.
The \defn{set of return times} of $\bx\in\generictorus$
to the \defn{window} $W$ under the $\Z^2$-action $R$ is the subset of
$\Z\times\Z$ defined as:
\[
    \delta_W(\bx) = \{ \bn\in\Z\times\Z\mid R^\bn(\bx)\in W \}.
\]

\begin{definition}\label{def:Cartesian-on-a-window}
Let $W\subset\generictorus$.
    We say that the $\Z^2$-action $R$ is \defn{Cartesian on} $W$ if
the set of return times $\delta_W(\bx)$ can be expressed as a Cartesian product,
that is,
for all $\bx\in\generictorus$
there exist two strictly increasing sequences $r_\bx^{(1)},r_\bx^{(2)}:\Z\to\Z$
such that
\[
\delta_W(\bx)=r_\bx^{(1)}(\Z)\times r_\bx^{(2)}(\Z).
\]
We always assume that the sequences are
shifted in such a way that
\[
r_\bx^{(i)}(n)\geq 0 \iff n\geq0
\qquad
\text{ for }
i\in\{1,2\}.
\]
In particular, if $\bx\in W$ then $(0,0)\in\delta_W(\bx)$, 
and therefore $r_\bx^{(1)}(0)=r_\bx^{(2)}(0)=0$.
\end{definition}

When the $\Z^2$-action $R$ is Cartesian on $W\subset\generictorus$,
we say that the tuple
\begin{equation}\label{eq:first-return-time-tuple}
    (r_\bx^{(1)}(1),r_\bx^{(2)}(1))
\end{equation}
is
the \defn{first return time} of a starting point $\bx\in\generictorus$
to $W\subset\generictorus$ under the action $R$.
When the $\Z^2$-action $R$ is Cartesian on $W\subset\generictorus$,
we may consider its return map on $W$ and we prove in the next lemma that this
induces a $\Z^2$-action on $W$.

\begin{lemma}\label{lem:Cartesian-induced-action}
If the $\Z^2$-action $R$ is Cartesian on $W\subset\generictorus$, then
the map
    $\widehat{R}|_W:\Z^2\times W\to W$ defined by
\[
    (\widehat{R}|_W)^\bn(\bx):=\widehat{R}|_W(\bn,\bx)=
    R^{(r_\bx^{(1)}(n_1),r_\bx^{(2)}(n_2))}(\bx)
\]
for every $\bn=(n_1,n_2)\in\Z^2$ 
and $\bx\in W$
    is a well-defined $\Z^2$-action on $W$.
\end{lemma}

We say that $\widehat{R}|_W$ is the \defn{induced $\Z^2$-action} of the $\Z^2$-action $R$ on $W$.

\begin{proof}
    Let $\bx\in W$.
    We have that
    \[
    \widehat{R}|_W(\zero,\bx)
    = R^{(r_\bx^{(1)}(0),r_\bx^{(2)}(0))}(\bx)
    = R^{(0,0)}(\bx)
    = \bx.
    \]
    Let $n\in\Z$ and
    $\bx'=\widehat{R}|_W(n\be_i,\bx)=R^{\be_i\cdot r_\bx^{(i)}(n)}(\bx)$
    be the $n$-th horizontal return to $W$ under $R^{\be_i}$ of the point $\bx$.
    For every $k\in\Z$, the horizontal return times to $\bx$ and to $\bx'$ satisfy
    the equation
    \[
        r_\bx^{(i)}(k+n) = r^{(i)}_{\bx'}(k)+r_\bx^{(i)}(n).
    \]
    We get
    \begin{align*}
    \widehat{R}|_W(k\be_i+n\be_i,\bx)
        &= R^{\be_i\cdot r_\bx^{(i)}(k+n)}(\bx)
         = R^{\be_i\cdot(r_{\bx'}(k)+r_\bx^{(i)}(n))}(\bx)\\
        &= R^{\be_i\cdot r_{\bx'}(k)}
           \left(R^{\be_i\cdot r_\bx^{(i)}(n)}(\bx)\right)
         = R^{\be_i\cdot r_{\bx'}(k)} \left(\bx'\right)\\
        &= R^{\be_i\cdot r_{\widehat{R}|_W(n\be_i,\bx)}(k)}
           \left(\widehat{R}|_W(n\be_i,\bx)\right)\\
        &= \widehat{R}|_W\left(k\be_i, \left(\widehat{R}|_W(n\be_i,\bx)\right)\right).
    \end{align*}
    Secondly, using the fact that
    \[
        r_\bx^{(1)}(k_1) = r_\by^{(1)}(k_1)
    \]
    whenever $\by=R^{(0,r_\bx^{(2)}(k_2))}\bx=\widehat{R}|_W(k_2\be_2,\bx)$, we
    obtain
    \begin{align*}
    \widehat{R}|_W(\bk,\bx)
    &= R^{(r_\bx^{(1)}(k_1),r_\bx^{(2)}(k_2))}(\bx)
    = R^{(r_\bx^{(1)}(k_1),0)} R^{(0,r_\bx^{(2)}(k_2))}(\bx)\\
    &= R^{(r_\bx^{(1)}(k_1),0)} \widehat{R}|_W(k_2\be_2,\bx)
    = \widehat{R}|_W\left(k_1\be_1,
      \widehat{R}|_W(k_2\be_2,\bx)\right).
    \end{align*}
    Therefore, for every $\bk,\bn\in\Z^2$, we have
    \begin{align*}
        (\widehat{R}|_W)^{\bk+\bn}(\bx)
    &= (\widehat{R}|_W)^{(k_1+n_1)\be_1}
        (\widehat{R}|_W)^{(k_2+n_2)\be_2}(\bx)\\
     &= (\widehat{R}|_W)^{k_1\be_1} (\widehat{R}|_W)^{n_1\be_1}
        (\widehat{R}|_W)^{k_2\be_2} (\widehat{R}|_W)^{n_2\be_2}(\bx)\\
     &= (\widehat{R}|_W)^{k_1\be_1}
        (\widehat{R}|_W)^{(n_1,k_2)} (\widehat{R}|_W)^{n_2\be_2}(\bx)\\
     &= (\widehat{R}|_W)^{k_1\be_1} (\widehat{R}|_W)^{k_2\be_2}
        (\widehat{R}|_W)^{n_1\be_1} (\widehat{R}|_W)^{n_2\be_2}(\bx)\\
     &= (\widehat{R}|_W)^\bk (\widehat{R}|_W)^\bn(\bx),
    \end{align*}
    which shows that $\widehat{R}|_W$ is a $\Z^2$-action on $W$.
\end{proof}

A consequence of the lemma is that
the induced $\Z^2$-action $\widehat{R}|_W$
is generated by two commutative maps
\[
    (\widehat{R}|_W)^{\be_1}(\bx)
    = R^{(r_\bx^{(1)}(1),0)}(\bx)
    \quad\text{ and }\quad
    (\widehat{R}|_W)^{\be_2}(\bx)
    = R^{(0,r_\bx^{(2)}(1))}(\bx)
\]
which are the first return maps of $R^{\be_1}$ and $R^{\be_2}$
to $W$:
\[
    (\widehat{R}|_W)^{\be_1}(\bx)
    = \widehat{R^{\be_1}}|_W(\bx)
    \quad\text{ and }\quad
    (\widehat{R}|_W)^{\be_2}(\bx)
    = \widehat{R^{\be_2}}|_W(\bx).
\]
Recall that the \defn{first return map}
$\widehat{T}|_W$ of a dynamical system $(X,T)$ maps a point $\bx\in W\subset X$
to the first point in the forward orbit of
$T$ lying in $W$, i.e.
\[
    \widehat{T}|_W(\bx) = T^{r(\bx)}(\bx) \quad\text{ where }
    r(\bx) = \min\{k\in\Z_{>0} : T^k(\bx)\in W\}.
\]
From Poincar\'e's recurrence theorem,
if $\mu$ is a finite $T$-invariant measure on $X$, then
the first return map $\widehat{T}|_W$ is well
defined for $\mu$-almost all $\bx\in W$.
When $T$ is a rotation on a torus, if the subset $W$ is
open, then the first return map is well-defined for every point $\bx\in W$.
Moreover if $W$ is a polygon, then the first return map $\widehat{T}|_W$
is a polygon exchange transformation.
An algorithm to compute the induced transformation
$\widehat{T}|_W=\widehat{R^{\be_i}}|_W$ of the sub-action $R^{\be_i}$
is provided in \cite{MR4347332}.

\begin{exo}[label={exo:PET-on-triangle-nontrivial}]
    Let $T(\bx) = \bx+(\frac{1}{3},\frac{1}{2})$
    be a $\Z$-action defined on $\torus^2$.
    Let $W\subset\torus^2$ be the triangle with vertices $(0,0)$, $(1,0)$ and
    $(0,1)$. Prove that $\widehat{T}|_W$ is the PET defined
    in Figure~\ref{fig:PET-triangle}.
\end{exo}

\begin{exo}[label={exo:induction-z2-action}]
    Recall that $R_\Zcal(\bn,\bx) = \bx+\varphi^{-2}\bn$
    is a $\Z^2$-action defined on $\torus^2$ where
    $\varphi=\frac{1+\sqrt{5}}{2}$.
    Let $W_0=(0,1)\times(0,\varphi^{-1})+\Z^2$ be a subset of $\torus^2$.
    \begin{itemize}
        \item Prove that the action $R_\Zcal$ is Cartesian on $W_0$.
        \item Prove that $\widehat{R_\Zcal}|_{W_0}:\Z^2\times W_0\to W_0$ is a
            well-defined induced $\Z^2$-action.
        \item Describe $\widehat{R_\Zcal^{\be_1}}|_{W_0}$ and 
                       $\widehat{R_\Zcal^{\be_2}}|_{W_0}$
                       as polygon exchange transformations on $W_0$.
        \item Describe $\widehat{R_\Zcal}|_{W_0}$ as a toral $\Z^2$-rotation on
            $\R^2/\Gamma_1$ with $\Gamma_1=\Z\times(\varphi^{-1}\Z)$.
    \end{itemize}
\end{exo}

\subsection{Toral partitions induced by toral $\Z^2$-rotations}

For IETs, the interval on which we define the Rauzy induction is usually given
by one of the atom of the partition which defines the IET itself. In our setting, it is not the
case. The partition that we use carries more information than the natural
partition which allows to define the $\Z^2$-rotation $R$ as a pair of polygon
exchange transformations.
The partition is a refinement  of the natural partition 
involving well-chosen sloped lines.

Let $\Gamma$ be a lattice in $\R^2$ and
$\generictorus=\R^2/\Gamma$ be a $2$-dimensional torus.
Let $(\generictorus,\Z^2,R)$ be the dynamical system 
given by a $\Z^2$-rotation $R$ on $\generictorus$.
Assuming the $\Z^2$-rotation $R$ is Cartesian on a window $W\subset\generictorus$,
then there exist
two strictly increasing sequences $r_\bx^{(1)},r_\bx^{(2)}:\Z\to\Z$ 
such that
\begin{equation}\label{eq:r-and-s}
    (r,s) = \left(r_\bx^{(1)}(1),r_\bx^{(2)}(1)\right)
\end{equation}
is the first return time of a starting point $\bx\in\generictorus$ to the
window $W$ under the action $R$, see
Equation~\eqref{eq:first-return-time-tuple}.
It allows to define the \defn{return word map} as
\[
\begin{array}{rccl}
    \scReturnWord:&W & \to & \Acal^{*^2}\\
    &\bx&\mapsto &
    \left(\begin{array}{ccc}
        \sccode(R^{(0,s-1)}\bx) & \cdots & 
        \sccode(R^{(r-1,s-1)}\bx)\\
        \cdots & \cdots & \cdots\\
        \sccode(R^{(0,0)}\bx) & \cdots & \sccode(R^{(r-1,0)}\bx)
    \end{array}\right)
\end{array},
\]
where $r,s\geq 1$ both obviously depend on $\bx$.

The image $\Lcal=\scReturnWord(W)\subset \Acal^{*^2}$ is a language called the \defn{set of return words}.
We identify each return word in $\Lcal$ to a letter $b$ of an
alphabet $\Bcal$ in such a way that $\Lcal=\{w_b\}_{b\in\Bcal}$.
When the return time to $W$ is bounded, the set of return words $\Lcal$ is
finite and the alphabet $\Bcal$ is finite. The examples that we consider in this
chapter are such that the return time is bounded, but this is not true in
general.

\begin{remark}
    The way the enumeration of $\Lcal$ is done influences the substitutions
    which are obtained afterward. To obtain a canonical ordering when the words are 1-dimensional, 
    we use the total order $(\Lcal,\prec)$
    defined by $u\prec v$ if $|u|<|v|$ or $|u|=|v|$ and $u<_{lex}v$.
\end{remark}

The \defn{induced partition} of $\Pcal$ by the action of $R$ on the sub-domain
$W$ is a topological partition of $W$ defined as the set of preimage sets
under $\scReturnWord$:
\[
    \widehat{\Pcal}|_W=\{\scReturnWord^{-1}(w_b)\}_{b\in\Bcal}.
\]
This yields the \defn{induced coding} on $W$
\[
\begin{array}{rccl}
    \sccode|_W:&W & \to &\Bcal\\
    &\by & \mapsto & b \quad \text{ if and only if } 
                       \quad\by\in\scReturnWord^{-1}(w_b).
\end{array}
\]
A \defn{natural substitution} comes out of this induction procedure:
\begin{equation}\label{eq:induced-substitution}
\begin{array}{rccl}
    \omega:&\Bcal & \to & \Acal^{*^2}\\
    &b & \mapsto & w_b.
\end{array}
\end{equation}
The partition $\widehat{\Pcal}|_W$ of $W$ can be effectively computed by the
refinement of the partition $\Pcal$ with translated copies of the sub-domain $W$
under the action of $R$. 
In \cite{MR4347332},
we propose an algorithm
to compute the induced partition $\widehat{\Pcal}|_W$ 
and substitution $\omega$. An implementation of it in SageMath
is provided in the optional package \texttt{slabbe} \cite{labbe_slabbe_0_7_6_2023} 
and is used below on an example.
The next result shows that
the coding of the orbit under the $\Z^2$-rotation $R$ is
the image under the
2-dimensional substitution $\omega$ of the coding of the orbit under
the $\Z^2$-action $\widehat{R}|_W$.

\begin{lemma}\label{lem:desubstitute}
If the $\Z^2$-action $R$ is Cartesian on a window $W\subset\generictorus$,
    then $\omega$ is a $2$-dimensional morphism, and
    for every $\bx\in W$ we have
    \[
        \scConfig^{\Pcal,R}_{\bx}
        = \omega\left(
        \scConfig^{\widehat{\Pcal}|_W,\widehat{R}|_W}_{\bx}
        \right).
    \]
\end{lemma}

\begin{proof}
    Let $\bx\in W$.
Since the $\Z^2$-action $R$ is Cartesian on $W$,
there exist two strictly increasing sequences $r_\bx^{(1)},r_\bx^{(2)}:\Z\to\Z$
such that $\delta_W(\bx)=r_\bx^{(1)}(\Z)\times r_\bx^{(2)}(\Z)$.
Since $\bx\in W$, we have $(0,0)\in\delta_W(\bx)$
and $r_\bx^{(1)}(0)=r_\bx^{(2)}(0)=0$.
To use a lighter notation in the argument that follows, let
$r_i=r_\bx^{(1)}(i)$
and
$s_j=r_\bx^{(2)}(j)$ for every $i,j\in\Z$.

    Using these two increasing sequences, $\scConfig^{\Pcal,R}_{\bx}$ may be decomposed into
    a lattice of rectangular blocks.
    More precisely, for every $i,j\in\Z$, the following block is the image of
    a letter under $\omega$:
    \begin{align*}
        &\left(\begin{array}{ccc}
            \sccode(R^{(r_i,s_{j+1}-1)}\bx) & \cdots 
                & \sccode(R^{(r_{i+1}-1,s_{j+1}-1)}\bx)\\
            \cdots & \cdots & \cdots\\
            \sccode(R^{(r_i,s_j)}\bx) & \cdots 
                & \sccode(R^{(r_{i+1}-1,s_j)}\bx)
        \end{array}\right)\\
        &\qquad= \scReturnWord(R^{(r_i,s_j)}\bx)
        =w_{b_{ij}}=\omega(b_{ij})
    \end{align*}
    for some letter $b_{ij}\in\Bcal$.
    Moreover,
    \[
        b_{ij}
        = \sccode|_W(R^{(r_i,s_j)}\bx)
        = \sccode|_W\left((\widehat{R}|_W)^{(i,j)}\bx\right).
    \]
    Since the adjacent blocks have matching dimensions, 
    for every $i,j\in\Z$, the following concatenations
    \begin{align*}
        & \omega\left(b_{ij}\odot^1 b_{(i+1)j}\right) = 
        \omega\left(b_{ij}\right)\odot^1 \omega\left(b_{(i+1)j}\right)
        \qquad\text{ and }
        \\
        & \omega\left(b_{ij}\odot^2 b_{i(j+1)}\right) = 
        \omega\left(b_{ij}\right)\odot^2 \omega\left(b_{i(j+1)}\right)
    \end{align*}
    are well defined. Thus $\omega$ is a $2$-dimensional morphism
    on the set
    $\left\{ \scConfig^{\widehat{\Pcal}|_W,\widehat{R}|_W}_{\bx}
            \,\middle|\, \bx\in W\right\}$
    and we have
    \[
        \scConfig^{\Pcal,R}_{\bx}
        =
        \omega\left(\scConfig^{\widehat{\Pcal}|_W,\widehat{R}|_W}_{\bx}\right)
    \]
    which ends the proof.
    Note that the domain of $\omega$ can be extended to
    its topological closure.
\end{proof}

\begin{proposition}\label{prop:orbit-preimage}
    Let $\Pcal$ be a topological partition of $\generictorus$.
If the $\Z^2$-action $R$ is Cartesian on a window $W\subset\generictorus$,
    then $\Xcal_{\Pcal,R}=\shiftclosure{\omega(\Xcal_{\widehat{\Pcal}|_W,\widehat{R}|_W})}$.
\end{proposition}

\begin{proof}
    Let
    \[
        Y = \left\{ \scConfig^{\Pcal,R}_{\bx}\,\middle|\,
                    \bx\in\generictorus\right\}
        \quad\text{ and }\quad
        Z = \left\{ \scConfig^{\widehat{\Pcal}|_W,\widehat{R}|_W}_{\bx}
               \,\middle|\, \bx\in W\right\},
    \]
    ($\supseteq$).
        Let $\bx\in W$.
        From Lemma~\ref{lem:desubstitute},
        $\omega\left(
         \scConfig^{\widehat{\Pcal}|_W,\widehat{R}|_W}_{\bx}\right)
         = \scConfig^{\Pcal,R}_{\bx}$ with $\bx\in W\subset\generictorus$.
    ($\subseteq$).
    Let $\bx\in\generictorus$. There exist $k_1,k_2\in\N$ such that
    $\bx'=R^{-(k_1,k_2)}(\bx)\in W$. 
    Therefore, we have $\bx=R^{(k_1,k_2)}(\bx')$ where $0\leq k_1<r_{\bx'}^{(1)}(1)$ and
    $0\leq k_2<r_{\bx'}^{(2)}(1)$.

    Thus the shift $\bk=(k_1,k_2)\in\Z^2$ is bounded by the maximal return time
    of $R^{\be_1}$ and $R^{\be_2}$ to $W$. We have
    \begin{align*}
        \scConfig^{\Pcal,R}_{\bx}
        &= \scConfig^{\Pcal,R}_{R^\bk\bx'}
        = \sigma^{\bk} \circ
        \scConfig^{\Pcal,R}_{\bx'}\\
        &= \sigma^{\bk} \circ \omega
        \left(\scConfig^{\widehat{\Pcal}|_W,\widehat{R}|_W}_{\bx'}\right)
    \end{align*}
    where we used
    Lemma~\ref{lem:desubstitute} with $\bx'\in W$.
    We conclude that $Y=\shiftclosure{\omega(Z)}$.
    The result follows from Lemma~\ref{lem:closure-of-tilings}
    by taking the topological closure on both sides.
\end{proof}

\begin{exo}[label={exo:induced-partition-on-triangle}]
    Let $T(\bx) = \bx+(\frac{1}{3},\frac{1}{2})$
    be a $\Z$-action defined on $\torus^2$ as in Exercise~\ref{exo:PET-on-triangle-nontrivial}.
    Let $W\subset\torus^2$ be the triangle with vertices $(0,0)$, $(1,0)$ and
    $(0,1)$.
    Let $\Pcal$ be the polygonal partition of $\torus^2$ into two parts
    separated by the positive diagonal.
    Compute the substitution and the partition $\widehat{\Pcal}|_{W}$ of
    $\Pcal$ induced by the action of $T$ on the sub-domain $W$.
\end{exo}

\subsection{Self-similarity of the subshift $\Xcal_{\Pcal_\Zcal,R_\Zcal}$}
\label{sec:self-similarity-of-X-PZ-RZ}

In this section, we induce the topological partition
$\Pcal_\Zcal$ until the process loops.
We need two induction steps before obtaining a topological partition
which is self-induced.


The proof contains SageMath code 
using the \texttt{slabbe} optional package
\cite{labbe_slabbe_0_7_6_2023}
to reproduce the computation of the induced partitions and
$2$-dimensional morphisms. 
The algorithm of the method 
\texttt{induced\_transformation}
and
\texttt{induced\_partition} are available in 
\cite{MR4347332}.

\begin{proof}[Proof of Theorem~\ref{thm:main-theoremB}]
First, we define the golden mean \texttt{phi} as an element of a number field
defined by a quadratic polynomial which is more efficient when doing
arithmetic operations and comparisons.
We also import the necessary functions.
\begin{sagecommandlinetcb}
\begin{sagecommandline}
sage: z = polygen(QQ, "z")
sage: K.<phi> = NumberField(z**2-z-1, "phi", embedding=RR(1.6))
sage: from slabbe import PolyhedronExchangeTransformation as PET
sage: from slabbe.arXiv_1903_06137 import self_similar_19_atoms_partition
\end{sagecommandline}
\end{sagecommandlinetcb}
The proof uses Proposition~\ref{prop:orbit-preimage} two times to induce both
the vertical and horizontal actions, starting with the vertical action.
We begin with the lattice $\Gamma_0=\Z^2$, the
partition $\Pcal_\Zcal$, the coding map $\sccode_0:\R^2/\Gamma_0\to\Acal_0$, the
alphabet $\Acal_0=\Zrange{15}$ and $\Z^2$-action $R_\Zcal$ defined on
$\torus^2$ as shown below.
\begin{sagecommandlinetcb}
\begin{sagecommandline}
sage: Gamma0 = matrix.column([(1,0), (0,1)])
sage: PU = self_similar_19_atoms_partition()
sage: merge_dict = {0:0, 1:1, 2:2, 3:3, 4:4, 5:5, 6:6, 7:6, 8:7, 9:8, 10:9, 
....:               11:10, 12:11, 13:11, 14:12, 15:12, 16:13, 17:14, 18:15}
sage: PZ = PU.merge_atoms(merge_dict)
sage: RZe1 = PET.toral_translation(Gamma0, vector((phi^-2,0)))
sage: RZe2 = PET.toral_translation(Gamma0, vector((0,phi^-2)))
\end{sagecommandline}
\end{sagecommandlinetcb}
    \[
    \begin{array}{c}
        \Pcal_\Zcal=
    \end{array}
    \begin{array}{c} 
        \sageplot[][pdf]{PZtikz}
    \end{array}
    \quad
    \begin{array}{l}
        R_\Zcal^\bn(\bx)=\bx+\varphi^{-2}\bn
    \end{array}
    \] 

    We consider the window $W_0=(0,1)\times(0,\varphi^{-1})+\Gamma_0$ as a
    subset of $\R^2/\Gamma_0$.
    The action $R_\Zcal$ is Cartesian on $W_0$.
    Thus from Lemma~\ref{lem:Cartesian-induced-action},
    $R_1:=\widehat{R_\Zcal}|_{W_0}:\Z^2\times W_0\to W_0$ is a well-defined
    $\Z^2$-action.
    From Lemma~\ref{lem:PET-iff-toral-rotation},
    the $\Z^2$-action $R_1$ can be seen as a toral rotation on
    $\R^2/\Gamma_1$ with $\Gamma_1=\Z\times(\varphi^{-1}\Z)$, see
    Exercise~\ref{exo:induction-z2-action}.
    Let $\Pcal_1=\widehat{\Pcal_\Zcal}|_{W_0}$ be the induced partition.
    From Proposition~\ref{prop:orbit-preimage}, then
    $\Xcal_{\Pcal_\Zcal,R_\Zcal}=
    \shiftclosure{\beta_0(\Xcal_{\Pcal_1,R_1})}$.
    The partition $\Pcal_1$, the action $R_1$ and substitution
    $\beta_0$ are given below with alphabet $\Acal_1=\Zrange{17}$:
\begin{sagecommandlinetcb}
\begin{sagecommandline}
sage: y_ineq = [phi^-1, 0, -1] # y <= phi^-1 (see Polyhedron?)
sage: P1,beta0 = RZe2.induced_partition(y_ineq,PZ,
....:                                   substitution_type="column")
sage: R1e1,_ = RZe1.induced_transformation(y_ineq)
sage: R1e2,_ = RZe2.induced_transformation(y_ineq)
\end{sagecommandline}
\end{sagecommandlinetcb}
\begin{sagesilent}
axis1 = axis_HV(vertical_axis=[0,-1,-2,-3])
P1tikz = P1.tikz(extra_code=axis1, fontsize=r'\normalsize', scale=5)
\end{sagesilent}
    \[
    \begin{array}{c}\beta_0:\Acal_1\to \Acal_0^{*^2}\\[2mm]
        \left\{
            \setlength{\arraycolsep}{0pt}
            \footnotesize
            \sage{beta0._latex_(ncolumns=5, align='l')}
    \right.
    \end{array}
    \]
    \[
    \begin{array}{c}
        \Pcal_1=
    \end{array}
    \begin{array}{c} 
        \sageplot[][pdf]{P1tikz}
    \end{array}
        \quad
        \begin{array}{l}
            R_1^\bn(\bx)=\bx+ \\
            \quad(\varphi^{-2}n_1,-\varphi^{-3}n_2)
        \end{array}
    \]
    We consider the window
    $W_1=(0,\varphi^{-1})\times(0,\varphi^{-1})+\Gamma_1$ as a subset of
    $\R^2/\Gamma_1$.
    The action $R_1$ is Cartesian on $W_1$.
    Thus from Lemma~\ref{lem:Cartesian-induced-action},
    $R_2:=\widehat{R_1}|_{W_1}:\Z^2\times W_1\to W_1$ is a well-defined
    $\Z^2$-action.
    From Lemma~\ref{lem:PET-iff-toral-rotation},
    the $\Z^2$-action $R_2$ can be seen as a toral rotation on
    $\R^2/\Gamma_2$ with $\Gamma_2=(\varphi^{-1}\Z)\times(\varphi^{-1}\Z)$.
    Let $\Pcal_2=\widehat{\Pcal_1}|_{W_1}$ be the induced partition.
    From Proposition~\ref{prop:orbit-preimage}, then
    $\Xcal_{\Pcal_1,R_1}=
    \shiftclosure{\beta_1(\Xcal_{\Pcal_2,R_2})}$.
    The partition $\Pcal_2$, the action $R_2$ and substitution
    $\beta_1$ are given below with alphabet $\Acal_2=\Zrange{15}$:
\begin{sagecommandlinetcb}
\begin{sagecommandline}
sage: x_ineq = [phi^-1, -1, 0] # x <= phi^-1 (see Polyhedron?)
sage: P2,beta1 = R1e1.induced_partition(x_ineq, P1, substitution_type="row")
sage: R2e1,_ = R1e1.induced_transformation(x_ineq)
sage: R2e2,_ = R1e2.induced_transformation(x_ineq)
\end{sagecommandline}
\end{sagecommandlinetcb}
\begin{sagesilent}
axis2 = axis_HV(horizontal_axis=[0,-1,-2,-3],vertical_axis=[0,-1,-2,-3])
P2tikz = P2.tikz(extra_code=axis2, fontsize=r'\normalsize', scale=5)
\end{sagesilent}
    \[
    \begin{array}{c}\beta_1:\Acal_2\to \Acal_1^{*^2}\\[2mm]
        \left\{
            \setlength{\arraycolsep}{0pt}
            \footnotesize
            \sage{beta1._latex_(ncolumns=4, align='l')}
    \right.
    \end{array}
    \]
\[
    \begin{array}{c}
        \Pcal_2=
    \end{array}
    \begin{array}{c} 
        \sageplot[][pdf]{P2tikz}
    \end{array}
        \quad
    \begin{array}{c}
        R_2^\bn(\bx)=\bx - \varphi^{-3}\bn
    \end{array}
\]
    Now it is appropriate to rescale the partition $\Pcal_2$ by
    the factor $-\varphi$.  
    Doing so, the new obtained action $R'_2$ is the same as two steps before,
    that is, $R_\Zcal$ on $\R^2/\Z^2$. 
    More formally, let $h:(\R/\varphi^{-1}\Z)^2\to(\R/\Z)^2$
be the homeomorphism defined by $h(\bx)=-\varphi\bx$. 
We define 
$\Pcal_2'=h(\Pcal_2)$,
$\sccode_2'=\sccode_2\circ h^{-1}$,
 $(R'_2)^\bn=h\circ (R_2)^\bn\circ h^{-1}$
as shown below:
\begin{sagecommandlinetcb}
\begin{sagecommandline}
sage: P2_scaled = (-phi*P2).translate((1,1))
\end{sagecommandline}
\end{sagecommandlinetcb}
\begin{sagesilent}
axis0 = axis_HV()
P2scaledtikz = P2_scaled.tikz(extra_code=axis0, fontsize=r'\normalsize', scale=5)
\end{sagesilent}
    \[
    \begin{array}{c}
        \Pcal_2'=
    \end{array}
    \begin{array}{c} 
        \sageplot[][pdf]{P2scaledtikz}
    \end{array}
    \quad
    \begin{array}{l}
        (R'_2)^\bn(\bx)=\bx+\varphi^{-2}\bn
    \end{array}
    \] 
We observe that the scaled partition $\Pcal_2'$ is the same as $\Pcal_\Zcal$
up to a permutation $\beta_2$ of the indices of the atoms
in such a way that $\beta_2\circ\sccode_0=\sccode_2'$.
The partition $\Pcal_\Zcal$, the action $R_\Zcal$ and substitution
$\beta_2:\Acal_2\to\Acal_0$ are given below.
\begin{sagecommandlinetcb}
\begin{sagecommandline}
sage: assert P2_scaled.is_equal_up_to_relabeling(PZ)
sage: beta2 = Substitution2d.from_permutation(PZ.keys_permutation(P2_scaled))
\end{sagecommandline}
\end{sagecommandlinetcb}
    \[
    \begin{array}{ll}
        \beta_2:&\Zrange{15}\to\Zrange{15}^{*^2}\\[2mm]
        &\left\{\arraycolsep=1.8pt
                \sage{beta2._latex_(ncolumns=4, align='l')}
        \right.
    \end{array}
    \]
By construction, the following diagrams commute:
\[ 
\begin{tikzcd}
    (\R/\varphi^{-1}\Z)^2 \arrow{r}{h} \arrow[swap]{d}{R_2^\bn}
    & (\R/\Z)^2 \arrow{d}{(R_2')^\bn} \\%
    (\R/\varphi^{-1}\Z)^2 \arrow{r}{h}
    & (\R/\Z)^2
\end{tikzcd}
\qquad
    \text{ and }
\qquad
\begin{tikzcd}
    (\R/\varphi^{-1}\Z)^2 \arrow{r}{h} \arrow[swap]{d}{\sccode_2}
    & (\R/\Z)^2 \arrow[swap,sloped,near start]{ld}{\sccode_2'} \arrow{d}{\sccode_0}\\%
    \Acal_2 
    & \Acal_2 \arrow[swap]{l}{\beta_2} 
\end{tikzcd}.
\]
Using the above commutative properties, for every $\by\in\R^2/\Gamma_2$ and
$\bn\in\Z^2$, we have
\begin{align*}
    \scConfig^{\Pcal_2,R_2}_{\by} (\bn)
    &= \sccode_2(R_2^\bn (\by))
    = \sccode_2' \circ h(R_2^\bn (\by))\\
    &= \beta_2\circ\sccode_0 \circ h(R_2^\bn (\by))
    = \beta_2\circ\sccode_0 \circ (R_2')^\bn (h(\by)))\\
    &= \beta_2\circ\sccode_0 \circ R_\Zcal^\bn (h(\by)))
    = \beta_2\left(\scConfig^{\Pcal_\Zcal,R_\Zcal}_{h(\by)} (\bn)\right).
\end{align*}
    Thus $\Xcal_{\Pcal_2,R_2} =\beta_2\left(\Xcal_{\Pcal_\Zcal,R_\Zcal}\right)$.
We may check that $\beta_0\circ\beta_1\circ\beta_2=\Phi$
where the variable \texttt{Phi} was created in Exercise~\ref{exo:Phi-sagemath}:
\begin{sagecommandlinetcb}
\begin{sagecommandline}
sage: beta0 * beta1 * beta2 == Phi
True
\end{sagecommandline}
\end{sagecommandlinetcb}
\noindent
We conclude that
\begin{align*}
    \Xcal_{\Pcal_\Zcal,R_\Zcal}
    &= \shiftclosure{\beta_0(\Xcal_{\Pcal_1,R_1})}
    = \shiftclosure{\beta_0\beta_1(\Xcal_{\Pcal_2,R_2})}
    = \shiftclosure{\beta_0\beta_1\beta_2\left(\Xcal_{\Pcal_\Zcal,R_\Zcal}\right)}
    = \shiftclosure{\Phi\left(\Xcal_{\Pcal_\Zcal,R_\Zcal}\right)}.
\end{align*}
\end{proof}

\begin{exo}[label={exo:XPZRZ-domino-language}]
    Compute the language of the dominoes and patterns of shape $2\times2$
    within $\Xcal_{\Pcal_\Zcal,R_\Zcal}$:
    \[
        \Lcal_{1\times 2}(\Xcal_{\Pcal_\Zcal,R_\Zcal}),\quad
        \Lcal_{2\times 1}(\Xcal_{\Pcal_\Zcal,R_\Zcal})
        \quad
        \text{ and }
        \quad
        \Lcal_{2\times 2}(\Xcal_{\Pcal_\Zcal,R_\Zcal}).
    \]
\end{exo}

\begin{exo}[label={exo:XPZRZ=XPhi}]
    Using the criterion given in Lemma~\ref{lem:criterion-for-minimality},
    prove that $\Xcal_{\Pcal_\Zcal,R_\Zcal}$ is minimal
    and $\Xcal_{\Pcal_\Zcal,R_\Zcal}=\Xcal_\Phi$.
\end{exo}

\begin{exo}[label={exo:induction-first-horizontal}]
    Prove the self-similarity of $\Xcal_{\Pcal_\Zcal,R_\Zcal}$ 
    by doing the induction first horizontaly with $R_\Zcal^{\be_1}$, 
    and then vertically with $R_\Zcal^{\be_2}$.
    Compare with the result of Exercise~\ref{exo:markers-e1-and-then-e2}.
    See also Exercise~\ref{exo:desubstitute-markers-e1}.
\end{exo}

\section{Conclusion}

We may now conclude with a proof of Theorem~\ref{thm:equality-three-subshifts}
providing three different characterizations of the same aperiodic 2-dimensional
subshift.

\begin{proof}[Proof of Theorem~\ref{thm:equality-three-subshifts}]
In Section~\ref{chap:Labbe:sec:aperiodic-self-similar-subshifts},
we defined the 2-dimensional subshift $\Xcal_\Phi$ from some 2-dimensional
morphism $\Phi$.

In Section~\ref{chap:Labbe:sec:wang-shifts}
we proved that
$\Omega_\Zcal=\shiftclosure{\Phi(\Omega_\Zcal)}$ using the desubstitution of
Wang shifts using marker tiles.
In Exercise~\ref{exo:OmegaU=XPhi}, we proved using the criterion 
given in Lemma~\ref{lem:criterion-for-minimality}, that
$\Omega_\Zcal=\Xcal_\Phi$.

In Section~\ref{chap:Labbe:sec:Z2-rotations}
we proved that
$\Xcal_{\Pcal_\Zcal,R_\Zcal}=\shiftclosure{\Phi(\Xcal_{\Pcal_\Zcal,R_\Zcal})}$ using
induction of $\Z^2$-rotations.
In Exercise~\ref{exo:XPZRZ=XPhi}, we proved using the criterion 
given in Lemma~\ref{lem:criterion-for-minimality}, that
$\Xcal_{\Pcal_\Zcal,R_\Zcal}=\Xcal_\Phi$.
Thus, we have the equality
\[
\Omega_\Zcal
=\Xcal_\Phi
=\Xcal_{\Pcal_\Zcal,R_\Zcal}.\qedhere
\]
\end{proof}

Since $\Omega_\Zcal$ is a SFT,
it implies that
$\Pcal_\Zcal$ is a Markov partition
for the $\Z^2$-action $\R_\Zcal$ on $\torus^2$.
Markov partitions were originally defined for one-dimensional dynamical
systems $(\generictorus,\Z,R)$ and were extended to $\Z^d$-actions by automorphisms of
compact Abelian group in \cite{MR1632169}.
Following \cite{MR4213162},
we use the same terminology 
and extend the definition proposed in \cite[\S 6.5]{MR1369092}
for dynamical systems defined by higher-dimensional actions by rotations.

\begin{definition}\label{def:Markov}
A topological partition $\Pcal$ of $\generictorus$ is a \defn{Markov partition} for
$(\generictorus,\Z^2,R)$ if 
\begin{itemize}
    \item $\Pcal$ gives a symbolic representation of $(\generictorus,\Z^2,R)$ and 
    \item $\Xcal_{\Pcal,R}$ is a shift of finite type (SFT).
\end{itemize}
\end{definition}

The link between the subshifts
$\Omega_\Zcal$ and
$\Xcal_{\Pcal_\Zcal,R_\Zcal}$
can be explained directly without the $2$-dimensional morphism $\Phi$ by the existence of
a factor map from $(\Omega_\Zcal,\Z^2,\sigma)$ to $(\torus^2,\Z^2,R_\Zcal)$.
It turns out that the factor is also an isomorphism of strictly ergodic
measure-preserving dynamical systems.
We refer the reader to \cite{MR4213162} for more details.
Moreover, there exists a 4-to-2 cut and project scheme such that
the set of occurrences of any pattern in $\Omega_\Zcal$
is a model set, see Theorem 14.1 in \cite{MR4213162}.

\begin{exo}[label={exo:Markof-Partition}]
    Prove that the topological partition
$\Pcal_\Zcal$ of $\torus^2$ is a Markov partition for the dynamical system
$(\torus^2,\Z^2,R_\Zcal)$.
\end{exo}

\begin{exo}[label={exo:equality-of-substitutions}]
    Using SageMath,
    verify that the equalities $\beta_0=\alpha_0$, $\beta_1=\alpha_1$
    and $\beta_2=\alpha_2$ hold. 
    As a consequence, what is the relation between $\Omega_\Vcal$ and
    $\Xcal_{\Pcal_1,R_1}$?  What is the relation between $\Omega_\Wcal$ and
    $\Xcal_{\Pcal_2,R_2}$?
\end{exo}

\begin{exo}[label={exo:further-simplification}]
    In her PhD thesis \cite{lepsova_thesis_2024}, Jana Lepšová
    simplified the Wang shift $\Omega_\Ucal$ based on 19 Wang tiles
    to the Wang shift $\Omega_\Zcal$ based on 16 Wang tiles
    and proved that $\Omega_\Ucal$ and $\Omega_\Zcal$ are topologically conjugate.

    Try merging some more labels of the set of Wang tiles $\Zcal$.
    Is the resulting Wang shift still aperiodic?
\end{exo}

Jana did not find any more simplifications of $\Zcal$ during her PhD, but
it is still open to show that there is no further simplification.
More precisely, we may state the following question.

\begin{openquestion}
    Prove that there is no Wang shift based on fewer than 16 tiles
    that is topologically conjugate to $\Omega_\Ucal$ and $\Omega_\Zcal$?
\end{openquestion}

More broadly, it seems 16 is a lower bound for the number of tiles of a self-similar
aperiodic Wang shift. 
The Jeandel--Rao Wang shift based on 11 tiles is substitutive
but is not self-similar \cite{MR4226493}. 
Also, the Kari--Culik Wang shifts \cite{MR1417578,MR1417576} (based on 14 and
13 Wang tiles) have positive entropy thus are not self-similar \cite{MR3606059}.
Another aperiodic and self-similar set of Wang tiles was described by Ammann
and is based on 16 Wang tiles \cite{MR857454}. It was generalized to a family of sets of
aperiodic and self-similar Wang tiles involving the metallic mean numbers
\cite{labbe_metallic_I_2024,labbe_metallic_II_2024}. But 16 tiles is the minimum
number of tiles in this family defined from 16, 25, 36, 49, 64, $\ldots,
(n+3)^2, \ldots$ Wang tiles where $n\geq 1$.

\begin{openquestion}
    Does there exist a self-similar and aperiodic Wang shift
    based on strictly less than 16 tiles?
\end{openquestion}

The Ammann Wang shift based on 16 Wang tiles and the Wang shift $\Omega_\Zcal$
are not equivalent up to a bijection of the labels, but both are related to
golden mean. This raises the following question.

\begin{openquestion}
    Is $\Omega_\Zcal$ topologically conjugate to the Ammann Wang shift?
\end{openquestion}

The above question might be answered by the computation of set of slopes of
nonexpansive directions in the Wang shifts, which is a topological invariant
\cite{MR4730985}.
In \cite{MR4730985}, the slope of the four nonexpansive directions
for the minimal subshift within the Jeandel--Rao Wang shift were computed
from its associated polygonal partition.

\begin{openquestion}
    What are the slopes of nonexpansive directions in the Wang shift $\Omega_\Zcal$?
    in the Ammann Wang shift ?
\end{openquestion}

Another open question raised by the current work is the characterization of
Markov polygonal partition for toral $\Z^2$-rotations.

\begin{openquestion}
    Characterize the polygonal partitions $\Pcal$ of $\torus^2$
    and toral $\Z^2$-action $R$ defined by rotations on $\torus^2$
    for which $\Pcal$ is a Markov partition.
\end{openquestion}

\section{Solutions to exercises}\label{sec:exosolutions}

\begin{exosolution}{\ref{exo:subshift-easy}}
    There are two elements in the subshift $\Xcal_{\Lcal(x)}$.
\end{exosolution}

\begin{exosolution}{\ref{exo:easy-image-by-Phi}}
    We have
    \[
        \Phi
    \left(\begin{smallmatrix}
        (11) \\
    \end{smallmatrix}\right)
    = 
    \left(\begin{smallmatrix}
        6 & 1 \\
        12 & 10
    \end{smallmatrix}\right)
    ,\quad
    \Phi
    \left((\begin{smallmatrix}
        13 & 7 \\
    \end{smallmatrix})\right)
    =
    \left(\begin{smallmatrix}
        5 & 1 & 6 \\
        15 & 9 & 12
    \end{smallmatrix}\right)
    ,\quad
    \Phi
    \left((\begin{smallmatrix}
    6 & 1 \\
    11 & 8
    \end{smallmatrix})\right)
    =
    \left(\begin{smallmatrix}
        12 & 7 & 13 \\
        6 & 1 & 3 \\
        12 & 10 & 14
    \end{smallmatrix}\right).
    \]
    The image $\Phi
    \left((\begin{smallmatrix}
    6 \\
    10 \\
    \end{smallmatrix})\right)$
    is not defined because
    the image of 6 has width 2 while the image of 10 has width 1.
\end{exosolution}

\begin{exosolution}{\ref{exo:image-of-subshift-is-a-subshift}}
    Left to the reader.
\end{exosolution}

\begin{exosolution}{\ref{exo:prove-expansive}}
    For every letter $a\in\Zrange{15}$, the 2-dimensional
    word $\Phi^2(a)$ has height and width at least 2.
\end{exosolution}

\begin{exosolution}{\ref{exo:substitutive-shift-nonempty}}
The configurations
$x = \lim_{n\to\infty}\Phi^{2n}\left(\begin{smallmatrix}8&12\\1&6\end{smallmatrix}\right)$
    and
$y = \lim_{n\to\infty}\Phi^{2n+1}\left(\begin{smallmatrix}8&12\\1&6\end{smallmatrix}\right)$
both belong to $\Xcal_\Phi$. Therefore, $\Xcal_\Phi\neq\varnothing$.
\end{exosolution}

\begin{exosolution}{\ref{exo:Phi-sagemath}}
The following computes
$\Phi^{n}\left(\begin{smallmatrix}12\end{smallmatrix}\right)$
and
$\Phi^{n}\left(\begin{smallmatrix}8&12\\1&6\end{smallmatrix}\right)$
when $n=4$.
\begin{sagecommandlinetcb}
\begin{sagecommandline}
sage: image = Phi([[12]], order=4)
sage: image
[[12, 6, 12, 11, 6, 12, 6, 12],
[10, 1, 7, 8, 1, 10, 1, 7],
[14, 3, 13, 12, 6, 14, 3, 13],
[11, 6, 12, 15, 5, 11, 6, 12],
[8, 1, 10, 9, 1, 8, 1, 7],
[12, 6, 14, 12, 6, 12, 2, 13],
[10, 1, 7, 10, 1, 7, 0, 7],
[14, 3, 13, 14, 3, 13, 3, 13]]
sage: seed = [[1,8],[6,12]]  # using Cartesian-like coordinates
sage: image = Phi(seed, order=4)
\end{sagecommandline}
\end{sagecommandlinetcb}
The reader may increase the power to $n=5$ or $n=6$
and compare the output with Figure~\ref{fig:Phi-on-seed-12} and
Figure~\ref{fig:Phi-on-seed-8-12-1-6}.
\end{exosolution}

\begin{exosolution}{\ref{exo:HV-dominoes}}
The method \texttt{list\_dominoes} returns the list of $1\times 2$ or $2\times
1$ factors in the language of the associated substitutive shift:
\begin{sagecommandlinetcb}
\begin{sagecommandline}
sage: XPhi_2x1 = Phi.list_dominoes(direction='horizontal')
sage: sorted(XPhi_2x1)
[[[0], [3]], [[1], [2]], [[1], [3]], [[1], [6]], [[2], [0]], [[2], [4]],
[[3], [6]], [[4], [1]], [[5], [1]], [[6], [1]], [[6], [5]], [[7], [13]],
[[8], [12]], [[9], [11]], [[9], [12]], [[10], [14]], [[11], [8]], [[12],
[7]], [[12], [10]], [[12], [11]], [[12], [15]], [[13], [7]], [[13], [11]],
[[13], [12]], [[14], [7]], [[14], [11]], [[15], [9]]]
sage: XPhi_1x2 = Phi.list_dominoes(direction='vertical')
sage: sorted(XPhi_1x2)
[[[0, 7]], [[1, 7]], [[1, 8]], [[1, 10]], [[2, 13]], [[3, 13]], [[4, 11]],
[[5, 11]], [[6, 12]], [[6, 14]], [[7, 0]], [[7, 8]], [[7, 10]], [[8, 1]],
[[8, 9]], [[9, 1]], [[10, 1]], [[10, 9]], [[11, 4]], [[11, 6]], [[11, 15]],
[[12, 2]], [[12, 6]], [[12, 11]], [[12, 15]], [[13, 3]], [[13, 12]], [[13,
14]], [[14, 3]], [[14, 12]], [[15, 5]]]
\end{sagecommandline}
\end{sagecommandlinetcb}
\end{exosolution}

\begin{exosolution}{\ref{exo:belongs-in-the-set}}
    The letter $9$ appears only in the image of $13$ by $\Phi$.
    Thus
    $\left(\begin{smallmatrix}1\\9\end{smallmatrix}\right)$
    must be obtained from an image of the letter $13$.
    The vertical domino $\left(\begin{smallmatrix}2\\12\end{smallmatrix}\right)$
    must be obtained from the image of $10$ or of $15$.
    Therefore
    $\left(\begin{smallmatrix}1 & 2\\9 & 12\end{smallmatrix}\right)$
        is obtained from the horizontal dominoes
    $\left(\begin{smallmatrix}13 & 10\end{smallmatrix}\right)$
        or
    $\left(\begin{smallmatrix}13 & 15\end{smallmatrix}\right)$.
        But, from Exercise~\ref{exo:HV-dominoes}, the only horizontal dominoes
        starting with $13$ in the language of $\Phi$ are
    $\left(\begin{smallmatrix}13 & 7\end{smallmatrix}\right)$,
    $\left(\begin{smallmatrix}13 & 11\end{smallmatrix}\right)$ or
    $\left(\begin{smallmatrix}13 & 12\end{smallmatrix}\right)$.
        Therefore
    $\left(\begin{smallmatrix}1 & 2\\9 & 12\end{smallmatrix}\right)\notin\Lcal_\Phi$.
\end{exosolution}

\begin{exosolution}{\ref{exo:list-45-2x2-elements-in-substitutive-shift}}
\begin{sagecommandlinetcb}
\begin{sagecommandline}
sage: sorted(Phi.list_2x2_factors())
[[[0, 7], [3, 13]], [[1, 7], [2, 13]], [[1, 7], [3, 13]], [[1, 8], [6, 12]], [[1, 10], [6, 14]], [[2, 13], [0, 7]], [[2, 13], [4, 11]], [[3, 13], [6, 12]], [[4, 11], [1, 8]], [[5, 11], [1, 8]], [[6, 12], [1, 7]], [[6, 12], [1, 10]], [[6, 12], [5, 11]], [[6, 14], [1, 7]], [[6, 14], [5, 11]], [[7, 0], [13, 3]], [[7, 8], [13, 12]], [[7, 10], [13, 14]], [[8, 1], [12, 2]], [[8, 1], [12, 6]], [[8, 9], [12, 11]], [[9, 1], [11, 6]], [[9, 1], [12, 6]], [[10, 1], [14, 3]], [[10, 9], [14, 12]], [[11, 4], [8, 1]], [[11, 6], [8, 1]], [[11, 15], [8, 9]], [[12, 2], [7, 0]], [[12, 2], [11, 4]], [[12, 6], [10, 1]], [[12, 6], [15, 5]], [[12, 11], [7, 8]], [[12, 15], [10, 9]], [[13, 3], [11, 6]], [[13, 3], [12, 6]], [[13, 12], [7, 10]], [[13, 12], [11, 15]], [[13, 12], [12, 11]], [[13, 12], [12, 15]], [[13, 14], [12, 11]], [[14, 3], [11, 6]], [[14, 12], [7, 10]], [[14, 12], [11, 15]], [[15, 5], [9, 1]]]
\end{sagecommandline}
\end{sagecommandlinetcb}
Note that the $2\times2$ words above are written in Cartesian coordinates.
    For example, the list of lists \texttt{[[0, 7], [3, 13]]}
    denotes the 2-dimensional word
    $\left(\begin{smallmatrix}7 & 13\\ 0 & 3\end{smallmatrix}\right)$.
\end{exosolution}

\begin{exosolution}{\ref{exo:substitutive-shift-8-periodic-points}}
There are 96 periodic points of $\Phi$, i.e. the configurations
$x\in\Acal^{\Z^2}$ such that $\Phi^k(x)=x$ for some $k\geq1$.
\begin{sagecommandlinetcb}
\begin{sagecommandline}
sage: seeds = sorted(flatten(Phi.prolongable_seeds_list()))
sage: len(seeds)
96
sage: seeds[0]
[ 8 11]
[ 1  4]
\end{sagecommandline}
\end{sagecommandlinetcb}
These seeds do not all belong to the language of the substitution.
If all seeds belong to the language of the substitution (computed
from the letters), then there exists a unique subshift which is
self-similar with respect to the substitution. In this example, it
is not true:
\begin{sagecommandlinetcb}
\begin{sagecommandline}
sage: seeds_as_lists_of_lists = [[list(col[::-1]) for col in m.columns()] for m in seeds]
sage: XPhi_2x2 = Phi.list_2x2_factors()
sage: all(seed in XPhi_2x2 for seed in seeds_as_lists_of_lists)
False
\end{sagecommandline}
\end{sagecommandlinetcb}
Indeed, only 8 of the 96 seeds belong to the language of the
substitution:
\begin{sagecommandlinetcb}
\begin{sagecommandline}
sage: [seed for seed in seeds_as_lists_of_lists if seed in XPhi_2x2]
[[[1, 8], [6, 12]],
[[7, 8], [13, 12]],
[[8, 9], [12, 11]],
[[10, 9], [14, 12]],
[[2, 13], [4, 11]],
[[3, 13], [6, 12]],
[[6, 14], [5, 11]],
[[13, 14], [12, 11]]]
\end{sagecommandline}
\end{sagecommandlinetcb}
These eight seeds 
\[
    \left\{
        \left(\begin{smallmatrix} 8 & 12  \\ 1 & 6   \end{smallmatrix}\right), 
        \left(\begin{smallmatrix} 8 & 12  \\ 7 & 13  \end{smallmatrix}\right),  
        \left(\begin{smallmatrix} 9 & 11  \\ 8 & 12  \end{smallmatrix}\right),  
        \left(\begin{smallmatrix} 9 & 12  \\ 10 & 14 \end{smallmatrix}\right), 
        \left(\begin{smallmatrix} 13 & 11 \\ 2 & 4   \end{smallmatrix}\right),  
        \left(\begin{smallmatrix} 13 & 12 \\ 3 & 6   \end{smallmatrix}\right),  
        \left(\begin{smallmatrix} 14 & 11 \\ 6 & 5   \end{smallmatrix}\right),  
        \left(\begin{smallmatrix} 14 & 11 \\ 13 & 12 \end{smallmatrix}\right)
    \right\}
\]
give rise to eight different configurations which are fixed by some power of
the substitution and are uniformly recurrent.
\end{exosolution}

\begin{exosolution}{\ref{exo:representations}}
    Left to the reader.
\end{exosolution}

\begin{exosolution}{\ref{exo:more-one-centered-representations}}
    Left to the reader.
\end{exosolution}

\begin{exosolution}{\ref{exo:prove-recognizable-aperiodic}}
    (1)
    The notion of markers is developed in Section~\ref{subsection:markers} to
    help prove that a substitution is recognizable.
    It can be used here for $\Phi$
    using the decomposition of $\Phi$ as product of simpler one-dimensional substitutions
    computed in Section~\ref{chap:Labbe:sec:wang-shifts}
    and in Section~\ref{chap:Labbe:sec:Z2-rotations}.

    (2)
    Aperiodicity of $\Xcal_\Phi$ follows from
    the recognizability of $\Phi$ in $\Xcal_\Phi$.
\end{exosolution}

\begin{exosolution}{\ref{exo:prove-primitive-minimal}}
The $2$-dimensional substitution $\Phi$ is primitive, because its incidence
matrix is primitive:
\begin{sagecommandlinetcb}
\begin{sagecommandline}
sage: matrix(Phi).is_primitive()
True
\end{sagecommandline}
\end{sagecommandlinetcb}
From Exercise~\ref{exo:prove-expansive} and
Lemma~\ref{lem:xcal_omega_minimal}, $\Xcal_\Phi$ is minimal.
\end{exosolution}

\begin{exosolution}{\ref{exo:recurrent-vertices}}
    The horizontal dominoes in $\RecurrentVertices(G^{2\times 1}_\Phi)$ are
\begin{sagecommandlinetcb}
\begin{sagecommandline}
sage: Gh = Phi.periodic_horizontal_domino_seeds_graph(clean_sources=True)
sage: Gh
Looped digraph on 36 vertices
sage: sorted(Gh.vertices())
[(1, 2), (1, 4), (1, 5), (1, 6), (2, 4), (2, 5), (2, 6), (3, 2), (3, 4), (3, 5), (3, 6), (6, 4), (6, 5), (6, 6), (7, 12), (7, 13), (7, 14), (8, 11), (8, 12), (8, 13), (8, 14), (9, 11), (9, 12), (9, 15), (10, 12), (10, 13), (10, 14), (12, 11), (12, 15), (13, 11), (13, 12), (13, 13), (13, 14), (14, 11), (14, 12), (14, 15)]
\end{sagecommandline}
\end{sagecommandlinetcb}
    The vertical dominoes in $\RecurrentVertices(G^{1\times 2}_\Phi)$ are
\begin{sagecommandlinetcb}
\begin{sagecommandline}
sage: Gv = Phi.periodic_vertical_domino_seeds_graph(clean_sources=True)
sage: Gv
Looped digraph on 36 vertices
sage: sorted(Gv.vertices())
[(1, 7), (1, 8), (1, 9), (1, 10), (2, 13), (2, 14), (3, 13), (3, 14), (4, 11), (4, 12), (5, 11), (5, 12), (5, 15), (6, 11), (6, 12), (6, 13), (6, 14), (7, 7), (7, 8), (7, 9), (7, 10), (8, 8), (8, 9), (10, 8), (10, 9), (11, 11), (11, 15), (12, 11), (12, 12), (12, 15), (13, 11), (13, 12), (13, 13), (13, 14), (14, 11), (14, 12)]
\end{sagecommandline}
\end{sagecommandlinetcb}
    There are 96 different $2\times 2$ patterns in
    $\RecurrentVertices(G^{2\times 2}_\Phi)$. The first of them is shown below:
\begin{sagecommandlinetcb}
\begin{sagecommandline}
sage: G = Phi.prolongable_seeds_graph(clean_sources=True)
sage: G
Looped digraph on 96 vertices
sage: sorted(G.vertices())[0]
[ 8 11]
[ 1  4]
\end{sagecommandline}
\end{sagecommandlinetcb}
The limit $x=\lim_{n\to\infty}\Phi^{2n}
\left(\begin{smallmatrix}8&11\\1&4\end{smallmatrix}\right)$
defines a configuration of $\Z^2$ which is fixed by
the $2$-dimensional morphism $\Phi^2$.
The topological closure 
    $X_x=\topologicalclosure{\{\sigma^\bn(x)\colon \bn\in\Z^2\}}$
    of the orbit under the $\Z^2$-shift $\sigma$ of the configuration $x$
is a subshift which is self-similar with respect to $\Phi$.
From Lemma~\ref{lem:substitutive-contains-self-similar-part}, we have
$\Xcal_\Phi\subseteq X_x$. 
    But $\Xcal_\Phi\neq X_x$, because $x\notin\Xcal_\Phi$.
    Indeed the central pattern
$\left(\begin{smallmatrix}8&11\\1&4\end{smallmatrix}\right)$
    of the configuration $x$
    is not in the language of $\Xcal_\Phi$, see
Exercise~\ref{exo:substitutive-shift-8-periodic-points}.
Also the configuration $x$ is not uniformly recurrent because
it contains only one occurrence of the pattern
$\left(\begin{smallmatrix}8&11\\1&4\end{smallmatrix}\right)$.
\end{exosolution}

\begin{exosolution}{\ref{exo:marker-e2-in-XPhi}}
    Let $\Acal=\Zrange{15}$ and $M=\{0,1,2,3,4,5,6\}\subset\Acal$.
    We have $H\cap((\Acal\setminus M)\odot^1 M)=\varnothing$
    and     $H\cap( M\odot^1(\Acal\setminus M))=\varnothing$.
    Also $M\odot^2M\cap V=\varnothing$.
    Therefore, using Lemma~\ref{lem:criteria-markers},
    $M$ is a subset of markers for the direction $\be_2$ in
    the subshift $\Xcal_\Phi$.
\end{exosolution}

\begin{exosolution}{\ref{exo:marker-e1-in-XPhi}}
    Let $\Acal=\Zrange{15}$ and $M=\{0, 1, 7, 8, 9, 10\}\subset\Acal$.
    We have $V\cap((\Acal\setminus M)\odot^2 M)=\varnothing$
    and     $V\cap( M\odot^2(\Acal\setminus M))=\varnothing$.
    Also $M\odot^1M\cap H=\varnothing$.
    Therefore, using Lemma~\ref{lem:criteria-markers},
    $M$ is a subset of markers for the direction $\be_1$ in
    the subshift $\Xcal_\Phi$.
\end{exosolution}

\begin{exosolution}{\ref{exo:markers-find-a-subshift}}
    Such a substitution can be found in
    the proof of Theorem~\ref{thm:main-theoremA}
    in Section~\ref{subsection:self-similarity-wang-shift}.
\end{exosolution}

\begin{exosolution}{\ref{exo:desubstitute-markers-e1}}
The following 2-dimensional morphism $\xi:\Ccal\to\Zrange{15}^{*^2}$,
where $\Ccal=\Zrange{15}$:
    \[
        \xi:\left\{
            {\arraycolsep=1.4pt
            \begin{array}{lllll}
                0\mapsto \left(2\right)
                ,&
                1\mapsto \left(3\right)
                ,&
                2\mapsto \left(6\right)
                ,&
                3\mapsto \left(12\right)
                ,&
                4\mapsto \left(13\right)
                ,\\
                5\mapsto \left(14\right)
                ,&
                6\mapsto \left(2,\,0\right)
                ,&
                7\mapsto \left(4,\,1\right)
                ,&
                8\mapsto \left(5,\,1\right)
                ,&
                9\mapsto \left(6,\,1\right)
                ,\\
                10\mapsto \left(11,\,8\right)
                ,&
                11\mapsto \left(12,\,7\right)
                ,&
                12\mapsto \left(12,\,10\right)
                ,&
                13\mapsto \left(13,\,7\right)
                ,&
                14\mapsto \left(14,\,7\right)
                ,\\
                15\mapsto \left(15,\,9\right)
                .
            \end{array}
            }
    \right.
    \] satisfies the requirements
    with $M=\{0, 1, 7, 8, 9, 10\}$.
\end{exosolution}

\begin{exosolution}{\ref{exo:valid-7x7}}
First we define the set $\Zcal$ of Wang tiles in SageMath:
\begin{sagecommandlinetcb}
\begin{sagecommandline}
sage: from slabbe import WangTileSet
sage: tiles = ["DOJO", "DOHL", "JMDP", "DMDK", "HPJP", "HPHN", "HKDP", "BOIO", 
....:   "ILEO", "ILCL", "ALIO", "EPIP", "IPIK", "IKBM", "IKAK", "CNIP"]
sage: Z = WangTileSet([tuple(tile) for tile in tiles])
\end{sagecommandline}
\end{sagecommandlinetcb}
Then using the dancing links solver:
\begin{sagecommandlinetcb}
\begin{sagecommandline}
sage: tiling = Z.solver(7,7).solve(solver="dancing_links")
sage: tiling.table()    # Cartesian-like coordinates
[[0, 7, 8, 1, 8, 9, 1],
 [2, 13, 12, 6, 12, 11, 6],
 [0, 7, 10, 1, 7, 8, 1],
 [3, 13, 14, 3, 13, 12, 2],
 [6, 12, 11, 6, 12, 11, 4],
 [1, 7, 8, 1, 7, 8, 1],
 [2, 13, 12, 2, 13, 12, 2]]
\end{sagecommandline}
\end{sagecommandlinetcb}
\end{exosolution}

\begin{exosolution}{\ref{exo:HVdominoes-in-OmegaU}}
\begin{sagecommandlinetcb}
\begin{sagecommandline}
sage: Z_2x1 = [t.table() for t in Z.tilings_with_surrounding(2,1,radius=2, solver="dancing_links")]
sage: sorted(Z_2x1)
[[[0], [3]], [[1], [2]], [[1], [3]], [[1], [6]], [[2], [0]], [[2], [4]], [[3], [6]], [[4], [1]], [[5], [1]], [[6], [1]], [[6], [5]], [[7], [13]], [[8], [12]], [[9], [11]], [[9], [12]], [[10], [14]], [[11], [8]], [[12], [7]], [[12], [10]], [[12], [11]], [[12], [15]], [[13], [7]], [[13], [11]], [[13], [12]], [[14], [7]], [[14], [11]], [[15], [9]]]
sage: Z_1x2 = [t.table() for t in Z.tilings_with_surrounding(1,2,radius=2, solver="dancing_links")]
sage: sorted(Z_1x2)
[[[0, 7]], [[1, 7]], [[1, 8]], [[1, 10]], [[2, 13]], [[3, 13]], [[4, 11]], [[5, 11]], [[6, 12]], [[6, 14]], [[7, 0]], [[7, 8]], [[7, 10]], [[8, 1]], [[8, 9]], [[9, 1]], [[10, 1]], [[10, 9]], [[11, 4]], [[11, 6]], [[11, 15]], [[12, 2]], [[12, 6]], [[12, 11]], [[12, 15]], [[13, 3]], [[13, 12]], [[13, 14]], [[14, 3]], [[14, 12]], [[15, 5]]]
\end{sagecommandline}
\end{sagecommandlinetcb}
\end{exosolution}

\begin{exosolution}{\ref{exo:markers-for-Zcal}}
\begin{sagecommandlinetcb}
\begin{sagecommandline}
sage: Z.find_markers(i=2, radius=2, solver="dancing_links")
[[0, 1, 2, 3, 4, 5, 6]]
\end{sagecommandline}
\end{sagecommandlinetcb}
\end{exosolution}

\begin{exosolution}{\ref{exo:fusion-of-tiles}}
We compute the set of fusion tiles:
\begin{sagecommandlinetcb}
\begin{sagecommandline}
sage: M = [0, 1, 2, 3, 4, 5, 6]
sage: dominoes_M = sorted((a,b) for [[a,b]] in Z_1x2 if b in M)
sage: from slabbe.wang_tiles import fusion
sage: new_tiles = [fusion(Z[a],Z[b],2) for (a,b) in dominoes_M]
sage: new_tiles
[('BD', 'O', 'IJ', 'O'), ('ID', 'O', 'EH', 'O'), ('ID', 'O', 'CH', 'L'), ('AD', 'O', 'IH', 'O'), ('EH', 'P', 'IJ', 'P'), ('EH', 'K', 'ID', 'P'), ('IJ', 'M', 'ID', 'K'), ('IH', 'K', 'ID', 'K'), ('ID', 'M', 'BD', 'M'), ('ID', 'M', 'AD', 'K'), ('CH', 'P', 'IH', 'P')]
sage: tikz = WangTileSet(new_tiles).tikz(id=None)
\end{sagecommandline}
\end{sagecommandlinetcb}
These tiles represent the result of merging every marker tile with tiles that
may appear below of them:
\begin{center}
    \sageplot[width=.75\linewidth][pdf]{tikz}
\end{center}
\end{exosolution}

\begin{exosolution}{\ref{exo:FindSubstitution}}
\begin{sagecommandlinetcb}
\begin{sagecommandline}
sage: M = [0, 1, 2, 3, 4, 5, 6]
sage: V,alpha0 = Z.find_substitution(M=M,i=2,radius=2,solver="dancing_links")
sage: V
Wang tile set of cardinality 18
\end{sagecommandline}
\end{sagecommandlinetcb}
    \[
        \alpha_0:
        \left\{
            \footnotesize
        \sage{alpha0}
        \right.
    \]
\[
    \Vcal=
    \left\{
        \raisebox{-10mm}{
            \sageplot[width=.75\linewidth][pdf]{V.tikz()}
        }
    \right\}
\]
\end{exosolution}

    \begin{exosolution}{\ref{exo:OmegaU=XPhi}}
        In Theorem~\ref{thm:main-theoremA}, we proved that the Wang shift
        $\Omega_\Zcal\subset\Zrange{15}^{\Z^2}$ is self-similar satisfying
        $\Omega_\Zcal=\shiftclosure{\Phi(\Omega_\Zcal)}$.
        Thus, we may use the criterion given in Lemma~\ref{lem:criterion-for-minimality}.
        More precisely, in this case, we may show that
        $\Lcal_s(\Omega_\Zcal)\subset\Lcal(\Xcal_\Phi)$ for every shape $s\in\{2\times2,
        2\times1, 1\times2\}$.
        This of course implies that
        $\Lcal(\Omega_\Zcal)\cap\RecurrentVertices(G^s_\Phi)\subset\Lcal(\Xcal_\Phi)$
              for every shape $s\in\{2\times2, 2\times1, 1\times2\}$.

        In Exercise~\ref{exo:HV-dominoes} and
        Exercise~\ref{exo:HVdominoes-in-OmegaU}, we observed that
        $\Lcal_{1\times 2}(\Omega_\Zcal)=\Lcal_{1\times 2}(\Xcal_\Phi)$
        and
        $\Lcal_{2\times 1}(\Omega_\Zcal)=\Lcal_{2\times 1}(\Xcal_\Phi)$:
\begin{sagecommandlinetcb}
\begin{sagecommandline}
sage: sorted(XPhi_2x1) == sorted(Z_2x1)
True
sage: sorted(XPhi_1x2) == sorted(Z_1x2)
True
\end{sagecommandline}
\end{sagecommandlinetcb}
In Exercise~\ref{exo:list-45-2x2-elements-in-substitutive-shift}, we computed
the patterns of shape $2\times 2$ in the language of the 2-dimensional
substitution $\Phi$.
We compute the set $\Lcal_{2\times 2}(\Omega_\Zcal)$
below and we observe that
$\Lcal_{2\times 2}(\Omega_\Zcal)=\Lcal_{2\times 2}(\Xcal_\Phi)$.
\begin{sagecommandlinetcb}
\begin{sagecommandline}
sage: tilings = Z.tilings_with_surrounding(2,2,radius=2, solver="dancing_links")
sage: Z_2x2 = [tiling.table() for tiling in tilings]
sage: len(Z_2x2)
45
sage: len(XPhi_2x2)
45
sage: sorted(XPhi_2x2) == sorted(Z_2x2)
True
\end{sagecommandline}
\end{sagecommandlinetcb}
We conclude from Lemma~\ref{lem:criterion-for-minimality},
that $\Omega_\Zcal$ is minimal and $\Omega_\Zcal=\Xcal_\Phi$.
\end{exosolution}

    \begin{exosolution}{\ref{exo:markers-e1-and-then-e2}}
        Left to the reader.
    \end{exosolution}

\begin{exosolution}{\ref{exo:merge-the-atoms}}
\begin{sagecommandlinetcb}
\begin{sagecommandline}
sage: from slabbe.arXiv_1903_06137 import self_similar_19_atoms_partition
sage: PU = self_similar_19_atoms_partition()
sage: merge_dict = {0:0, 1:1, 2:2, 3:3, 4:4, 5:5, 6:6, 7:6, 8:7, 9:8, 10:9, 
....:               11:10, 12:11, 13:11, 14:12, 15:12, 16:13, 17:14, 18:15}
sage: PZ = PU.merge_atoms(merge_dict)
sage: graphics_array([PU.plot(), PZ.plot()])
Graphics Array of size 1 x 2
\end{sagecommandline}
\end{sagecommandlinetcb}
\end{exosolution}

\begin{exosolution}{\ref{exo:word-shape-68-8-10}}
\begin{sagecommandlinetcb}
\begin{sagecommandline}
sage: z = polygen(QQ, "z")
sage: K.<phi> = NumberField(z**2-z-1, "phi", embedding=RR(1.6))
sage: from slabbe import PolyhedronExchangeTransformation as PET
sage: Gamma0 = matrix.column([(1,0), (0,1)])
sage: RZe1 = PET.toral_translation(Gamma0, vector((phi^-2,0)))
sage: RZe2 = PET.toral_translation(Gamma0, vector((0,phi^-2)))
sage: from slabbe.arXiv_1903_06137 import self_similar_19_atoms_partition
sage: PU = self_similar_19_atoms_partition()
sage: merge_dict = {0:0, 1:1, 2:2, 3:3, 4:4, 5:5, 6:6, 7:6, 8:7, 9:8, 10:9, 
....:               11:10, 12:11, 13:11, 14:12, 15:12, 16:13, 17:14, 18:15}
sage: PZ = PU.merge_atoms(merge_dict)
sage: from slabbe import PETsCoding
sage: X_PZ_RZ = PETsCoding((RZe1,RZe2), PZ)
sage: pattern = X_PZ_RZ.pattern((.1357+1/phi, .2938+1/phi), (8,10))
sage: m8_10 = matrix.column(col[::-1] for col in pattern)
\end{sagecommandline}
\end{sagecommandlinetcb}
The pattern of shape $8\times 10$ is
    \[
        \sage{m8_10}
    \]
\end{exosolution}

\begin{exosolution}{\ref{exo:PETs-easy}}
    Left to the reader.
\end{exosolution}

\begin{exosolution}{\ref{exo:symbolic-representation}}
    The dynamical system $(\torus^2,\Z^2,R_\Zcal)$ is minimal.
    Every atom of the partition $\Pcal_\Zcal$ is invariant only under the trivial
    translation in $\torus^2$.
    Thus, from Lemma~\ref{lem:symbolic-representation},
    $\Pcal_\Zcal$ gives a symbolic representation of $(\torus^2,\Z^2,R_\Zcal)$.
\end{exosolution}

\begin{exosolution}{\ref{exo:symbolic-ds-minimal-aperiodic}}
    It follows from Lemma~\ref{lem:minimal-aperiodic}.
\end{exosolution}

\begin{exosolution}{\ref{exo:PET-on-triangle-nontrivial}}
    We compute the induced map $\widehat{T}|_W$ in SageMath:
\begin{sagecommandlinetcb}
\begin{sagecommandline}
sage: from slabbe import PolyhedronExchangeTransformation as PET
sage: lattice_base = matrix.column([(1,0), (0,1)])
sage: translation = vector((1/3, 1/2))
sage: T = PET.toral_translation(lattice_base, translation)
sage: ieq = [1, -1, -1]     # inequality 0 <= 1 - x - y, that is, x + y <= 1
sage: induced_map,substitution = T.induced_transformation(ieq)
sage: substitution
{0: [0],
 1: [1],
 2: [2],
 3: [0, 1],
 4: [0, 3],
 5: [0, 1, 2],
 6: [0, 1, 2, 1],
 7: [0, 1, 2, 1, 0, 3]}
sage: induced_map.partition().plot()
Graphics object consisting of 53 graphics primitives
sage: induced_map.image_partition().plot()
Graphics object consisting of 53 graphics primitives
\end{sagecommandline}
\end{sagecommandlinetcb}
    \begin{center}
        \begin{tabular}{cc}
            domain partition & image partition\\
        \sageplot[width=.45\linewidth][pdf]{induced_map.partition().plot()}
        & \sageplot[width=.45\linewidth][pdf]{induced_map.image_partition().plot()}
        \end{tabular}
    \end{center}
\end{exosolution}

\begin{exosolution}{\ref{exo:induction-z2-action}}
    The first two questions are left to the reader.
    We compute the induced maps in SageMath:
\begin{sagecommandlinetcb}
\begin{sagecommandline}
sage: z = polygen(QQ, "z")
sage: K.<phi> = NumberField(z**2-z-1, "phi", embedding=RR(1.6))
sage: from slabbe import PolyhedronExchangeTransformation as PET
sage: Gamma0 = matrix.column([(1,0), (0,1)])
sage: RZe1 = PET.toral_translation(Gamma0, vector((phi^-2,0)))
sage: RZe2 = PET.toral_translation(Gamma0, vector((0,phi^-2)))
sage: y_ineq = [phi^-1, 0, -1] # y <= phi^-1: the window W_0
sage: R1e1,_ = RZe1.induced_transformation(y_ineq)
sage: R1e2,_ = RZe2.induced_transformation(y_ineq)
sage: R1e1
Polyhedron Exchange Transformation of
Polyhedron partition of 2 atoms with 2 letters
with translations {0: (-phi + 2, 0), 1: (-phi + 1, 0)}
sage: R1e2
Polyhedron Exchange Transformation of
Polyhedron partition of 2 atoms with 2 letters
with translations {0: (0, -phi + 2), 1: (0, -2*phi + 3)}
\end{sagecommandline}
\end{sagecommandlinetcb}
    \begin{center}
        \begin{tabular}{cc}
            \texttt{R1e1.plot()} & \texttt{R1e2.plot()}\\
        \sageplot[width=.45\linewidth][pdf]{R1e1.plot()}
        &
        \sageplot[width=.45\linewidth][pdf]{R1e2.plot()}
        \end{tabular}
    \end{center}
\end{exosolution}

\begin{exosolution}{\ref{exo:induced-partition-on-triangle}}
    We compute the induced substitution and the induced partition
    $\widehat{\Pcal}|_W$ in SageMath:
\begin{sagecommandlinetcb}
\begin{sagecommandline}
sage: from slabbe import PolyhedronExchangeTransformation as PET
sage: from slabbe import PolyhedronPartition
sage: lattice_base = matrix.column([(1,0), (0,1)])
sage: translation = vector((1/3, 1/2))
sage: T = PET.toral_translation(lattice_base, translation)
sage: square = polytopes.hypercube(2, intervals='zero_one')
sage: P = PolyhedronPartition([square]).refine_by_hyperplane([0,1,-1])
sage: ieq = [1, -1, -1]     # inequality 0 <= 1 - x - y, that is, x + y <= 1
sage: induced_partition,substitution = T.induced_partition(ieq, P)
sage: substitution
{0: [0],
 1: [1],
 2: [0, 0],
 3: [0, 1],
 4: [1, 1],
 5: [0, 0, 0],
 6: [0, 1, 0],
 7: [1, 1, 0],
 8: [0, 1, 0, 1],
 9: [1, 1, 0, 1],
 10: [1, 1, 0, 1, 0, 0],
 11: [1, 1, 0, 1, 0, 1]}
\end{sagecommandline}
\end{sagecommandlinetcb}
    \begin{center}
        \begin{tabular}{cc}
            the starting partition & the induced partition\\
            \texttt{P.plot()} & \texttt{induced\_partition.plot()}\\
        \sageplot[width=.45\linewidth][pdf]{P.plot()}
        & \sageplot[width=.45\linewidth][pdf]{induced_partition.plot()}
        \end{tabular}
    \end{center}
\end{exosolution}

\begin{exosolution}{\ref{exo:XPZRZ-domino-language}}
\begin{sagecommandlinetcb}
\begin{sagecommandline}
sage: z = polygen(QQ, "z")
sage: K.<phi> = NumberField(z**2-z-1, "phi", embedding=RR(1.6))
sage: from slabbe import PolyhedronExchangeTransformation as PET
sage: Gamma0 = matrix.column([(1,0), (0,1)])
sage: RZe1 = PET.toral_translation(Gamma0, vector((phi^-2,0)))
sage: RZe2 = PET.toral_translation(Gamma0, vector((0,phi^-2)))
sage: from slabbe.arXiv_1903_06137 import self_similar_19_atoms_partition
sage: PU = self_similar_19_atoms_partition()
sage: merge_dict = {0:0, 1:1, 2:2, 3:3, 4:4, 5:5, 6:6, 7:6, 8:7, 9:8, 10:9, 
....:               11:10, 12:11, 13:11, 14:12, 15:12, 16:13, 17:14, 18:15}
sage: PZ = PU.merge_atoms(merge_dict)
sage: from slabbe import PETsCoding
sage: X_PZ_RZ = PETsCoding((RZe1,RZe2), PZ)
sage: _,d22 = X_PZ_RZ.partition_for_patterns((2,2))
sage: XPZRZ_2x2 = [[list(col) for col in v] for v in d22.values()]
sage: XPZRZ_2x2
[[[0, 7], [3, 13]], [[1, 7], [2, 13]], [[1, 7], [3, 13]], [[1, 8], [6, 12]], [[1, 10], [6, 14]], [[2, 13], [0, 7]], [[2, 13], [4, 11]], [[3, 13], [6, 12]], [[4, 11], [1, 8]], [[5, 11], [1, 8]], [[6, 12], [1, 7]], [[6, 12], [1, 10]], [[6, 12], [5, 11]], [[6, 14], [1, 7]], [[6, 14], [5, 11]], [[7, 0], [13, 3]], [[7, 8], [13, 12]], [[7, 10], [13, 14]], [[8, 1], [12, 2]], [[8, 1], [12, 6]], [[8, 9], [12, 11]], [[9, 1], [11, 6]], [[9, 1], [12, 6]], [[10, 1], [14, 3]], [[10, 9], [14, 12]], [[11, 4], [8, 1]], [[11, 6], [8, 1]], [[11, 15], [8, 9]], [[12, 2], [7, 0]], [[12, 2], [11, 4]], [[12, 6], [10, 1]], [[12, 6], [15, 5]], [[12, 11], [7, 8]], [[12, 15], [10, 9]], [[13, 3], [11, 6]], [[13, 3], [12, 6]], [[13, 12], [7, 10]], [[13, 12], [11, 15]], [[13, 12], [12, 11]], [[13, 12], [12, 15]], [[13, 14], [12, 11]], [[14, 3], [11, 6]], [[14, 12], [7, 10]], [[14, 12], [11, 15]], [[15, 5], [9, 1]]]
sage: _,d12 = X_PZ_RZ.partition_for_patterns((1,2))
sage: XPZRZ_1x2 = [[list(col) for col in v] for v in d12.values()]
sage: XPZRZ_1x2
[[[0, 7]], [[1, 7]], [[1, 8]], [[1, 10]], [[2, 13]], [[3, 13]], [[4, 11]], [[5, 11]], [[6, 12]], [[6, 14]], [[7, 0]], [[7, 8]], [[7, 10]], [[8, 1]], [[8, 9]], [[9, 1]], [[10, 1]], [[10, 9]], [[11, 4]], [[11, 6]], [[11, 15]], [[12, 2]], [[12, 6]], [[12, 11]], [[12, 15]], [[13, 3]], [[13, 12]], [[13, 14]], [[14, 3]], [[14, 12]], [[15, 5]]]
sage: _,d21 = X_PZ_RZ.partition_for_patterns((2,1))
sage: XPZRZ_2x1 = [[list(col) for col in v] for v in d21.values()]
sage: XPZRZ_2x1
[[[0], [3]], [[1], [2]], [[1], [3]], [[1], [6]], [[2], [0]], [[2], [4]], [[3], [6]], [[4], [1]], [[5], [1]], [[6], [1]], [[6], [5]], [[7], [13]], [[8], [12]], [[9], [11]], [[9], [12]], [[10], [14]], [[11], [8]], [[12], [7]], [[12], [10]], [[12], [11]], [[12], [15]], [[13], [7]], [[13], [11]], [[13], [12]], [[14], [7]], [[14], [11]], [[15], [9]]]
\end{sagecommandline}
\end{sagecommandlinetcb}
\end{exosolution}

    \begin{exosolution}{\ref{exo:XPZRZ=XPhi}}
        In Theorem~\ref{thm:main-theoremB}, we proved that the symbolic dynamical system
        $\Xcal_{\Pcal_\Zcal,R_\Zcal}\subset\Zrange{15}^{\Z^2}$ is self-similar satisfying
        $\Xcal_{\Pcal_\Zcal,R_\Zcal}=\shiftclosure{\Phi(\Xcal_{\Pcal_\Zcal,R_\Zcal})}$.
        Thus, we may use the criterion given in Lemma~\ref{lem:criterion-for-minimality}.

        In Exercise~\ref{exo:XPZRZ-domino-language}, 
        we computed
        $\Lcal_{1\times 2}(\Xcal_{\Pcal_\Zcal,R_\Zcal})$,
        $\Lcal_{2\times 1}(\Xcal_{\Pcal_\Zcal,R_\Zcal})$ and
        $\Lcal_{2\times 2}(\Xcal_{\Pcal_\Zcal,R_\Zcal})$.

        Comparing with 
        Exercise~\ref{exo:HV-dominoes} and
        Exercise~\ref{exo:list-45-2x2-elements-in-substitutive-shift},
        we observed that
        \begin{align*}
            \Lcal_{1\times 2}(\Xcal_{\Pcal_\Zcal,R_\Zcal})&=\Lcal_{1\times 2}(\Xcal_\Phi),\\
            \Lcal_{2\times 1}(\Xcal_{\Pcal_\Zcal,R_\Zcal})&=\Lcal_{2\times 1}(\Xcal_\Phi),\\
            \Lcal_{2\times 2}(\Xcal_{\Pcal_\Zcal,R_\Zcal})&=\Lcal_{2\times 2}(\Xcal_\Phi).
        \end{align*}
\begin{sagecommandlinetcb}
\begin{sagecommandline}
sage: sorted(XPhi_2x1) == sorted(XPZRZ_2x1)
True
sage: sorted(XPhi_1x2) == sorted(XPZRZ_1x2) 
True
sage: sorted(XPhi_2x2) == sorted(XPZRZ_2x2)
True
\end{sagecommandline}
\end{sagecommandlinetcb}
We conclude from Lemma~\ref{lem:criterion-for-minimality},
that $\Xcal_{\Pcal_\Zcal,R_\Zcal}$ is minimal and $\Xcal_{\Pcal_\Zcal,R_\Zcal}=\Xcal_\Phi$.
\end{exosolution}

\begin{exosolution}{\ref{exo:induction-first-horizontal}}
    Left to the reader.
\end{exosolution}

\begin{exosolution}{\ref{exo:Markof-Partition}}
    Since $\Omega_\Zcal =\Xcal_\Phi =\Xcal_{\Pcal_\Zcal,R_\Zcal}$
    and $\Omega_\Zcal$ is a subshift of finite type,
    it follows that $\Pcal_\Zcal$ is a Markov partition
    for the dynamical system $(\torus^2,\Z^2,R_\Zcal)$.
\end{exosolution}

\begin{exosolution}{\ref{exo:equality-of-substitutions}}
    Left to the reader.
\end{exosolution}

\begin{exosolution}{\ref{exo:further-simplification}}
    Left to the reader.
\end{exosolution}

\bibliographystyle{myalpha} 
\bibliography{bib_Labbe}

\end{document}